\documentclass[reqno]{amsart}

\usepackage{amsmath,amssymb,amsthm,enumerate} 
\usepackage{amsfonts} 
\usepackage{color}
\usepackage{dsfont} 

\theoremstyle{plain}
\newtheorem{thm}{Theorem}[section]
\newtheorem{asm}[thm]{Assumption}

\newtheorem{dfn}[thm]{Definition}

\newtheorem{lem}[thm]{Lemma}
\newtheorem{prp}[thm]{Proposition}
\newtheorem{rmk}[thm]{Remark}

\DeclareMathOperator{\dv}{div}

\newcommand{\intd}{\,d}

\newcommand{\lc}{{\rm loc}}
\newcommand{\pd}{\partial}
\newcommand{\wht}[1]{\widehat{#1}}
\newcommand{\wtd}[1]{\widetilde{#1}}

\newcommand{\Ba}{\mathbf{a}}
\newcommand{\Bb}{\mathbf{b}}

\newcommand{\Be}{\mathbf{e}}
\newcommand{\Bf}{\mathbf{f}}
\newcommand{\Bg}{\mathbf{g}}
\newcommand{\Bh}{\mathbf{h}}

\newcommand{\Bn}{\mathbf{n}}

\newcommand{\Bu}{\mathbf{u}}
\newcommand{\Bv}{\mathbf{v}}
\newcommand{\Bw}{\mathbf{w}}

\newcommand{\BC}{\mathbf{C}}
\newcommand{\BD}{\mathbf{D}}
\newcommand{\BE}{\mathbf{E}}

\newcommand{\BI}{\mathbf{I}}
\newcommand{\BJ}{\mathbf{J}}
\newcommand{\BK}{\mathbf{K}}

\newcommand{\BN}{\mathbf{N}}

\newcommand{\BQ}{\mathbf{Q}}
\newcommand{\BR}{\mathbf{R}}

\newcommand{\CA}{\mathcal{A}}
\newcommand{\CB}{\mathcal{B}}

\newcommand{\CD}{\mathcal{D}}
\newcommand{\CE}{\mathcal{E}}
\newcommand{\CF}{\mathcal{F}}

\newcommand{\CK}{\mathcal{K}}
\newcommand{\CL}{\mathcal{L}}

\newcommand{\CU}{\mathcal{U}}

\newcommand{\Fg}{\mathfrak{g}}

\newcommand{\Fp}{\mathfrak{p}}
\newcommand{\Fq}{\mathfrak{q}}

\newcommand{\SSD}{\mathsf{D}}

\newcommand{\SSF}{\mathsf{F}}
\newcommand{\SSG}{\mathsf{G}}
\newcommand{\SSH}{\mathsf{H}}

\newcommand{\SSK}{\mathsf{K}}
\newcommand{\SSL}{\mathsf{L}}

\newcommand{\SST}{\mathsf{T}}

\numberwithin{equation}{section} 
\allowdisplaybreaks[4]

\begin{document}
\title[Viscous free surface flows of infinite depth]
{Global solvability for viscous free surface flows of infinite depth in three and higher dimensions}


\author[H. Saito]{Hirokazu Saito}
\address[H. Saito]{Graduate School of Informatics and Engineering,
The University of Electro-Communications,
5-1 Chofugaoka 1-chome, Chofu, Tokyo 182-8585, Japan}
\email{hsaito@uec.ac.jp}

\author[Y. Shibata]{Yoshihiro Shibata}
\address[Y. Shibata]{Professor Emeritus of Waseda University;
adjunct faculty member
in the Department of Mechanical Engineering and Materials Science, University of Pittsburgh}
\email{yshibata325@gmail.com}


\thanks{This work was supported by JSPS KAKENHI Grant Numbers JP21K13817, JP22H01134.}



\begin{abstract}
This paper is concerned with the global solvability for the Navier-Stokes equations
describing viscous free surface flows of infinite depth in three and higher dimensions.
We first prove time weighted estimates of solutions to a linearized system of the Navier-Stokes equations
by time decay estimates of a $C_0$-analytic semigroup
and maximal regularity estimates in an $L_p$-in-time and $L_q$-in-space setting
with suitable $p$, $q$. 
The time weighted estimates then enable us to show
the global solvability of the Navier-Stokes equations for small initial data by the contraction mapping principle.
\end{abstract}

\maketitle


\section{Introduction}\label{sec:intro}

Let $\Omega(\tau)$ and $\Gamma(\tau)$, $\tau>0$, be given by 
\begin{align}
\Omega(\tau)
&=\{(y',y_N) : y'=(y_1,\dots,y_{N-1})\in\BR^{N-1}, y_N<\eta(y',\tau)\}, \notag \\ 
\Gamma(\tau)
&=\{(y',y_N) : y'=(y_1,\dots,y_{N-1})\in\BR^{N-1}, y_N=\eta(y',\tau)\},  \label{om-gam-2}
\end{align}
where $N\geq 3$ and $\eta=\eta(y',\tau)$ needs to be determined as part of the problem.
For $T>0$ or $T=\infty$, we set
\begin{equation}\label{om-gam-1}
\Omega_T=\bigcup_{0<\tau<T}\Omega(\tau)\times\{\tau\}, \quad 
\Gamma_T=\bigcup_{0<\tau<T}\Gamma(\tau)\times\{\tau\}.
\end{equation}

This paper is concerned with the motion of an incompressible viscous fluid occupying $\Omega(\tau)$.
We denote the fluid velocity by $\Bv=\Bv(y,\tau)=(v_1(y,\tau),\dots,v_N(y,\tau))^\SST$,
where the superscript $\SST$ denotes the transpose,
and the pressure by $\Fq=\Fq(y,\tau)$ at position $y\in\Omega(\tau)$ with time $\tau>0$.
The motion is governed by the Navier-Stokes equations as follows: 
\begin{equation}\label{nonl:eq1}
\left\{\begin{aligned}
\pd_\tau\eta -v_N &=-\Bv'\cdot\nabla'\eta 
&& \text{on $\Gamma_\infty$,} \\
\pd_\tau\Bv+ (\Bv\cdot\nabla)\Bv&=\mu \Delta \Bv-\nabla \Fq 
&& \text{in $\Omega_\infty$,} \\
\dv\Bv&=0 
&& \text{in $\Omega_\infty$,} \\
(\mu\BD(\Bv)-\Fq\BI)\Bn_{\Gamma(\tau)}+c_g\eta\Bn_{\Gamma(\tau)}
&=c_\sigma\kappa_{\Gamma(\tau)}\Bn_{\Gamma(\tau)} 
&& \text{on $\Gamma_\infty$,}
\end{aligned}\right.
\end{equation}
subject to the initial conditions
\begin{equation}\label{initial-cond}
\eta|_{\tau=0}=\eta_0 \quad \text{on $\BR^{N-1}$}, \quad \Bv|_{\tau=0}=\Bv_0 \quad \text{in $\Omega_0$.} 
\end{equation}

Here the positive constants $\mu$, $c_g$, and $c_\sigma$ describe the viscosity coefficient,
the acceleration of gravity, and the surface tension coefficient, respectively.
The $\eta_0=\eta_0(y')$ and $\Bv_0=\Bv_0(y)=(v_{01}(y),\dots,v_{0N}(y))^\SST$ are given initial date,
and
\begin{equation}\label{omega-zero}
\Omega_0=\{(y',y_N) : y'=(y_1,\dots,y_{N-1})\in\BR^{N-1}, y_N<\eta_0(y')\}.
\end{equation}
The boundary of $\Omega_0$ is denoted by $\Gamma_0$, i.e.,
\begin{equation}\label{gam-zero}
\Gamma_0 =\{(y',y_N) : y'=(y_1,\dots,y_{N-1})\in\BR^{N-1}, y_N=\eta_0(y')\}.
\end{equation}

Let $\pd_\tau=\pd/\pd \tau$, $D_j=\pd/\pd y_j$, and $\Delta=\sum_{j=1}^ND_j^2$. We set
\begin{align*}
&\Bv'\cdot\nabla'\eta=\sum_{j=1}^{N-1}v_j D_j \eta, \quad
(\Bv\cdot\nabla)\Bv
=\bigg(\sum_{j=1}^N v_j D_ jv_1,\dots, \sum_{j=1}^N v_j D_j v_N\bigg)^\SST, \quad
\\
&\Delta\Bv=(\Delta v_1,\dots,\Delta v_N)^\SST, \quad
\nabla\Fq =(D_1 \Fq,\dots,D_N\Fq)^\SST, \quad
\dv\Bv=\sum_{j=1}^N D_jv_j.
\end{align*}
In addition, $\BI$ is the $N\times N$ identity matrix and
$\BD(\Bv)$ is the doubled deformation rate tensor, i.e.,
$\BD(\Bv)=\nabla\Bv+(\nabla\Bv)^\SST$ with
\begin{equation*}
\nabla\Bv=(D_j v_i)_{1\leq i,j\leq N}=
\begin{pmatrix}
D_1 v_1 & \dots & D_N v_1 \\
\vdots & \ddots & \vdots \\
D_1 v_N & \dots & D_N v_N
\end{pmatrix}.
\end{equation*}
The $\Bn_{\Gamma(\tau)}$ and $\kappa_{\Gamma(\tau)}$
stand for the unit outward normal vector to $\Gamma(\tau)$
and the mean curvature of $\Gamma(\tau)$, respectively,
and they are given by
\begin{align}
\Bn_{\Gamma(\tau)}
&=\frac{1}{\sqrt{1+|\nabla'\eta(y',\tau)|^2}}
\begin{pmatrix}
-\nabla'\eta(y',\tau) \\
1
\end{pmatrix}, \label{normal-vec} \\
\kappa_{\Gamma(\tau)}
&=\sum_{j=1}^{N-1} D_j \bigg(\frac{D_j\eta(y',\tau)}{\sqrt{1+|\nabla'\eta(y',\tau)|^2}}\bigg), \label{mean-cuv}
\end{align}
with $\nabla' \eta=(D_1 \eta, \dots, D_{N-1} \eta)^\SST$ and $|\nabla'\eta|^2=\sum_{j=1}^{N-1}(D_j\eta)^2$.

Let us introduce short history of mathematical studies for free surface flows of incompressible viscous fluids.

Beale \cite{Beale81} considered finite-depth flows bounded above by a free surface and below by a solid surface
in the presence of a uniform gravitational field in three dimensions,
where surface tension was not taken into account.
He proved in \cite{Beale81} a local existence theorem for large initial data in an $L_2$-based Sobolev space. 
Along with this direction based on Hilbert space settings,
Sylvester \cite{Sylvester90} proved a global existence theorem for small initial data,
while Hataya and Kawashima \cite{HK09} considered large time decay of solutions.
The free surface becomes more regular if surface tension is taken into account.
Using this fact, Beale \cite{Beale83} proved a global existence theorem for small initial data.
Furthermore, large time decay of solutions of \cite{Beale83} was shown by Beale and Nishida \cite{BN85}.
For the problem including surface tension,
we also refer to Allain \cite{Allain87}, which treats the two-dimensional problem,
Tani \cite{Tani96}, Hataya \cite{Hataya11}, and Bae \cite{Bae11}.
Tani and Tanaka \cite{TT95} proved global existence theorems with or without surface tension,
which relaxed the regularity assumption on the initial data.

Guo and Tice \cite{GT13b, GT13} considered the same problem as in \cite{Beale81},
i.e.,  finite-depth flows without surface tension in three dimensions,
and proved that the problem is locally and globally well-posed
and that solutions decay to equilibrium at an algebraic rate
by introducing a new  approach based on a high-regularity energy method.
Wu \cite{Wu14} extended their result in local well-posedness from the small data case to the general data case,
while Wang \cite{Wang20} proved global well-posedness for the two-dimensional problem
based on the approach of Guo and Tice.

Abels \cite{Abels05} introduced an $L_q$-based Sobolev space 
to prove local well-posedness of the same problem as in \cite{Beale81}.
Furthermore,  Saito \cite{Saito18} proved global well-posedness
in an $L_p$-in-time and $L_q$-in-space setting  and exponential decay of solutions
in the case where surface tension is not taken into account
and also gravity does not work.

All of the above results are interested in non-compact free surfaces,
while compact cases are also important in applications.
Solonnikov did a lot of works for problems on bounded domains
in $L_2$-based Sobolev spaces, $L_q$-based Sobolev spaces, or H$\ddot{\rm o}$lder spaces,
see \cite{Solonnikov03} and references therein.
On the other hand, Shibata \cite{Shibata20} proved the global well-posedness
in an $L_p$-in-time and $L_q$-in-space setting
for the free surface flow in an exterior domain
under the condition that neither surface tension nor gravity takes into account.

If we consider the case where surface tension is not taken into account
and also gravity does not work,
then the free surface flow becomes a parabolic system in Lagrangian coordinates.
In this situation, infinite-depth flows were considered in 
Oishi and Shibata \cite{OS22};
Danchin, Hieber, Mucha, and Tolksdorf \cite{DHMTpreprint};
Ogawa and Shimizu \cite{OSpreprint};
Shibata and Watanabe \cite{SWpreprint}.
The first paper proved the global well-posedness in a similar way to \cite{Shibata20},
while the other three papers developed maximal $L_1$-regularity theory
for the Stokes equations in the half-space.
On the other hand,
if we include surface tension or gravity, 
then some hyperbolic effect appears from the transport equation of the height function,
i.e., from the first equation of \eqref{nonl:eq1}, see 
Pr$\ddot{\rm u}$ss and Simonett \cite{PS09},
Saito and Shibata \cite{SaS16} for more details.
Our system \eqref{nonl:eq1} is thus completely different from
the parabolic system as above.

The aim of this paper is to show the global solvability of \eqref{nonl:eq1}--\eqref{initial-cond}
and large time decay of the solution.
To this end, we first transform \eqref{nonl:eq1}--\eqref{initial-cond} to 
a system in a fixed domain, see \eqref{nonl:fix1}--\eqref{initial-cond-flat} below.
We establish a global existence result for the transformed system in Theorem \ref{thm:transformed} below
in an $L_p$-in-time and $L_q$-in-space setting with a suitable choice of $p$ and $q$.
Our main results for \eqref{nonl:eq1}--\eqref{initial-cond} as stated above
are then proved by the inverse transformation of the solution to the system in the fixed domain,
see Subsection \ref{subsec:2-4} below for more details.

At this point, we introduce differences between this paper and our previous work \cite{SaSpreprint}.
Although \cite{SaSpreprint} treats the three-dimensional case only, 
this paper treats the $N$-dimensional case for $N\geq 3$.
Furthermore, we establish in this paper
a new estimate of solutions to the linearized system \eqref{lin-eq:1} below for \eqref{nonl:eq1}--\eqref{initial-cond}.
To introduce the estimate, we define for $\langle t\rangle =\sqrt{1+t^2}$
\begin{equation*}
\|\langle t\rangle^a f\|_{L_p(\BR_+,X)}=
\bigg(\int_0^\infty \big(\langle t\rangle^a\|f(t)\|_X\big)^p \intd t\bigg)^{1/p},
\end{equation*}
where $a\geq 0$, $1\leq p<\infty$, and $X$ is a Banach space.
In the case $N\geq 4$, 
\begin{align*}
&\|\langle t\rangle^{1/2}(\pd_t\eta,\pd_t\Bu)\|_{L_p(\BR_+,W_q^{2-1/q}(\BR^{N-1})\times L_q(\BR_-^N)^N)} \\
&+\|\langle t\rangle^{1/2}(\eta,\Bu)\|_{L_p(\BR_+,W_q^{3-1/q}(\BR^{N-1})\times H_q^2(\BR_-^N)^N)} \\
&\leq C\Big(\text{(norm of the initial data)} \notag  \\
&+ \text{(norm of the right members with time weight $\langle t\rangle$)}\Big)
\end{align*}
holds for suitable $p, q\in(1,\infty)$, see Proposition \ref{prp:linear4d} below,
where $\BR_-^N$ is the lower half-space and
the right members stand for $d$, $\Bf$, $\Fg$, $g$, and $\Bh$ in the linearized system.
Since the right members are replaced by nonlinear terms, e.g., $\Bu\cdot\nabla\Bu$,
in application to the nonlinear problem \eqref{nonl:fix1}--\eqref{initial-cond-flat} below,
one writes
\begin{equation}\label{nonl-est-intro}
\langle t\rangle \Bu\cdot\nabla\Bu=\langle t\rangle^{1/2}\Bu\cdot(\langle t\rangle^{1/2}\nabla\Bu)
\end{equation}
and estimates $\langle t\rangle \Bu\cdot\nabla\Bu$ as follows:
\begin{align*}
&\|\langle t\rangle \Bu\cdot\nabla\Bu\|_{L_p(\BR_+,L_r(\BR_-^N)^N)} \\
&\leq C\|\langle t\rangle^{1/2}\Bu\|_{L_\infty(\BR_+,L_{r_1}(\BR_-^N)^N)}
\|\langle t\rangle^{1/2}\nabla \Bu\|_{L_p(\BR_+,L_{r_2}(\BR_-^N)^{N\times N})}
\end{align*}
for $r_1,r_2\in[1,\infty]$ satisfying $1/r=1/r_1+1/r_2$.
Combining this with an embedding 
shows that the nonlinear term can be estimated from above by 
norm of solutions with time weight $\langle t\rangle^{1/2}$.
This enables us to construct global-in-time solutions to the nonlinear problem by the contraction mapping principle
for small initial data.
In the case $N=3$, we use another time weight for solutions to the linearized system,
see Proposition \ref{prp:linear3d} below for more details,\footnote{
In the proposition, $\langle t\rangle^{1/3}\Bu$ and $\langle t\rangle^{2/3}\nabla\Bu$
are estimated from above by norm of the right members with time weight $\langle t\rangle$.
}
and then we replace \eqref{nonl-est-intro} by
\begin{equation*}
\langle t\rangle \Bu\cdot\nabla\Bu=\langle t\rangle^{1/3}\Bu \cdot(\langle t\rangle^{2/3}\nabla \Bu)
\end{equation*}
and estimate it as above in application to the nonlinear problem.

We think that the above approach is much simpler than \cite{SaSpreprint}
and useful for treating other quasilinear systems of parabolic-hyperbolic equations.
In addition, we would liked to emphasize that 
Proposition \ref{prp:duha} below plays a crucial role
in proving time weighted estimates of solutions to the linearized system in this paper.
That proposition and our linear theory given by Section \ref{sec:lin-theory}
are based on the approach of Shibata \cite{Shibata22}.

This paper is organized as follows.
The next section first transforms \eqref{nonl:eq1}--\eqref{initial-cond} to a half-space problem.
Secondly, we introduce the notation.
Global solvability results, i.e., our main results of this paper,
are then stated for the transformed problem and the original problem \eqref{nonl:eq1}--\eqref{initial-cond}.
Section \ref{sec:prelim} introduces fundamental tools of our analysis.
Section \ref{sec:duhamel} proves an estimate for Duhamel's integral.
Section \ref{sec:sg} considers a $C_0$-analytic semigroup associated with 
the linearized system of the transformed problem
and shows several properties of the semigroup.
Section \ref{sec:lin-theory} introduces our linear theory for the transformed problem.
Section \ref{sec:nonl-terms} estimates nonlinear terms.
Section \ref{sec:nonlinear1} proves, for the transformed problem,
the global solvability result stated in Section \ref{sec:main}
by means of results obtained in Sections \ref{sec:lin-theory} and \ref{sec:nonl-terms}
together with the contraction mapping principle.
Section \ref{sec:nonlinear2} proves, for the original problem \eqref{nonl:eq1}--\eqref{initial-cond},
the global solvability result stated in Section \ref{sec:main} by the solution of the transformed problem.

\section{Main results}\label{sec:main}

Let $\BR_-^N$ be the lower half space and $\BR_0^N$ its boundary, i.e.,
\begin{align*}
\BR_-^N
&=\{(x',x_N) : x'=(x_1,\dots,x_{N-1})\in\BR^{N-1},x_N<0\}, \\
\BR_0^N
&=\{(x',x_N) : x'=(x_1,\dots,x_{N-1})\in\BR^{N-1}, x_N=0\}.
\end{align*}
In addition, we set 
\begin{equation}\label{dfn:Q-Q0}
\BQ_-=\BR_-^N\times(0,\infty), \quad \BQ_0=\BR_0^N\times(0,\infty).
\end{equation}

\subsection{The transformed problem}
Let us define the partial Fourier transform with respect to $x'=(x_1,\dots,x_{N-1})$
and its inverse transform by
\begin{align}\label{fourier-t}
\wht  f(\xi')&=\int_{\BR^{N-1}} e^{-ix'\cdot\xi'} f(x') \intd x', \notag \\
\CF_{\xi'}^{-1}[g(\xi')](x')&=\frac{1}{(2\pi)^{N-1}}\int_{\BR^{N-1}}e^{ix'\cdot\xi'} g(\xi') \intd \xi',
\end{align}
where $\xi'=(\xi_1,\dots,\xi_{N-1})$. 
Define the extension $\CE_N\eta$ of $\eta$ by
\begin{equation}\label{ext:eta}
\CE_N\eta
=
\left\{\begin{aligned}
&\CF_{\xi'}^{-1}[e^{|\xi'| x_N}\wht \eta(\xi',t)](x') && \text{when $N=3$,} \\
&\CF_{\xi'}^{-1}[e^{\sqrt{1+|\xi'|^2} x_N}\wht \eta(\xi',t)](x') && \text{when $N\geq 4$,}
\end{aligned}\right.
\end{equation}
where $x_N<0$.
Notice that $u=\CE_N\eta$ is the solution to
\begin{equation*}
\Delta u=0 \quad \text{in $\BQ_-$,} \quad 
u=\eta \quad \text{on $\BQ_0$}
\end{equation*}
for $N=3$, while it is the solution to 
\begin{equation*}
(1-\Delta)u=0 \quad \text{in $\BQ_-$,} \quad u=\eta \quad \text{on $\BQ_0$}
\end{equation*}
for $N\geq 4$.
By using the extension $\CE_N\eta$,
we define the mapping
\begin{align}\label{dfn:theta}
&\Theta:  
\BQ_- \ni (x',x_N,t)\mapsto \Theta(x',x_N,t)  \in \Omega_\infty, \notag \\
&\Theta(x',x_N,t)  :=(x',x_N+(\CE_N\eta)(x',x_N,t),t),
\end{align}
which is a diffeomorphism from $\BQ_-$ onto $\Omega_\infty$
under the assumption that $\eta$ is smooth enough and 
\begin{equation}\label{cond-diffeo}
\sup_{(x,t)\in \BR_-^N \times (0,\infty)}|(\pd_N \CE_N\eta)(x,t)|\leq \frac{1}{2},
\end{equation}
where $\pd_N=\pd/\pd x_N$.

Let $x=(x',x_N)\in\BR_-^N$ and $t>0$. We set
\begin{equation}\label{dfn:u-p}
\Bu(x,t)=\Bv(\Theta(x,t)), \quad \Fp(x,t)=\Fq(\Theta(x,t)).
\end{equation}
Then $(\eta,\Bu,\Fp)$ satisfies 
\begin{equation}\label{nonl:fix1}
\left\{\begin{aligned}
\pd_t\eta-u_N &=\SSD(\eta,\Bu) && \text{on $\BQ_0$, }\\
\pd_t\Bu-\mu\Delta \Bu +\nabla \Fp&=\SSF(\eta,\Bu) && \text{in $\BQ_-$,} \\
\dv\Bu=\SSG(\eta,\Bu)&=\dv\wtd\SSG(\eta,\Bu) && \text{in $\BQ_-$,} \\
(\mu\BD(\Bu)-\Fp\BI)\Be_N+(c_g-c_\sigma\Delta')\eta\Be_N
&=\SSH(\eta,\Bu) 
&& \text{on $\BQ_0$,} 
\end{aligned}\right.
\end{equation}
subject to the initial conditions
\begin{equation}\label{initial-cond-flat}
\eta|_{t=0}=\eta_0 \quad \text{on $\BR^{N-1}$,}  \quad
\Bu|_{t=0}=\Bu_0  \quad \text{in $\BR_-^N$.}
\end{equation}

Here and subsequently, $u_j$ denotes the $j$th component of $\Bu=\Bu(x,t)$ for $j=1,\dots,N$,
while $\Be_N=(0,\dots,0,1)^\SST$ and $\Delta'=\sum_{j=1}^{N-1}\pd^2/\pd x_j^2$.
The $\SSD$, $\SSF$, $\SSG$, $\wtd\SSG$, and $\SSH$ consist of nonlinear terms and
they are given as follows.\footnote{
See the appendix below for the derivation of \eqref{nonl:fix1}.}

{\it Nonlinear terms}.
Let $j,k=1,\dots,N$ and set for $N$-vectors $\Ba=(a_1,\dots,a_N)^\SST$ and $\Bb=(b_1,\dots,b_N)^\SST$
\begin{equation*}
\Ba\cdot\Bb=\sum_{j=1}^N a_j b_j, \quad 
\Ba\otimes\Bb=(a_ib_j)_{1\leq i,j\leq N}
=\begin{pmatrix}
a_1 b_1 & \dots & a_1 b_N \\
\vdots & \ddots & \vdots \\
a_N b_1 & \dots & a_Nb_N
\end{pmatrix}.
\end{equation*}
Let $\pd_j=\pd/\pd x_j$ and  
$\CD_{jk}(\eta)$ be the second order differential operator defined as
\begin{align}\label{deriv-second}
\CD_{jk}(\eta)
&=(1+\pd_N\CE_N\eta)^2(\pd_k\CE_N\eta)\pd_j\pd_N \notag \\
& +(1+\pd_N\CE_N\eta)^2(\pd_j\CE_N\eta)\pd_N\pd_k
-(1+\pd_N\CE_N\eta)(\pd_j\CE_N\eta)(\pd_k\CE_N\eta)\pd_N^2 \notag \\
& +\Big\{(1+\pd_N\CE_N\eta)^2(\pd_j\pd_k\CE_N\eta) 
-(1+\pd_N\CE_N\eta)(\pd_k\CE_N\eta)(\pd_j\pd_N\CE_N\eta) \notag \\
& -(1+\pd_N\CE_N\eta)(\pd_j\CE_N\eta)(\pd_N\pd_k\CE_N\eta)
+(\pd_j\CE_N\eta)(\pd_k\CE_N\eta)(\pd_N^2\CE_N\eta)
\Big\}\pd_N.
\end{align}
Define the $N\times N$ matrices $\BJ(\eta)$ and $\BK(\eta)$ by
\begin{equation}\label{matrix:JK}
\BJ(\eta)=
\begin{pmatrix}
0 &  \dots & 0 & \pd_1\CE_N\eta \\
\vdots &  \ddots &\vdots & \vdots \\
0 & \dots & 0 & \pd_{N-1}\CE_N\eta \\
0 & \dots & 0 & \pd_N\CE_N\eta
\end{pmatrix}, \quad
\BK(\eta)= 
\begin{pmatrix}
0 &  \dots & 0 & \pd_1\CE_N\eta \\
\vdots &  \ddots &\vdots & \vdots \\
0 & \dots & 0 & \pd_{N-1}\CE_N\eta \\
0 & \dots & 0 & 0
\end{pmatrix}.
\end{equation}
In addition,
\begin{align}\label{dfn:E}
\wtd \BE(\eta,\Bu) &=\nabla\CE_N\eta\otimes\pd_N\Bu+(\nabla\CE_N\eta\otimes\pd_N\Bu)^\SST, \notag \\
\BE(\eta,\Bu)&= \frac{1}{1+\pd_N\CE_N\eta}\wtd \BE(\eta,\Bu),
\end{align}
and also
\begin{equation}\label{normal-vec2}
\Bn(\eta)=
\begin{pmatrix}
-\nabla'\CE_N\eta \\ 1
\end{pmatrix}, \quad 
\wht \Bn(\eta) = 
\begin{pmatrix}
-\nabla'\CE_N\eta \\ 0
\end{pmatrix}
\end{equation}
with $\nabla'\CE_N\eta=(\pd_1\CE_N\eta,\dots,\pd_{N-1}\CE_N\eta)^\SST$.
Then
\begin{align}
\SSD(\eta,\Bu) 
&=-\Bu'\cdot\nabla'\CE_N\eta  =-\sum_{j=1}^{N-1} u_j\pd_j\CE_N\eta, \label{def:D} \\
\SSG(\eta,\Bu)
&=-(\pd_N\CE_N\eta)\dv\Bu+\nabla\CE_N\eta\cdot \pd_N\Bu, \label{def:G} \\
\wtd\SSG(\eta,\Bu)
&=-(\pd_N\CE_N\eta)\Bu+\BJ(\eta)^\SST\Bu,  \label{def:G-tilde}
\end{align}
and 
\begin{equation}\label{dfn:F}
\SSF(\eta,\Bu)
=\frac{1}{(1+\pd_N\CE_N\eta)^3}(\BI+\BJ(\eta))\wtd \SSF(\eta,\Bu) -\BJ(\eta)(\pd_t\Bu-\mu\Delta\Bu),
\end{equation}
with
\begin{align}\label{dfn:F-tilde}
\wtd \SSF(\eta,\Bu)&=
(1+\pd_N\CE_N\eta)^2(\pd_t\CE_N\eta)\pd_N\Bu-
(1+\pd_N\CE_N\eta)^3(\Bu\cdot\nabla)\Bu \notag \\
&+(1+\pd_N\CE_N\eta)^2\bigg(\sum_{j=1}^N\Bu_j\pd_j\CE_N\eta\bigg)\pd_N\Bu+\mu\sum_{j=1}^N\CD_{jj}(\eta)\Bu.
\end{align}
Furthermore,
\begin{equation} \label{dfn:H}
\SSH(\eta,\Bu)=\wtd \SSH(\eta,\Bu)-\sigma\SSH_\kappa(\eta)\Be_N,
\end{equation}
where
\begin{align}\label{dfn:H-1}
\wtd \SSH(\eta,\Bu) =-\mu\BD(\Bu)\wht\Bn(\eta)+\mu\BE(\eta,\Bu)\Bn(\eta)
-\mu\BK(\eta)(\BD(\Bu)-\BE(\eta,\Bu))\Bn(\eta)
\end{align}
for $\BD(\Bu)=(\pd_i u_j+\pd_j u_i)_{1\leq i,j\leq N}$ and 
\begin{align}\label{dfn:H-2}
\SSH_\kappa(\eta)&=\frac{|\nabla'\CE_N\eta|^2\Delta'\CE_N\eta}{(1+\sqrt{1+|\nabla'\CE_N\eta|^2})\sqrt{1+|\nabla'\CE_N\eta|^2}} \notag \\
&+\sum_{j,k=1}^{N-1}\frac{(\pd_j\CE_N\eta)(\pd_k\CE_N\eta)(\pd_j\pd_k\CE_N\eta)}{(1+|\nabla'\CE_N\eta|^2)^{3/2}}.
\end{align}

\subsection{Notation}

Let $\BN$ be the set of all positive integers and $\BN_0=\BN\cup\{0\}$.
For any multi-index $\alpha=(\alpha_1,\dots,\alpha_N)\in\BN_0^N$
and scalar function $u=u(x_1,\dots,x_N)$
\begin{equation*}
\pd^\alpha u =\pd_x^\alpha u 
=\frac{\pd^{|\alpha|}}{\pd x_1^{\alpha_1}\dots \pd x_N^{\alpha_N}} u(x_1,\dots,x_N), \quad 
|\alpha|=\alpha_1+\dots+\alpha_N.
\end{equation*}

Let $X$ be a Banach space.
Then $X^M$, $M\geq 2$, denotes the $M$-product space of $X$,
while norm of $X^M$ is usually denoted by $\|\cdot \|_X$ instead of $\|\cdot\|_{X^M}$ for the sake of simplicity. 
Let $Y$ be another Banach space. 
Then $\CL(X,Y)$ is the Banach space of all bounded linear operators from $X$ to $Y$,
and $\CL(X)$ is the abbreviation of $\CL(X,X)$. In addition,
norm of $X\cap Y$ is given by $\|\cdot\|_{X\cap Y}=\|\cdot\|_X+\|\cdot\|_Y$.

Let $p\in[1,\infty]$, $q\in(1,\infty)$, $m\in\BN$, and $s\in\BR_+:=(0,\infty)$.
Let $G$ be a domain in $\BR^N$. 
Then $L_p(G)$, $H_p^m(G)$, and $B_{q,p}^s(G)$
stand for the usual Lebesgue spaces, Sobolev spaces, and Besov spaces on $G$, respectively.
Their respective norms are denoted by $\|\cdot\|_{L_p(G)}$, $\|\cdot\|_{H_p^m(G)}$, and $\|\cdot\|_{B_{q,p}^s(G)}$. 
We set $H_p^0(G)=L_p(G)$ and $W_q^s(G)=B_{q,q}^s(G)$ for $s\in\BR_+\setminus\BN$.
In addition, $C_0^\infty(G)$ stands for the set of all $C^\infty$ functions
whose supports are compact and contained in $G$,
while $C^n(G)$, $n\in\BN_0$, the set of all $C^n$ functions in $G$.
Let
\begin{align*}
C_B^n(G) &=\{f\in C^n(G) : \max_{0\leq|\alpha|\leq n}\sup_{x\in G}|\pd_x^\alpha f(x)|<\infty\}. 
\end{align*}
Analogously, $C^n(\overline{G})$ and $C_B^n(\overline{G})$ are defined for the closure $\overline{G}$ of $G$.
We set $C(G)=C^0(G)$, $C_B(G)=C_B^0(G)$, $C(\overline{G})=C^0(\overline{G})$, and $C_B(\overline{G})=C_B^0(\overline{G})$.
Define
\begin{equation}\label{dfn:hat-0}
\wht H_{q,0}^1(\BR_-^N)=\{f\in\wht H_q^1(\BR_-^N) : f=0 \text{ on $\BR_0^N$}\}.
\end{equation}
For $\Ba=\Ba(x)=(a_1(x),\dots,a_N(x))^\SST$ and $\Bb=\Bb(x)=(b_1(x),\dots,b_N(x))^\SST$,
we set
\begin{equation*}
(\Ba,\Bb)_{\BR_-^N}=\int_{\BR_-^N}\Ba(x)\cdot\Bb(x)\intd x=\sum_{j=1}^N\int_{\BR_-^N}a_j(a)b_j(x)\intd x.
\end{equation*}
Let $q'=q/(q-1)$ and define the space of solenoidal vector fields by
\begin{equation}\label{dfn:solenoidal}
J_q(\BR_-^N)=\{\Bf\in L_q(\BR_-^N)^N : (\Bf,\nabla\varphi)_{\BR_-^N}=0 \text{ for any $\varphi\in\wht H_{q',0}^1(\BR_-^N)$}\}.
\end{equation}

Let us introduce the function space $I_{q,p}$ of initial data as follows:
\begin{equation}\label{dfn:ini-sp}
I_{q,p}
=B_{q,p}^{3-1/p-1/q}(\BR^{N-1})\times B_{q,p}^{2-2/p}(\BR_-^N)^N
\end{equation}
endowed with the norm
\begin{equation*}
\|(\eta_0,\Bu_0)\|_{I_{q,p}}=\|\eta_0\|_{B_{q,p}^{3-1/p-1/q}(\BR^{N-1})}+\|\Bu_0\|_{B_{q,p}^{2-2/p}(\BR_-^N)}.
\end{equation*}

Let $J$ be an interval of $\BR$.
Then $L_p(J,X)$ and $H_p^m(J,X)$, $m\in\BN$, denote the $X$-valued Lebesgue spaces on $J$
and the $X$-valued Sobolev spaces on $J$, respectively.
Their respective norms are denoted by $\|\cdot\|_{L_p(J,X)}$ and $\|\cdot\|_{H_p^m(J,X)}$.
We denote the set of all $C^n$ functions $f:J\to X$ by $C^n(J,X)$ for $n\in\BN_0$ and set $C(J,X)=C^0(J,X)$.
By the complex interpolation functor $[\cdot,\cdot]_\theta$ with $\theta\in(0,1)$, we set
\begin{equation*}
H_p^\theta(J,X)=[L_p(J,X),H_p^1(J,X)]_\theta.
\end{equation*}

\subsection{Global solvability for the transformed problem}
For an $N$-vector $\Ba$, we set
\begin{equation}\label{dfn:tng}
\Ba_\tau = \Ba-(\Ba\cdot\Be_N)\Be_N.
\end{equation}
Our main result for the transformed problem \eqref{nonl:fix1}--\eqref{initial-cond-flat} reads as follows.

\begin{thm}\label{thm:transformed}
Let $N\geq 3$.
Then there exist a large positive number $p_0$ and small positive numbers
$q_0$, $r_0$, and $\varepsilon_0$
such that the following assertion holds for any $p$, $q_1$, and $q_2$ satisfying
\begin{equation}\label{pq-cond}
p_0 \leq p <\infty, \quad 2 < q_1\leq 2+q_0, \quad 
N<q_2\left\{\begin{aligned}
&\leq 6 && \text{when $N=3$,} \\
&<\infty && \text{when $N\geq 4$,}
\end{aligned}\right.
\end{equation} 
and
\begin{equation}\label{pq-cond-2}
\frac{2}{p}+\frac{1}{(q_1/2)} \neq 1.
\end{equation}

For any $(\eta_0,\Bu_0)\in I_{q_1/2,p}\cap I_{q_2,p}$ with the smallness condition
\begin{equation}\label{small-1}
\|(\eta_0,\Bu_0)\|_{I_{q_1/2,p}} + \|(\eta_0,\Bu_0)\|_{I_{q_2,p}} \leq \varepsilon_0
\end{equation}
and the compatibility conditions
\begin{alignat}{2}
\dv\Bu_0&=\SSG(\eta_0,\Bu_0) && \quad \text{in $\BR_-^N$,}  \label{comp-1} \\
(\mu\BD(\Bu_0)\Be_N)_\tau &= (\SSH(\eta_0,\Bu_0))_\tau && \quad  \text{on $\BR_0^N$,}  \label{comp-2}
\end{alignat}
the system \eqref{nonl:fix1}--\eqref{initial-cond-flat} admits a unique solution $(\eta,\Bu)$
in $K_{p,q_1,q_2}^N(r_0;\eta_0,\Bu_0)$ with some pressure $\Fp$.
Here $K_{p,q_1,q_2}^N(r_0;\eta_0,\Bu_0)$ is given by \eqref{dfn:ul-space} below.
\end{thm}

\begin{rmk}\label{rmk:rest-q2}
\begin{enumerate}[$(1)$]
\item
The $p_0$ and $q_0$ are independent of $N$
and given by Propositions $\ref{prp:decay}$ and $\ref{prp:linear4d}$, respectively, below. 
\item
The restriction $q_2\leq 6$ comes from the embedding $H_2^1(\BR_-^N)\subset L_{q_2}(\BR_-^N)$.
We use this embedding to estimate $(\pd_t\CE_N\eta)\pd_N\Bu$ in $\SSF(\eta,\Bu)$ when $N=3$,
see Lemma $\ref{lem:nonl-F-0}$ below for more details.
\item
The condition \eqref{pq-cond-2} guarantees that the assumption $2/p+1/q\neq 1$ of Proposition $\ref{prp:t-shift}$ below
is satisfied for $q=q_1/2$,
while $2/p+1/q<1$ for $q=q_2$ because $p$ is large enough.
Here $2/p+1/q\neq 1$ is related to the existence of the boundary trace in \eqref{comp-2},
see e.g. \cite[Remark 3.4.19]{Shibata20}. 
\end{enumerate}
\end{rmk}

\subsection{Global solvability for the original problem}\label{subsec:2-4}

Let us introduce the definition of solutions of the original problem \eqref{nonl:eq1}--\eqref{initial-cond},
see e.g. \cite[page 32]{EPS03}.

\begin{dfn}
We say that the system \eqref{nonl:eq1}--\eqref{initial-cond} is globally solvable in time 
if the following assertions hold for some $p,q\in(1,\infty)$
and 
\begin{equation*}
(\eta_0,\Bv_0)\in B_{q,p}^{3-1/p-1/q}(\BR^{N-1})\times B_{q,p}^{2-2/p}(\Omega_0)^N
\end{equation*}
with $\Omega_0$ given by \eqref{omega-zero}.
\begin{enumerate}[$(1)$]
\item
There exists a diffeomorphism $\Theta_0$ from $\BR_-^N$ onto $\Omega_0$
such that
the transformed problem \eqref{nonl:fix1}--\eqref{initial-cond-flat} 
with the initial data $\eta_0$ and $\Bu_0=\Bv_0\circ\Theta_0$
admits a solution $(\eta,\Bu,\Fp)$ satisfying
\begin{align*}
\eta&\in H_p^1(\BR_+,W_q^{2-1/q}(\BR^{N-1}))\cap L_p(\BR_+,W_q^{3-1/q}(\BR^{N-1})), \\
\Bu&\in H_p^1(\BR_+,L_q(\BR_-^N)^N)\cap L_p(\BR_+,H_q^2(\BR_-^N)^N), \\
\Fp&\in L_p(\BR_+,\wht H_q^1(\BR_-^N)).
\end{align*}
Here $\Bv_0\circ\Theta_0=(\Bv_0\circ\Theta_0)(x)=\Bv_0(\Theta_0(x))$ for $x\in\BR_-^N$.
\item
Let $\Omega_\infty$ and $\Theta$ be as in \eqref{om-gam-1} and \eqref{dfn:theta}, respectively.
Then $\Theta$ is a diffeomorphism from $\BQ_-$ onto $\Omega_\infty$. 
\end{enumerate}

In this case,
setting $\Bv=\Bu(\Theta^{-1}(y,\tau))$ and $\Fq=\Fp(\Theta^{-1}(y,\tau))$ for $(y,\tau)\in\Omega_\infty$,
we call $(\eta,\Bv,\Fq)$ an $L_p\text{-}L_q$ solution global in time to \eqref{nonl:eq1}--\eqref{initial-cond}.
\end{dfn}

Recall $\Omega_0$ and $\Gamma_0$ are given by 
\eqref{omega-zero} and \eqref{gam-zero}, respectively.
Let $\Bn_0$ be the unit outward normal vector to $\Gamma_0$, i.e.,
\begin{equation*}
\Bn_0=
\frac{1}{\sqrt{1+|\nabla'\eta_0|^2}}
\begin{pmatrix}
-\nabla'\eta_0 \\ 1
\end{pmatrix}
=
\frac{1}{\sqrt{1+|\nabla'\CE_N\eta_0|^2}}
\begin{pmatrix}
-\nabla'\CE_N\eta_0 \\ 1
\end{pmatrix} \quad \text{on $\BR_0^N$.}
\end{equation*}
Our main result for the original problem \eqref{nonl:eq1}--\eqref{initial-cond} reads as follows.

\begin{thm}\label{thm:original}
Let $N\geq 3$ and $p_0,q_0$ be as in Theorem $\ref{thm:transformed}$.
Suppose that $p$, $q_1$, and $q_2$ satisfy \eqref{pq-cond}, \eqref{pq-cond-2}, and
\begin{equation}\label{add-pq}
\frac{2}{p}+\frac{N}{q_2}<1.
\end{equation}
Then there exists a small positive number $\varepsilon_1$
such that for any 
\begin{equation*}
\eta_0\in\bigcap_{r\in\{q_1/2,q_2\}}B_{r,p}^{3-1/p-1/q}(\BR^{N-1}), \quad 
\Bv_0\in\bigcap_{r\in\{q_1/2,q_2\}}B_{r,p}^{2-2/p}(\Omega_0)^N
\end{equation*}
with the smallness condition
\begin{equation}\label{small-original}
\sum_{r\in\{q_1/2,q_2\}}
\Big(\|\eta_0\|_{B_{r,p}^{3-1/p-1/r}(\BR^{N-1})}+\|\Bv_0\|_{B_{r,p}^{2-2/p}(\Omega_0)}\Big)
\leq\varepsilon_1
\end{equation}
and the compatibility conditions
\begin{align}
\dv\Bv_0&=0 \quad \text{in $\Omega_0$,} \label{comp-ori-1}\\
\mu\BD(\Bv_0)\Bn_0-\{(\mu\BD(\Bv_0)\Bn_0)\cdot\Bn_0\}\Bn_0&=0 \quad  \text{on $\Gamma_0$,} \label{comp-ori-2}
\end{align}
the system \eqref{nonl:eq1}--\eqref{initial-cond} is globally solvable in time,
i.e., the following assertions hold.
\begin{enumerate}[$(1)$]
\item
Let $\Theta_0(x)=(x',x_N+(\CE_N\eta_0)(x))$ for $x=(x',x_N)\in\BR_-^N$.
Then $\Theta_0$ is a $C^2$ diffeomorphism from $\BR_-^N$ onto $\Omega_0$.
\item
Let $\Bu_0=\Bv_0\circ\Theta_0$.
Then $(\eta_0,\Bu_0)$ satisfies \eqref{small-1}--\eqref{comp-2}, i.e.,
the transformed problem \eqref{nonl:fix1}--\eqref{initial-cond-flat} admits
a global-in-time solution $(\eta,\Bu,\Fp)$ by Theorem $\ref{thm:transformed}$.
\item
$\Theta$ is a $C^1$ diffeomorphism from $\BQ_-$ onto $\Omega_\infty$.
Furthermore, $\Theta(\cdot,t)$ is a $C^2$ diffeomorphism from $\BR_-^N$ onto $\Omega(t)$,
given by \eqref{om-gam-2}, for each $t>0$.
\end{enumerate}
\end{thm}

Let $(\eta,\Bv,\Fq)$ be the $L_p\text{-}L_{q_2}$ solution
global in time to \eqref{nonl:eq1}--\eqref{initial-cond},
i.e.,
$\Bv=\Bu(\Theta^{-1}(y,\tau))$ and $\Fq=\Fp(\Theta^{-1}(y,\tau))$ for $(y,\tau)\in \Omega_\infty$ 
with the solution $(\eta,\Bu,\Fp)$ of \eqref{nonl:fix1}--\eqref{initial-cond-flat} introduced in Theorem \ref{thm:original}.

\begin{thm}\label{thm:decay}
Let $\eta$ and $\Bv$ be as above.
Then they satisfy the following time decay properties for $q\in\{q_1,q_2\}$.
\begin{enumerate}[$(1)$]
\item
Let $N=3$. Then 
\begin{align*}
\|\Bv(\tau)\|_{B_{q,p}^{2-2/p}(\Omega(t))}&=O(\tau^{-1/3}) \quad (\tau\to \infty), \\
\|\eta(\tau)\|_{B_{q,p}^{3-1/p-1/q}(\BR^{N-1})}&=O(\tau^{-1/3}) \quad (\tau\to\infty).
\end{align*}
\item
Let $N\geq 4$. Then 
\begin{align*}
\|\Bv(\tau)\|_{B_{q,p}^{2-2/p}(\Omega(\tau))}&=O(\tau^{-1/2}) \quad (\tau\to \infty), \\
\|\eta(\tau)\|_{B_{q,p}^{3-1/p-1/q}(\BR^{N-1})}&=O(\tau^{-1/2}) \quad (\tau\to\infty).
\end{align*}
\end{enumerate}
\end{thm}

\section{Preliminaries}\label{sec:prelim}

Lemma \ref{lem:ext-interp} holds for $N\geq 3$,
while the other results of this section hold for $N\geq 2$.

\subsection{Interpolation inequalities}\label{sec:int-ineq}
Let us start with the following lemma.

\begin{lem}\label{lem:int-p}
Let $r$, $r_0$, and $r_1$ satisfy
\begin{equation*}
1< r_0<r<r_1 < \infty, \quad \frac{1}{r}=\frac{1-\theta}{r_0}+\frac{\theta}{r_1}
\end{equation*}
for some $\theta\in(0,1)$.
Then the following assertions hold.
\begin{enumerate}[$(1)$]
\item\label{lem:int-p-1}
Let $m\in\BN_0$ and $f\in H_{r_0}^m(\BR_-^N)\cap H_{r_1}^m(\BR_-^N)$. Then $f\in H_r^m(\BR_-^N)$ with
\begin{equation*}
\|f\|_{H_r^m(\BR_-^N)}\leq C\|f\|_{H_{r_0}^m(\BR_-^N)}^{1-\theta}\|f\|_{H_{r_1}^m(\BR_-^N)}^\theta
\end{equation*}
for some positive constant $C$ independent of $f$.
\item\label{lem:int-p-2}
Let $s\in\BR_+\setminus\BN$ and $f\in W_{r_0}^s(\BR^{N-1})\cap W_{r_1}^s(\BR^{N-1})$.
Then $f\in W_r^s(\BR^{N-1})$ with
\begin{equation*}
\|f\|_{W_r^s(\BR^{N-1})}\leq 
C\|f\|_{W_{r_0}^s(\BR^{N-1})}^{1-\theta}\|f\|_{W_{r_1}^s(\BR^{N-1})}^\theta
\end{equation*}
for some positive constant $C$ independent of $f$.
\end{enumerate}
\end{lem}

\begin{proof}
(1)
By \cite[Theorem 1.9.3 (f)]{Triebel78} 
\begin{equation}\label{fund-int-prp1}
\|f\|_{[X_0,X_1]_\theta}\leq C\|f\|_{X_0}^{1-\theta}\|f\|_{X_1}^\theta
\quad \text{for any $f\in X_0\cap X_1$,}
\end{equation}
where $\theta\in(0,1)$ and $\{X_0,X_1\}$ is an interpolation couple of Banach spaces $X_0$ and $X_1$.
It holds by \cite[Theorem 2.10.1]{Triebel78} that
\begin{equation}\label{fund-int-prp2}
[H_{r_0}^m(\BR_-^N),H_{r_1}^m(\BR_-^N)]_\theta=H_r^m(\BR_-^N),
\end{equation}
and thus the desired inequality follows from \eqref{fund-int-prp1}
with $X_0=H_{r_0}^m(\BR_-^N)$ and $X_1=H_{r_1}^m(\BR_-^N)$.

(2) Suppose that
\begin{equation*}
 t=(1-\theta)t_0+\theta t_1\quad \text{with $t_0,t_1>0$}
\end{equation*}
and 
\begin{equation*}
\frac{1}{q}=\frac{1-\theta}{q_0}+\frac{\theta}{q_1} \quad \text{with $1<q_0,q_1<\infty$.}
\end{equation*}
By \cite[Theorem 2.4.1 (d)]{Triebel78}
\begin{equation*}
[B_{r_0,q_0}^{t_0}(\BR^{N-1}),B_{r_1,q_1}^{t_1}(\BR^{N-1})]_{\theta}=B_{r,q}^t(\BR^{N-1}).
\end{equation*}
Choose $t_0=t_1=s$, $q_0=r_0$, and $q_1=r_1$ in this equation, and then
\begin{equation}\label{int-ss-sp}
[W_{r_0}^s(\BR^{N-1}),W_{r_1}^s(\BR^{N-1})]_\theta = W_{r}^s(\BR^{N-1}).
\end{equation}
The desired inequality thus follows from \eqref{int-ss-sp}
and \eqref{fund-int-prp1} with $X_0=W_{r_0}^s(\BR^{N-1})$ and $X_1=W_{r_1}^s(\BR^{N-1})$.
This completes the proof of Lemma \ref{lem:int-p}.
\end{proof}

Furthermore, we have

\begin{lem}\label{lem:int-p-v2}
Let $p\in(1,\infty)$.
Let $r$, $r_0$, and $r_1$ satisfy
\begin{equation*}
1< r_0<r<r_1 < \infty, \quad \frac{1}{r}=\frac{1-\theta}{r_0}+\frac{\theta}{r_1}
\end{equation*}
for some $\theta\in(0,1)$.
Then the following assertions hold.
\begin{enumerate}[$(1)$]
\item\label{lem:int-p-1-2}
Let $m\in\BN_0$ and $f\in L_p(\BR_+,H_{r_0}^m(\BR_-^N))\cap L_p(\BR_+,H_{r_1}^m(\BR_-^N))$. 
Then 
\begin{align*}
&f\in L_p(\BR_+,H_r^m(\BR_-^N)) \text{ with} \\
&\|f\|_{L_p(\BR_+,H_r^m(\BR_-^N))}\leq C\|f\|_{L_p(\BR_+,H_{r_0}^m(\BR_-^N))}^{1-\theta}
\|f\|_{L_p(\BR_+,H_{r_1}^m(\BR_-^N))}^\theta
\end{align*}
for some positive constant $C$ independent of $f$.
\item\label{lem:int-p-5}
Let $s\in\BR_+\setminus\BN$ and $f\in L_p(\BR_+,W_{r_0}^s(\BR^{N-1}))\cap L_p(\BR_+,W_{r_1}^s(\BR^{N-1}))$.
Then 
\begin{align*}
&f\in L_p(\BR_+,W_r^s(\BR^{N-1})) \text{ with} \\
&\|f\|_{L_p(\BR_+,W_r^s(\BR^{N-1}))}\leq 
C\|f\|_{L_p(\BR_+,W_{r_0}^s(\BR^{N-1}))}^{1-\theta}\|f\|_{L_p(\BR_+,W_{r_1}^s(\BR^{N-1}))}^\theta
\end{align*}
for some positive constant $C$ independent of $f$.
\end{enumerate}
\end{lem}

\begin{proof}
By \cite[Theorem 1.18.4]{Triebel78}
\begin{equation*}
[L_p(\BR_+,X_0),L_p(\BR_+,X_1)]_\theta =L_p(\BR_+,[X_0,X_1]_\theta), 
\end{equation*}
where $\{X_0,X_1\}$ is an interpolation couple of Banach spaces. 
Combining this property with \eqref{fund-int-prp1}--\eqref{int-ss-sp}
yields the desired inequalities of \eqref{lem:int-p-1-2} and \eqref{lem:int-p-5}.
This completes the proof of Lemma \ref{lem:int-p-v2}.
\end{proof}

\subsection{Embeddings}\label{subsec:embed}

Let $(\cdot, \cdot)_{\theta,p}$ be the real interpolation functor for $\theta\in(0,1)$ and $p\in(1,\infty)$.
Recall the argumentation in  \cite[Section 1.4]{Tanabe97}, i.e.,

\begin{lem}\label{lem:tanabe}
Let $p\in(1,\infty)$.
Let $X_0$ and $X_1$ be Banach spaces such that $X_1$ is a dense subspace of $X_0$
and $X_1$ is continuously embedded into $X_0$. Then
\begin{equation*}
H_p^1(\BR_+,X_0)\cap L_p(\BR_+,X_1)
\subset C([0,\infty),(X_0,X_1)_{1-1/p,p})
\end{equation*}
with
\begin{equation*}
\sup_{t\in [0,\infty)}\|f(t)\|_{(X_0,X_1)_{1-1/p,p}}
\leq \Big(\|f\|_{H_p^1(\BR_+,X_0)}^p+\|f\|_{L_p(\BR_+,X_1)}^p
\Big)^{1/p}
\end{equation*}
for any $f\in H_p^1(\BR_+,X_0)\cap L_p(\BR_+,X_1)$, where
\begin{equation}\label{dfn:norm}
\|f\|_{H_p^1(I,X_0)}=\Big(\|f\|_{L_p(I,X_0)}^p+\|\pd_t f\|_{L_p(I,X_0)}^p\Big)^{1/p}
\end{equation}
for a time interval $I$. 
\end{lem}

The first application of Lemma \ref{lem:tanabe} is as follows.

\begin{lem}\label{lem:H^1-embed}
Let $p\in(1,\infty)$ and $X$ be a Banach space with the norm $\|\cdot\|_X$.
Then $H_p^1(\BR_+,X)\subset C([0,\infty),X)$ with
\begin{equation*}
\sup_{t\in [0,\infty)}\|f(t)\|_X \leq C\|f\|_{H_p^1(\BR_+,X)}
\end{equation*}
for any $f\in H_p^1(\BR_+,X)$ and some positive constant $C$ independent of $f$.
\end{lem}

\begin{proof}
Let us write 
\begin{equation*}
H_p^1(\BR_+,X)=H_p^1(\BR_+,X)\cap L_p(\BR_+,X).
\end{equation*}
Since $(X,X)_{\theta,p}=X$ for $\theta\in(0,1)$,
Lemma \ref{lem:tanabe} with $X_0=X_1=X$  yields the desired property.
This completes the proof of Lemma \ref{lem:H^1-embed}.
\end{proof}

The second application of Lemma \ref{lem:tanabe} is as follows.

\begin{lem}\label{lem:time-sp}
Let $p,q\in (1,\infty)$.
Then the following assertions hold.
\begin{enumerate}[$(1)$]
\item\label{lem:time-sp-1}
There holds
\begin{align*}
&H_p^1(\BR_+,W_q^{2-1/q}(\BR^{N-1}))\cap L_p(\BR_+,W_q^{3-1/q}(\BR^{N-1})) \\
&\subset C([0,\infty),B_{q,p}^{3-1/p-1/q}(\BR^{N-1}))
\end{align*}
with
\begin{align*}
&\sup_{t\in [0,\infty)}\|f(t)\|_{B_{q,p}^{3-1/p-1/q}(\BR^{N-1})} \\
&\leq 
C\Big(
\|f\|_{H_p^1(\BR_+,W_q^{2-1/q}(\BR^{N-1}))}
+\|f\|_{L_p(\BR_+,W_q^{3-1/q}(\BR^{N-1}))}
\Big)
\end{align*}
for any $f\in H_p^1(\BR_+,W_q^{2-1/q}(\BR^{N-1}))\cap L_p(\BR_+,W_q^{3-1/q}(\BR^{N-1}))$
and some positive constant $C$ independent of $f$.
\item\label{lem:time-sp-2}
There holds
\begin{equation*}
H_p^1(\BR_+,L_q(\BR_-^N))\cap L_p(\BR_+,H_q^{2}(\BR_-^{N})) 
\subset C([0,\infty),B_{q,p}^{2-2/p}(\BR_-^{N}))
\end{equation*}
with
\begin{equation*}
\sup_{t\in [0,\infty)}\|f(t)\|_{B_{q,p}^{2-2/p}(\BR_-^{N})} 
\leq 
C\Big(
\|f\|_{H_p^1(\BR_+,L_q(\BR_-^N))}
+\|f\|_{L_p(\BR_+,H_q^{2}(\BR_-^{N}))}
\Big)
\end{equation*}
for any $f \in H_p^1(\BR_+,L_q(\BR_-^N))\cap L_p(\BR_+,H_q^{2}(\BR_-^{N}))$
and some positive constant $C$ independent of $f$.
\end{enumerate}
\end{lem}

\begin{proof}
It holds that
\begin{align}
(W_q^{2-1/q}(\BR^{N-1}),W_q^{3-1/q}(\BR^{N-1}))_{1-1/p,p} &= B_{q,p}^{3-1/p-1/q}(\BR^{N-1}), \label{int-p-SS-sp-1} \\
(H_q^j(\BR_-^{N}),H_q^2(\BR_-^{N}))_{1-1/p,p} &= B_{q,p}^{2(1-1/p)+j/p}(\BR_-^{N}), \label{int-p-SS-sp-2}
\end{align}
where $j=0,1$,
see Subsections 2.4.2 and 2.10.1 of \cite{Triebel78}.
Combining \eqref{int-p-SS-sp-1} and \eqref{int-p-SS-sp-2} for $j=0$
with Lemma \ref{lem:tanabe} yields
the desired properties of \eqref{lem:time-sp-1} and \eqref{lem:time-sp-2}.
This completes the proof of Lemma \ref{lem:time-sp}.
\end{proof}

We next consider embeddings related to $H_p^{1/2}(\BR_+,H_q^1(\BR_-^N))$.

\begin{lem}\label{int-H-1/2}
Let $p,q\in (1,\infty)$.
Then the following assertions hold.
\begin{enumerate}[$(1)$]
\item\label{int-H-1/2-1}
There holds
\begin{equation*}
H_p^1(\BR_+,L_q(\BR_-^N))\cap L_p(\BR_+,H_q^{2}(\BR_-^{N})) 
\subset H_p^{1/2}(\BR_+,H_q^1(\BR_-^N))
\end{equation*}
with
\begin{equation*}
\|f\|_{H_p^{1/2}(\BR_+,H_q^1(\BR_-^N))}
\leq C\Big(
\|f\|_{H_p^1(\BR_+,L_q(\BR_-^N))}
+\|f\|_{L_p(\BR_+,H_q^{2}(\BR_-^{N}))}
\Big)
\end{equation*}
for any $f\in H_p^1(\BR_+,L_q(\BR_-^N))\cap L_p(\BR_+,H_q^{2}(\BR_-^{N}))$
and some positive constant $C$ independent of $f$.
\item\label{int-H-1/2-2}
For any $f\in H_p^1(\BR_+,L_q(\BR_-^N))\cap L_p(\BR_+,H_q^{2}(\BR_-^{N}))$ and for $j=1,\dots,N$,
\begin{equation*}
\|\pd_j f\|_{H_p^{1/2}(\BR_+,L_q(\BR_-^N))}
\leq C
\Big(
\|f\|_{H_p^1(\BR_+,L_q(\BR_-^N))}
+\|f\|_{L_p(\BR_+,H_q^{2}(\BR_-^{N}))}
\Big),
\end{equation*}
where $\pd_j=\pd /\pd x_j$ and $C$ is a positive constant independent of $f$.
\end{enumerate}
\end{lem}

\begin{proof}
\eqref{int-H-1/2-1}
Let $E^e f$ be the even extension of $f$ with respect to the time variable $t$, i.e.,
\begin{equation*}
E^e f=
\left\{\begin{aligned}
&f(t) && (t>0), \\
& f(-t) && (t<0),
\end{aligned}\right.
\end{equation*}
and let $R$ be the restriction operator from $\BR$ to $\BR_+$.
Suppose $X$ is a Banach space and recall \eqref{dfn:norm}.
Then
\begin{align*}
E^e\in \CL(L_p(\BR_+,X),L_p(\BR,X))\cap \CL(H_p^1(\BR_+,X),H_p^1(\BR,X))
\end{align*}
with 
\begin{align}\label{est:odd-ext}
&\|E^e\|_{\CL(L_p(\BR_+,X),L_p(\BR,X))}
\leq 2^{1/p}, \notag \\
&\|E^e\|_{\CL(H_p^1(\BR_+,X),H_p^1(\BR,X))}
\leq  2^{1/p},
\end{align}
while 
\begin{equation*}
R\in \CL(L_p(\BR,X),L_p(\BR_+,X))\cap \CL(H_p^1(\BR,X),H_p^1(\BR_+,X))
\end{equation*}
with
\begin{equation*}
\|R\|_{\CL(L_p(\BR,X),L_p(\BR_+,X))}\leq 1, \quad 
\|R\|_{\CL(H_p^1(\BR,X),H_p^1(\BR_+,X))} \leq 1.
\end{equation*}
The complex interpolation method\footnote{
Here and subsequently, the complex interpolation method
stands for \cite[Theorem 2.6]{Lunardi18}.
}
thus yields
\begin{align}\label{comp-int:1}
&E^e\in\CL(H_p^{1/2}(\BR_+,X),H_p^{1/2}(\BR,X)), \notag \\
&\|E^e\|_{\CL(H_p^{1/2}(\BR_+,X),H_p^{1/2}(\BR,X))} \leq  2^{1/p},
\end{align}
and also
\begin{align}\label{comp-int:2}
&R\in \CL(H_p^{1/2}(\BR,X),H_p^{1/2}(\BR_+,X)), \notag \\
&\|R\|_{\CL(H_p^{1/2}(\BR,X),H_p^{1/2}(\BR_+,X))}\leq 1.
\end{align}

Recall \cite[Proposition 1]{Shibata18}, which tells us that 
\begin{equation}\label{lem:Shibata18}
\|g\|_{H_p^{1/2}(\BR,H_q^1(\BR_-^N))}
\leq C\Big(\|g\|_{H_p^1(\BR,L_q(\BR_-^N))}+\|g\|_{L_p(\BR,H_q^2(\BR_-^N))}\Big)
\end{equation}
for any $g\in H_p^1(\BR,L_q(\BR_-^N))\cap L_p(\BR,H_q^2(\BR_-^N))$.

Let $f\in H_p^1(\BR_+,L_q(\BR_-^N))\cap L_p(\BR_+,H_q^2(\BR_-^N))$.
From \eqref{est:odd-ext}--\eqref{lem:Shibata18}, we see that
\begin{align*}
\|f\|_{H_p^{1/2}(\BR_+,H_q^1(\BR_-^N))}
&=\|R E^e f\|_{H_p^{1/2}(\BR_+,H_q^1(\BR_-^N))} \\
&\leq \|E^e f\|_{H_p^{1/2}(\BR,H_q^1(\BR_-^N))} \\
&\leq 
C\Big(\|E^e f\|_{H_p^1(\BR,L_q(\BR_-^N))}+\|E^e f\|_{L_p(\BR,H_q^2(\BR_-^N))}\Big) \\
&\leq 
C\Big(\|f\|_{H_p^1(\BR_+,L_q(\BR_-^N))}+\|f\|_{L_p(\BR_+,H_q^2(\BR_-^N))}\Big). 
\end{align*}
The desired property thus holds.

\eqref{int-H-1/2-2} It holds that
\begin{align*}
\pd_j &\in \CL(H_p^1(\BR_+,H_q^1(\BR_-^N)),H_p^1(\BR_+,L_q(\BR_-^N))) \\
&\cap \CL(L_p(\BR_+,H_q^1(\BR_-^N)),L_p(\BR_+,L_q(\BR_-^N)))
\end{align*}
and their operator norms are bounded by $1$. The complex interpolation method then gives us
\begin{align*}
&\pd_j \in\CL(H_p^{1/2}(\BR_+,H_q^1(\BR_-^N)),H_p^{1/2}(\BR_+,L_q(\BR_-^N))), \\
&\|\pd_j f\|_{H_p^{1/2}(\BR_+,L_q(\BR_-^N))}\leq \|f\|_{H_p^{1/2}(\BR_+,H_q^1(\BR_-^N))},
\end{align*}
for any $f\in H_p^{1/2}(\BR_+,H_q^1(\BR_-^N))$.
Combining the last inequality with \eqref{int-H-1/2-1} furnishes the desired inequality.
This completes the proof of Lemma \ref{int-H-1/2}.
\end{proof}

We here collect fundamental embeddings in the following two lemmas.

\begin{lem}\label{lem:fund-embed}
Let $p,q\in(1,\infty)$.
Then the following assertions hold.
\begin{enumerate}[$(1)$]
\item\label{lem:fund-embed-1}
$B_{q,p}^{3-1/p-1/q}(\BR^{N-1})\subset W_q^{2-1/q}(\BR^{N-1})$ with
\begin{equation*}
\|f\|_{W_q^{2-1/q}(\BR^{N-1})}
\leq C\|f\|_{B_{q,p}^{3-1/p-1/q}(\BR^{N-1})}
\end{equation*}
for any $f\in B_{q,p}^{3-1/p-1/q}(\BR^{N-1})$ and some positive constant $C$ independent of $f$.
\item\label{lem:fund-embed-2}
If $p>2$ additionally, then
$B_{q,p}^{2-2/p}(\BR_-^N)\subset H_q^1(\BR_-^N)$ with
\begin{equation*}
\|f\|_{H_q^1(\BR_-^N)}\leq C\|f\|_{B_{q,p}^{2-2/p}(\BR_-^N)}
\end{equation*}
for any $f\in B_{q,p}^{2-2/p}(\BR_-^N)$ and some positive constant $C$ independent of $f$.
\item\label{lem:fund-embed-3}
Let $m\in\BN$.
If $q>N$ additionally, then $H_q^m(\BR_-^N)\subset C_B^{m-1}(\overline{\BR_-^N})$ with
\begin{equation*}
\|f\|_{H_\infty^{m-1}(\BR_-^N)}\leq C\|f\|_{H_q^m(\BR_-^N)}
\end{equation*}
for any $f\in H_q^m(\BR_-^N)$ and some positive constant $C$ independent of $f$. 
\end{enumerate}
\end{lem}

\begin{proof}
(1) See Remark 1 (a) in Subsection 2.8.1 of \cite{Triebel78}.

(2) We first prove the whole space case, i.e.,
\begin{equation}\label{embed-whole-besov}
B_{q,p}^{2-2/p}(\BR^N)\subset H_q^1(\BR^N)
\text{ with }
\|f\|_{H_q^1(\BR^N)}\leq C\|f\|_{B_{q,p}^{2-2/p}(\BR^N)}
\end{equation}
for any $f\in B_{q,p}^{2-2/p}(\BR^N)$.
Let $s\in\BR$ and $\varepsilon>0$. Recall the well-known property
\begin{equation*}
B_{q,p}^{s}(\BR^N) \subset B_{q,r}^{s-\varepsilon}(\BR^N) \text{ with }
\|f\|_{B_{q,q}^{s-\varepsilon}(\BR^N)}\leq C\|f\|_{B_{q,p}^{s}(\BR^N)}
\end{equation*}
for any $f\in B_{q,p}^{s}(\BR^N)$ and $r\in(1,\infty)$. 
Choose $r=q$ and $s=2-2/p$ together with $\varepsilon>0$ so small that $2/p+\varepsilon<1$. Then
$B_{q,r}^{s-\varepsilon}(\BR^N)=W_q^{2-2/p-\varepsilon}(\BR^N)$
and $W_q^{2-2/p-\varepsilon}(\BR^N)$ is continuously embeded into $H_q^1(\BR^N)$.
Thus \eqref{embed-whole-besov} hold.

The desired property for the half space $\BR_-^N$ follows from \eqref{embed-whole-besov}
and a suitable extension operator, e.g. \cite[Theorem 2.2]{Rychkov99}.
This completes the proof of \eqref{lem:fund-embed-2}.

(3) See \cite[Theorem 4.12]{AF03}.
\end{proof}

\begin{lem}\label{lem:imbed-bexov-conti}
Let $p\in(1,\infty)$ and $q\in(N,\infty)$.
Suppose that $m\in\BN$ and that
$G=\BR_-^N$ or $G=\BR^{N-1}\times (-L,0)$ for $L>0$.
Then the following assertions hold.
\begin{enumerate}[$(1)$]
\item\label{lem:imbed-bexov-conti-1}
$B_{q,p}^{m+1-1/p}(G) \subset C_B^{m-1}(\overline{G})$ with
\begin{equation*}
\|f\|_{H_\infty^{m-1}(G)}\leq C_G\|f\|_{B_{q,p}^{m+1-1/p}(G)} \quad  \text{for any $f\in B_{q,p}^{m+1-1/p}(G)$,}
\end{equation*}
where $C_G$ is a positive constant depending on $G$, but independent of $f$.
\item\label{lem:imbed-bexov-conti-2}
Let $1/p+N/q<1$ additionally. Then
$B_{q,p}^{m-1/p}(G) \subset C_B^{m-1}(\overline{G})$ with
\begin{equation*}
\|f\|_{H_\infty^{m-1}(G)}\leq C_G\|f\|_{B_{q,p}^{m-1/p}(G)} \quad  \text{for any $f\in B_{q,p}^{m-1/p}(G)$,}
\end{equation*}
where $C_G$ is a positive constant depending on $G$, but independent of $f$.
\item\label{lem:imbed-bexov-conti-3}
Let $p\in(2,\infty)$ and $2/p+N/q<1$ additionally.
Then $B_{q,p}^{m-2/p}(G)\subset C_B^{m-1}(\overline{G})$
with
\begin{equation*}
\|f\|_{H_\infty^{m-1}(G)}
\leq C_G\|f\|_{B_{q,p}^{m-2/p}(G)} \quad \text{for any $f\in B_{q,p}^{m-2/p}(G)$,}
\end{equation*}
where  $C_G$ is a positive constant depending on $G$, but independent of $f$.
\end{enumerate}
\end{lem}

\begin{proof}
See Remark 1 (b) in Subsection 2.8.1 of \cite{Triebel78} for the whole space case.
Then, similarly to the proof of Lemma \ref{lem:fund-embed} \eqref{lem:fund-embed-2},
we can prove by the extension operator
the desired properties of \eqref{lem:imbed-bexov-conti-1}--\eqref{lem:imbed-bexov-conti-3}.
This completes the proof of Lemma \ref{lem:imbed-bexov-conti}.
\end{proof}

The following lemma then holds.

\begin{lem}\label{lem:height-func}
Let $p\in(1,\infty)$ and $q\in(N,\infty)$. Then 
\begin{equation*}
H_p^1(\BR_+,H_q^1(\BR_-^N))\cap L_p(\BR_+,H_q^2(\BR_-^N))
\subset C([0,\infty),C_B(\overline{\BR_-^N}))
\end{equation*}
with
\begin{equation*}
\sup_{t\in [0,\infty)}\|f(t)\|_{L_\infty(\BR_-^N)}
\leq M_0
\Big(\|\pd_t f\|_{L_p(\BR_+,H_q^1(\BR_-^N))}+\|f\|_{L_p(\BR_+,H_q^2(\BR_-^N))}\Big)
\end{equation*}
for any $f\in H_p^1(\BR_+,H_q^1(\BR_-^N))\cap L_p(\BR_+,H_q^2(\BR_-^N))$
and some positive constant $M_0$ independent of $f$.
\end{lem}

\begin{proof}
We use \eqref{int-p-SS-sp-2} for $j=1$ and Lemma \ref{lem:tanabe} in order to obtain
\begin{equation*}
H_p^1(\BR_+,H_q^1(\BR_-^N))\cap L_p(\BR_+,H_q^2(\BR_-^N))
\subset C([0,\infty),B_{q,p}^{2-1/p}(\BR_-^N))
\end{equation*}
with
\begin{equation*}
\sup_{t\in [0,\infty)}\|f(t)\|_{B_{q,p}^{2-1/p}(\BR_-^N)}
\leq C\Big(\|f\|_{H_p^1(\BR_+,H_q^1(\BR_-^N))}+\|f\|_{L_p(\BR_+,H_q^2(\BR_-^N))}\Big).
\end{equation*}
It then follows from Lemma \ref{lem:imbed-bexov-conti} \eqref{lem:imbed-bexov-conti-1}
with $m=1$ that the desired property holds.
This completes the proof of Lemma \ref{lem:height-func}.
\end{proof}

\subsection{Estimates of $H_p^{1/2}$ norm in time}

This subsection proves

\begin{lem}\label{lem:H-half}
Let $p,q\in(1,\infty)$ and suppose that $r,s$ satisfy
$q\leq  r,s \leq \infty$ and  $1/q=1/r+1/s$.
Then for any $f\in H_p^1(\BR_+,L_r(\BR_-^N))$ and $g\in H_p^{1/2}(\BR_+,L_s(\BR_-^N))$
\begin{equation*}
\|fg\|_{H_p^{1/2}(\BR_+,L_q(\BR_-^N))}
\leq C\|f\|_{H_p^1(\BR_+,L_r(\BR_-^N))}\|g\|_{H_p^{1/2}(\BR_+,L_s(\BR_-^N))},
\end{equation*}
where $C$ is a positive constant independent of $f$ and $g$.
\end{lem}

\begin{proof}
Let us define $T_f g=fg$ for $f\in H_p^1(\BR_+,L_r(\BR_-^N))$.
Then
\begin{equation*}
\|T_f g\|_{L_p(\BR_+,L_q(\BR_-^N))}
\leq \|f\|_{L_\infty(\BR_+,L_r(\BR_-^N))}\|g\|_{L_p(\BR_+,L_s(\BR_-^N))}.
\end{equation*}
It holds by Lemma \ref{lem:H^1-embed} that 
\begin{equation*}
\|f\|_{L_\infty(\BR_+,L_r(\BR_-^N))}\leq C\|f\|_{H_p^1(\BR_+,L_r(\BR_-^N))},
\end{equation*}
which, combined the estimate of $\|T_f g\|_{L_p(\BR_+,L_q(\BR_-^N))}$ above, furnishes
\begin{equation}\label{Tf-est}
\|T_f g\|_{L_p(\BR_+,L_q(\BR_-^N))}
\leq  C\|f\|_{H_p^1(\BR_+,L_r(\BR_-^N))}\|g\|_{L_p(\BR_+,L_s(\BR_-^N))}.
\end{equation}
Hence,
\begin{align}\label{comp-int:3}
&T_f\in\CL(L_p(\BR_+,L_s(\BR_-^N)),L_p(\BR_+,L_q(\BR_-^N))), \notag \\
&\|T_f\|_{\CL(L_p(\BR_+,L_s(\BR_-^N)),L_p(\BR_+,L_q(\BR_-^N)))}\leq C \|f\|_{H_p^1(\BR_+,L_r(\BR_-^N))}.
\end{align}
On the other hand, $\pd_t(T_fg)=(\pd_t f)g+f\pd_t g$ and Lemma \ref{lem:H^1-embed} yield
\begin{align*}
\|\pd_t(T_fg)\|_{L_p(\BR_+,L_q(\BR_-^N))}
&\leq \|\pd_t f\|_{L_p(\BR_+,L_r(\BR_-^N))}\|g\|_{L_\infty(\BR_+,L_s(\BR_-^N))} \\
&+\|f\|_{L_\infty(\BR_+,L_r(\BR_-^N))}\|\pd_t g\|_{L_p(\BR_+,L_s(\BR_-^N))} \\
&\leq C\| f\|_{H_p^1(\BR_+,L_r(\BR_-^N))}\|g\|_{H_p^1(\BR_+,L_s(\BR_-^N))},
\end{align*}
which, combined with \eqref{Tf-est}, furnishes
\begin{equation*}
\|T_f g\|_{H_p^1(\BR_+,L_q(\BR_-^N))}
\leq C\|f\|_{H_p^1(\BR_+,L_r(\BR_-^N))}\|g\|_{H_p^1(\BR_+,L_s(\BR_-^N))}.
\end{equation*}
Hence,
\begin{align}\label{comp-int:4}
&T_f\in \CL(H_p^1(\BR_+,L_s(\BR_-^N)),H_p^1(\BR_+,L_q(\BR_-^N))), \notag \\
&\|T_f\|_{\CL(H_p^1(\BR_+,L_s(\BR_-^N)),H_p^1(\BR_+,L_q(\BR_-^N)))}\leq C\| f\|_{H_p^1(\BR_+,L_r(\BR_-^N))}.
\end{align}

From \eqref{comp-int:3}, \eqref{comp-int:4}, and the complex interpolation method,
we obtain 
\begin{align*}
&T_f\in \CL(H_p^{1/2}(\BR_+,L_s(\BR_-^N)),H_p^{1/2}(\BR_+,L_q(\BR_-^N))), \\
&\|T_f\|_{\CL(H_p^{1/2}(\BR_+,L_s(\BR_-^N)),H_p^{1/2}(\BR_+,L_q(\BR_-^N)))}
\leq C\| f\|_{H_p^1(\BR_+,L_r(\BR_-^N))}.
\end{align*}
These properties imply that the desired inequality holds.
This completes the proof of Lemma \ref{lem:H-half}.
\end{proof}

\subsection{Properties of $\CE_N$}

Let us first introduce the following lemma.

\begin{lem}\label{lem-sym1}
Let $\xi'=(\xi_1,\dots,\xi_{N-1})\in\BR^{N-1}$
and $\alpha'=(\alpha_1,\dots,\alpha_{N-1})\in\BN_0^{N-1}$. Then the following assertions hold.
\begin{enumerate}[$(1)$]
\item\label{lem-sym1-1}
Let $s\in\BR$. Then
\begin{equation*}
|\pd_{\xi'}^{\alpha'}(1+|\xi'|^2)^{s/2}|\leq C(1+|\xi'|)^{s-|\alpha'|},
\end{equation*}
where $C=C(\alpha',s)$ is a positive constant independent of $\xi'$.
\item\label{lem-sym1-2}
Let $\Xi=(i\xi_1,\dots,i\xi_{N-1},\sqrt{1+|\xi'|^2})$ and
$\beta=(\beta_1,\dots,\beta_{N-1},\beta_N)\in\BN_0^N$.
Then
\begin{equation*}
|\pd_{\xi'}^{\alpha'}\Xi^\beta|\leq C(1+|\xi'|)^{|\beta|-|\alpha'|},
\end{equation*}
where $\Xi^\beta=(i\xi_1)^{\beta_1}\dots(i\xi_{N-1})^{\beta_{N-1}}\sqrt{1+|\xi'|^2}^{\,\beta_N}$
and $C=C(\alpha',\beta)$ is a positive constant independent of $\xi'$.
\end{enumerate}
\end{lem}

\begin{proof}
(1) See \cite[Lemma 5.2]{SS12}.

(2) It holds for $j=1,\dots,N-1$ that 
\begin{align*}
\bigg|\frac{\pd^{\alpha_j}}{\pd \xi_j^{\alpha_j}}(i\xi_j)^{\beta_j}\bigg|
&\leq C\left\{\begin{aligned}
&|\xi'|^{\beta_j-\alpha_j} && (\beta_j\geq \alpha_j) \\
&0 && (\beta_j<\alpha_j) 
\end{aligned}\right. \\
&\leq C (1+|\xi'|)^{\beta_j-\alpha_j},
\end{align*}
which furnishes
\begin{equation*}
|\pd_{\xi'}^{\alpha'}(i\xi_j)^{\beta_j}|\leq C(1+|\xi'|)^{\beta_j-|\alpha'|}.
\end{equation*}
On the other hand, it follows from the inequality of \eqref{lem-sym1-1} with $s=\beta_N$ that
\begin{equation*}
|\pd_{\xi'}^{\alpha'}\sqrt{1+|\xi'|^2}^{\,\beta_N}|
\leq C(1+|\xi'|)^{\beta_N-|\alpha'|}.
\end{equation*}
Combining these two inequalities with the Leibniz rule,
we immediately obtain the desired inequality.
This completes the proof of Lemma \ref{lem-sym1}.
\end{proof}

Recall the partial Fourier transform and its inverse transform given by \eqref{fourier-t}.
The next lemma is proved in \cite[Lemma 5.4]{SS12}.

\begin{lem}\label{lem:tech}
Let $q\in(1,\infty)$ and $m(\xi')$ be a $C^\infty$ function on $\BR^{N-1}\setminus\{0\}$ with
\begin{equation*}
|\pd_{\xi'}^{\alpha'}m(\xi')|\leq M(1+|\xi'|)^{-|\alpha'|} \quad (\xi'\in\BR^{N-1}\setminus\{0\}),
\end{equation*}
where $\alpha'\in\BN_0^{N-1}$ and $M$ is a positive constant independent of $\xi'$.
Define
\begin{equation*}
[K f](x)=\int_{-\infty}^0\CF_{\xi'}^{-1}[m(\xi')\sqrt{1+|\xi'|^2} e^{\sqrt{1+|\xi'|^2}(x_N+y_N)}\wht f(\xi',y_N)](x')\intd y_N 
\end{equation*}
for $x=(x',x_N)\in\BR_-^N$. Then
\begin{equation*}
\|K f\|_{L_q(\BR_-^N)}\leq CM\|f\|_{L_q(\BR_-^N)}
\end{equation*}
for some positive constant $C$ independent of $f$.
\end{lem}

Define the operators $\CA$ and $\CB$ by
\begin{align}\label{dfn:KL}
[\CA f](x)&=\CF_{\xi'}^{-1}[e^{|\xi'|x_N}\wht f(\xi')](x'), \notag \\
[\CB f](x)&=\CF_{\xi'}^{-1}[e^{\sqrt{1+|\xi'|^2}x_N}\wht f(\xi')](x'),
\end{align}
where $x=(x',x_N)\in\BR_-^N$ and $N\geq 2$. We then have

\begin{lem}\label{lem:ext}
Let $q\in(1,\infty)$ and $m\geq 1$.
Then the following assertions hold.
\begin{enumerate}[$(1)$]
\item\label{lem:ext-1}
There exists a positive constant $C$ such that for any $f\in L_q(\BR^{N-1})$
\begin{equation}\label{A-ineq-3}
\sup_{x_N\leq 0}\|[\CA f](\cdot,x_N)\|_{L_q(\BR^{N-1})} \leq C \|f\|_{L_q(\BR^{N-1})}.
\end{equation}
Furthermore, there exists a positive constant $C$ such that for any $f \in H_q^m(\BR_-^N)$
and for any multi-index $\alpha\in\BN_0^N$ with $|\alpha|=m$
\begin{equation}\label{A-ineq-4}
\|\pd_x^\alpha \CA(f|_{\BR_0^N})\|_{L_q(\BR_-^N)}
\leq C\sum_{|\beta|=m}\|\pd_x^\beta f\|_{L_q(\BR_-^N)},
\end{equation}
where $f|_{\BR_0^N}$ denotes the boundary trace of $f$ on $\BR_0^N$.
\item\label{lem:ext-3}
There exists a positive constant $C$ such that
for any $f \in W_q^{m-1/q}(\BR^{N-1})$
\begin{equation}\label{A-ineq-1}
\|\nabla \CA f\|_{H_q^{m-1}(\BR_-^{N})}
\leq C\|f\|_{W_q^{m-1/q}(\BR^{N-1})}.
\end{equation}
Furthermore, for any $L>0$ and $f \in W_q^{m-1/q}(\BR^{N-1})$,
\begin{equation}\label{A-ineq-2}
\|\CA f\|_{H_q^m(\BR^{N-1}\times (-L,0))} \leq C_L\|f\|_{W_q^{m-1/q}(\BR^{N-1})}
\end{equation}
with some positive constant $C_L$ depending on $L$, but independent of $f$.
\item\label{lem:ext-2}
There exist a positive constant $C$ such that
for any $f\in W_q^{m-1/q}(\BR^{N-1})$
\begin{equation*}
C^{-1} \|f\|_{W_q^{m-1/q}(\BR^{N-1})}
\leq \|\CB f\|_{H_q^m(\BR_-^N)}\leq C\|f\|_{W_q^{m-1/q}(\BR^{N-1})}.
\end{equation*}
\end{enumerate}
\end{lem}

\begin{proof}
(1) See  the proof of Lemma 7.1 in \cite{SS12}.

(2) Let $F\in H_q^m(\BR_-^N)$ be an extension of $f$ such that $F=f$ on $\BR_0^N$ and 
$\|F\|_{H_q^m(\BR_-^N)}\leq C\|f\|_{W_q^{m-1/q}(\BR^{N-1})}$. 
Since $F|_{\BR_0^N}=f$, one observes by \eqref{A-ineq-4} that
\begin{align*}
\sum_{|\alpha|=1}^m \|\pd_x^\alpha \CA f\|_{L_q(\BR_-^N)}
 \leq C\sum_{|\beta|=1}^m\|\pd_x^\beta F\|_{L_q(\BR_-^N)} 
\leq C\|f\|_{W_q^{m-1/q}(\BR^{N-1})}.
\end{align*}
Thus \eqref{A-ineq-1} holds.

We next prove \eqref{A-ineq-2}. Let $L>0$ and $\varphi$ be a function in $C_0^\infty(\BR)$ such that
$\varphi(x_N)=1$ when $|x_N|\leq 2L$ and $\varphi(x_N)=0$ when $|x_N|\geq 3L$. 
Then  for $x_N\in(-L,0)$
\begin{align*}
&[\CA f](x',x_N)
=\int_{-\infty}^{x_N} \frac{d}{d y_N}\left(\varphi(y_N)[\CA f](x',y_N)\right)\intd y_N \\
&=\int_{-\infty}^{x_N}\varphi'(y_N)[\CA f](x',y_N)\intd y_N 
+\int_{-\infty}^{x_N}\varphi(y_N)[\pd_N\CA f](x',y_N)\intd y_N \\
&=:I_1+I_2. 
\end{align*}
It holds that
\begin{align*}
\|I_1(\cdot,x_N)\|_{L_q(\BR^{N-1})}
&\leq \int_{-\infty}^{x_N}|\varphi'(y_N)|\|[\CA f](\cdot,y_N)\|_{L_q(\BR^{N-1})}\intd y_N \\
&\leq \Big(\sup_{x_N\leq 0}\|[\CA f](\cdot,x_N)\|_{L_q(\BR^{N-1})}\Big) \int_{-\infty}^{0}|\varphi'(y_N)|\intd y_N, 
\end{align*}
which, combined with \eqref{A-ineq-3}, furnishes
\begin{equation*}
\|I_1(\cdot,x_N)\|_{L_q(\BR^{N-1})}
\leq C\|f\|_{L_q(\BR^{N-1})} \quad \text{for $x_N\in(-L,0)$.}
\end{equation*}
This gives us
\begin{equation}\label{est-I1-A}
\|I_1\|_{L_q(\BR^{N-1}\times(-L,0))}\leq C_L\|f\|_{L_q(\BR^{N-1})}
\end{equation}
with some positive constant $C_L$ depending on $L$.
In addition, 
\begin{align*}
\|I_2(\cdot,x_N)\|_{L_q(\BR^{N-1})}
&\leq \int_{-\infty}^{0}|\varphi(y_N)|\|[\pd_N\CA f](\cdot,y_N)\|_{L_q(\BR^{N-1})}\intd y_N \\
&\leq \|\varphi\|_{L_{q'}(\BR_-)}\|\pd_N\CA f\|_{L_q(\BR_-^N)}
\end{align*}
for $q'=q/(q-1)$, which combined with \eqref{A-ineq-1}, furnishes
\begin{equation*}
\|I_2(\cdot,x_N)\|_{L_q(\BR^{N-1})}
\leq C\|f\|_{W_q^{1-1/q}(\BR^{N-1})} \quad \text{for $x_N\in(-L,0)$.}
\end{equation*}
This gives us
\begin{equation*}
\|I_2\|_{L_q(\BR^{N-1}\times(-L,0))}\leq C_L\|f\|_{W_q^{1-1/q}(\BR^{N-1})},
\end{equation*}
which, combined with \eqref{est-I1-A}, furnishes
\begin{equation*}
\|\CA f\|_{L_q(\BR^{N-1}\times(-L,0))}\leq C\|f\|_{W_q^{1-1/q}(\BR^{N-1})}.
\end{equation*}
This inequality together with \eqref{A-ineq-1} yields \eqref{A-ineq-2}.

(3) Let us first prove the right inequality.
Choose an extension $F\in H_q^m(\BR_-^N)$ such that
$F=f$ on $\BR_0^N$
and $\|F\|_{H_q^m(\BR_-^N)}\leq C\|f\|_{W_q^{m-1/q}(\BR^{N-1})}$.
There holds 
\begin{align*}
e^{\sqrt{1+|\xi'|^2}x_N}\wht f(\xi')
&=\int_{-\infty}^0 \frac{\pd}{\pd y_N} \Big(e^{\sqrt{1+|\xi'|^2}(x_N+y_N)}\wht F(\xi',y_N)\Big)\intd y_N \\
&=\int_{-\infty}^0 \sqrt{1+|\xi'|^2} e^{\sqrt{1+|\xi'|^2}(x_N+y_N)}\wht F(\xi',y_N)\intd y_N \\
&+\int_{-\infty}^0e^{\sqrt{1+|\xi'|^2}(x_N+y_N)}\wht{\pd_N F}(\xi',y_N)\intd y_N.
\end{align*}

Let $\Xi=(i\xi_1,\dots,i\xi_{N-1},\sqrt{1+|\xi|^2})$
and $\beta=(\beta_1,\dots,\beta_{N-1},\beta_N)\in\BN_0^N$ with $|\beta|\leq m$.
Then
\begin{align*}
\pd_x^\beta \CB f
&=\int_{-\infty}^0 \CF_{\xi'}^{-1}\Big[\Xi^\beta 
\sqrt{1+|\xi'|^2} e^{\sqrt{1+|\xi'|^2}(x_N+y_N)}\wht F(\xi',y_N)\Big](x')\intd y_N \\
&+\int_{-\infty}^0
\CF_{\xi'}^{-1}
\Big[\Xi^\beta e^{\sqrt{1+|\xi'|^2}(x_N+y_N)}\wht{\pd_N F}(\xi',y_N)\Big](x')\intd y_N \\
&=:I_3+I_4,
\end{align*}
where $\Xi^\beta=(i\xi_1)^{\beta_1}\dots(i\xi_{N-1})^{\beta_{N-1}}\sqrt{1+|\xi|^2}^{\,\beta_N}$.
The $I_3$ can be written as
\begin{align*}
I_3 &=
\int_{-\infty}^0 \CF_{\xi'}^{-1}\Big[\frac{\Xi^\beta}{(1+|\xi'|^2)^{m/2}} 
\sqrt{1+|\xi'|^2} e^{\sqrt{1+|\xi'|^2}(x_N+y_N)} \\
&\times (1+|\xi'|^2)^{m/2} \wht F(\xi',y_N)\Big](x')\intd y_N.
\end{align*}
For any multi-index $\alpha'\in\BN_0^{N-1}$,
we obtain from Lemma \ref{lem-sym1} and the Leibniz rule
\begin{equation}\label{eq:multi-1}
\bigg|\pd_{\xi'}^{\alpha'}\bigg(\frac{\Xi^\beta}{(1+|\xi'|^2)^{m/2}}\bigg)\bigg|
\leq C(1+|\xi'|)^{|\beta|-m-|\alpha'|}
\leq C(1+|\xi'|)^{-|\alpha'|}.
\end{equation}
Lemma \ref{lem:tech} thus yields
\begin{equation*}
\|I_3\|_{L_q(\BR_-^N)}
\leq C\|\CF_{\xi'}^{-1}[(1+|\xi'|^2)^{m/2} \wht F(\xi',\cdot )]\|_{L_q(\BR_-^N)}.
\end{equation*}
Since $\|\CF_{\xi'}^{-1}[(1+|\xi'|^2)^{m/2}\wht u(\xi')]\|_{L_q(\BR^{N-1})}$
is an equivalent norm of $\|u\|_{H_q^m(\BR^{N-1})}$,
one sees from the last inequality that
\begin{align*}
\|I_3\|_{L_q(\BR_-^N)}^q
&= C \int_{-\infty}^0\|\CF_{\xi'}^{-1}[(1+|\xi'|^2)^{m/2} \wht F(\xi',x_N)]\|_{L_q(\BR^{N-1})}^q\intd x_N \\
&\leq C \int_{-\infty}^0\|F(\cdot,x_N)\|_{H_q^{m}(\BR^{N-1})}^q\intd x_N \\
&\leq C\|F\|_{H_q^m(\BR_-^N)}^q
\leq C\|f\|_{W_q^{m-1/q}(\BR^{N-1})}^q.
\end{align*}

We next consider $I_4$. Let us write $I_4$ as follows:
\begin{align*}
I_4 &=\int_{-\infty}^0
\CF_{\xi'}^{-1}\Big[\frac{\Xi^\beta}{(1+|\xi'|^2)^{m/2}}
\sqrt{1+|\xi'|^2}e^{\sqrt{1+|\xi'|^2}(x_N+y_N)} \\
&\times (1+|\xi'|^2)^{(m-1)/2}\wht{\pd_N F}(\xi',y_N)\Big](x')\intd y_N.
\end{align*}
By \eqref{eq:multi-1} and Lemma \ref{lem:tech} 
\begin{equation*}
\|I_4\|_{L_q(\BR_-^N)}
\leq C\|\CF_{\xi'}^{-1}[(1+|\xi'|^2)^{(m-1)/2}\wht{\pd_N F}(\xi',\cdot)]\|_{L_q(\BR_-^N)},
\end{equation*}
and thus the equivalence of the norm
$\|\CF_{\xi'}^{-1}[(1+|\xi'|^2)^{(m-1)/2}\wht u(\xi')]\|_{L_q(\BR^{N-1})}$
with $\|u\|_{H_q^{m-1}(\BR^{N-1})}$ yields
\begin{align*}
\|I_4\|_{L_q(\BR_-^N)}^q
&= C \int_{-\infty}^0\|\CF_{\xi'}^{-1}[(1+|\xi'|^2)^{(m-1)/2} \wht{\pd_N F}(\xi',x_N)]\|_{L_q(\BR^{N-1})}^q\intd x_N \\
&\leq C \int_{-\infty}^0\|\pd_N F(\cdot,x_N)\|_{H_q^{m-1}(\BR^{N-1})}^q\intd x_N \\
&\leq C\|F\|_{H_q^m(\BR_-^N)}^q 
\leq C\|f\|_{W_q^{m-1/q}(\BR^{N-1})}^q.
\end{align*}

Summing up the estimates of $I_3$ and $I_4$ above,
we have obtained the right inequality $\|\CB f\|_{H_q^m(\BR_-^N)}\leq C\|f\|_{W_q^{m-1/q}(\BR^{N-1})}$.

The left inequality follows from $\CB f = f$ on $\BR_0^N$
and  the trace theorem  immediately.
This completes the proof of Lemma \ref{lem:ext}.
\end{proof}

Lemma \ref{lem:ext} can be extended to the following lemma.

\begin{lem}\label{lem:AB-bdd-lin}
Let $p,q\in(1,\infty)$. Then the following assertions hold.
\begin{enumerate}[$(1)$]
\item\label{lem:AB-bdd-lin-1}
Let $f\in B_{q,p}^{3-1/p-1/q}(\BR^{N-1})$.
Then $\CA f\in B_{q,p}^{3-1/p}(\BR^{N-1}\times(-L,0))$ for any $L>0$ and
\begin{equation*}
\|\CA f\|_{B_{q,p}^{3-1/p}(\BR^{N-1}\times(-L,0))}
\leq C_L\|f\|_{B_{q,p}^{3-1/p-1/q}(\BR^{N-1})}, 
\end{equation*}
where $C_L$ is a positive constant depending on $L$, but independent of $f$.
In addition, $\nabla \CA f\in B_{q,p}^{2-1/p}(\BR_-^{N})$ with
\begin{align*}
\|\nabla \CA f\|_{B_{q,p}^{2-1/p}(\BR_-^{N})}
&\leq C\|f\|_{B_{q,p}^{3-1/p-1/q}(\BR^{N-1})}, 
\end{align*}
where $C$ is a positive constant independent of and $f$.
\item\label{lem:AB-bdd-lin-2}
For any $f\in B_{q,p}^{3-1/p-1/q}(\BR^{N-1})$, $\CB f\in B_{q,p}^{3-1/p}(\BR_-^{N})$ with
\begin{equation*}
\|\CB f\|_{B_{q,p}^{3-1/p}(\BR_-^{N})}
\leq C\|f\|_{B_{q,p}^{3-1/p-1/q}(\BR^{N-1})}, 
\end{equation*}
where $C$ is a positive constant independent of $f$.
\item\label{lem:AB-bdd-lin-3}
Let $f\in H_p^1(\BR_+,W_q^{2-1/q}(\BR^{N-1}))\cap L_p(\BR_+,W_q^{3-1/q}(\BR^{N-1}))$.
Then
\begin{equation*}
\CA f\in H_p^1(\BR_+,H_q^2(\BR^{N-1}\times(-L,0)))\cap L_p(\BR_+,H_q^3(\BR^{N-1}\times(-L,0)))
\end{equation*}
for any $L>0$.
\item\label{lem:AB-bdd-lin-4}
Let $f\in H_p^1(\BR_+,W_q^{2-1/q}(\BR^{N-1}))\cap L_p(\BR_+,W_q^{3-1/q}(\BR^{N-1}))$. Then
\begin{equation*}
\CB f\in H_p^1(\BR_+,H_q^{2}(\BR_-^{N}))\cap L_p(\BR_+,H_q^{3}(\BR_-^{N})).
\end{equation*}
\end{enumerate}
\end{lem}

\begin{proof}
(1)
Let $G=\BR_-^N$ or $G=\BR^{N-1}\times(-L,0)$.
Then 
\begin{align}
(H_q^2(G),H_q^3(G))_{1-1/p,p} &= B_{q,p}^{3-1/p}(G), \label{int-sob-bes-1} \\
(H_q^1(G),H_q^2(G))_{1-1/p,p} &= B_{q,p}^{2-1/p}(G), \label{int-sob-bes-2}
\end{align}
see \cite[Subsection 7.32]{AF03}.
It follows from Lemma \ref{lem:ext} \eqref{lem:ext-3} that
\begin{align*}
\CA &\in \CL(W_q^{2-1/q}(\BR^{N-1}),H_q^2(\BR^{N-1}\times(-L,0))) \\
&\cap \CL(W_q^{3-1/q}(\BR^{N-1}),H_q^3(\BR^{N-1}\times(-L,0)))
\end{align*}
and
\begin{align*}
\nabla \CA &\in \CL(W_q^{2-1/q}(\BR^{N-1}),H_q^1(\BR_-^{N})) \cap
\CL(W_q^{3-1/q}(\BR^{N-1}),H_q^2(\BR_-^{N})).
\end{align*}
Combining these with \eqref{int-p-SS-sp-1}, \eqref{int-sob-bes-1}, and \eqref{int-sob-bes-2},
we obtain the desired properties from the real interpolation method.\footnote{
Here and subsequently, the real interpolation method
stands for \cite[Theorem 1.6]{Lunardi18}.}

(2) We can prove the desired property from 
\eqref{int-sob-bes-1} and Lemma \ref{lem:ext} \eqref{lem:ext-2} in the same manner as in \eqref{lem:AB-bdd-lin-1},
so that we may omit the detailed proof.

(3), (4)
The desired properties follow from Lemma \ref{lem:ext} immediately.
This completes the proof of Lemma \ref{lem:AB-bdd-lin}.
\end{proof}

We next consider the smoothness of $\CA f$ and $\CB f$ in the classical sense.

\begin{lem}\label{lem:classic-embed-eta}
Let $p\in(1,\infty)$ and $q\in(N,\infty)$ with $1/p+N/q<1$.
Then the following assertions hold.
\begin{enumerate}[$(1)$]
\item
Let $f\in B_{q,p}^{3-1/p-1/q}(\BR^{N-1})$.
Then 
\begin{equation*}
\CA f \in C_B^2(\overline{\BR^{N-1}\times(-L,0)}) \quad \text{for any $L>0$}
\end{equation*}
with
\begin{equation*}
\|\nabla\CA f\|_{H_\infty^1(\BR_-^N)} \leq C_1\|f\|_{B_{q,p}^{3-1/p-1/q}(\BR^{N-1})},
\end{equation*}
where $C_1$ is a positive constant independent of $f$.
In particular, $\CA f\in C^2(\overline{\BR_-^N})$.
\item
Let $f\in B_{q,p}^{3-1/p-1/q}(\BR^{N-1})$.
Then $\CB f \in C_B^2(\overline{\BR_-^N})$ with
\begin{equation*}
\|\CB f\|_{H_\infty^2(\BR_-^N)} \leq C_2\|f\|_{B_{q,p}^{3-1/p-1/q}(\BR^{N-1})},
\end{equation*}
where $C_2$ is a positive constant independent of $f$.
\item
Let $f\in H_p^1(\BR_+,W_q^{2-1/q}(\BR_-^N))\cap L_p(\BR_+,W_q^{3-1/q}(\BR_-^N))$.
Then
\begin{equation*}
\CA f\in C([0,\infty),C_B^2(\overline{\BR^{N-1}\times(-L,0)})) \quad \text{for any $L>0$.}
\end{equation*}
In particular,
$[\CA f](\cdot ,t) \in C^2(\overline{\BR_-^N})$ for each $t\geq 0$.
\item
Let $f\in H_p^1(\BR_+,W_q^{2-1/q}(\BR_-^N))\cap L_p(\BR_+,W_q^{3-1/q}(\BR_-^N))$.
Then
\begin{equation*}
\CB f\in C([0,\infty),C_B^2(\overline{\BR_-^N})).
\end{equation*}
\end{enumerate}
\end{lem}

\begin{proof}
(1), (2)
The desired properties follow from Lemma \ref{lem:AB-bdd-lin} \eqref{lem:AB-bdd-lin-1}, \eqref{lem:AB-bdd-lin-2}
and  Lemma \ref{lem:imbed-bexov-conti} \eqref{lem:imbed-bexov-conti-2} immediately.

(3), (4)
Let $G=\BR_-^N$ or $G=\BR^{N-1}\times(-L,0)$ for $L>0$.
Lemma \ref{lem:tanabe} and \eqref{int-sob-bes-1} yield
\begin{equation*}
H_p^1(\BR_+, H_q^2(G))\cap L_p(\BR_+,H_q^3(G))
\subset C([0,\infty),B_{q,p}^{3-1/p}(G)),
\end{equation*}
which, combined with Lemma \ref{lem:imbed-bexov-conti} \eqref{lem:imbed-bexov-conti-2}, furnishes
\begin{equation*}
H_p^1(\BR_+, H_q^2(G))\cap L_p(\BR_+,H_q^3(G))
\subset C([0,\infty),C_B^2(\overline{G})).
\end{equation*}
We thus obtain the desired properties from 
Lemma \ref{lem:AB-bdd-lin} \eqref{lem:AB-bdd-lin-3}, \eqref{lem:AB-bdd-lin-4}.
This completes the proof of Lemma \ref{lem:classic-embed-eta}.
\end{proof}

Furthermore, we have

\begin{lem}\label{lem:extra-embed}
Let $p\in(2,\infty)$ and $q\in(N,\infty)$ with $2/p+N/q<1$.
Then the following assertions hold.
\begin{enumerate}[$(1)$]
\item\label{lem:extra-embed-1}
Let $f$ and $g$ satisfy
\begin{align*}
f&\in H_p^1(\BR_+,L_q(\BR_-^N))\cap L_p(\BR_+,H_q^2(\BR_-^N)), \\
g&\in H_p^1(\BR_+,W_q^{2-1/q}(\BR^{N-1}))\cap L_p(\BR_+,W_q^{3-1/q}(\BR^{N-1})).
\end{align*}
Then for $\pd_j=\pd/\pd x_j$ and $j=1,\dots,N-1$
\begin{equation*}
f\pd_j g \in C([0,\infty),B_{q,p}^{2-2/p-1/q}(\BR_0^N)).
\end{equation*}
\item\label{lem:extra-embed-2}
Let $f\in C([0,\infty),B_{q,p}^{2-2/p-1/q}(\BR^{N-1}))$. Then
\begin{equation*}
\CA f \in C([0,\infty), C_B^1(\overline{\BR^{N-1}\times(-L,0)}))
\quad \text{for any $L>0$.}
\end{equation*}
In particular, $[\CA f](\cdot ,t)\in C^1(\overline{\BR_-^N})$ for each $t\geq 0$.
\item\label{lem:extra-embed-3}
Let $f\in C([0,\infty),B_{q,p}^{2-2/p-1/q}(\BR^{N-1}))$. Then
$\CB f\in C([0,\infty),C_B^1(\overline{\BR_-^N}))$.
\end{enumerate}
\end{lem}

\begin{proof}
\eqref{lem:extra-embed-1}
Lemma \ref{lem:time-sp} \eqref{lem:time-sp-2} and the trace theorem yield
\begin{equation}\label{f-trace-est}
f\in C([0,\infty),B_{q,p}^{2-2/p-1/q}(\BR_0^N)),
\end{equation}
while Lemma \ref{lem:time-sp} \eqref{lem:time-sp-1} gives us
\begin{equation*}
\pd_j g \in C([0,\infty),B_{q,p}^{2-1/p-1/q}(\BR^{N-1})).
\end{equation*}
We combine the last property with the fact that 
$B_{q,p}^{2-1/p-1/q}(\BR^{N-1})$ is continuously embeded into $B_{q,p}^{2-2/p-1/q}(\BR^{N-1})$,
and then we have
\begin{equation}\label{g-est-rough}
\pd_j g \in C([0,\infty),B_{q,p}^{2-2/p-1/q}(\BR^{N-1})).
\end{equation}
Notice that $B_{q,p}^{2-2/p-1/q}(\BR^{N-1})$ is a multiplication algebra
by the assumption $2/p+N/q<1$, see Theorem 1 in Subsection 4.6.4 of \cite{RS96}.
The desired property thus follows from \eqref{f-trace-est} and \eqref{g-est-rough} immediately.

\eqref{lem:extra-embed-2}, \eqref{lem:extra-embed-3}
Let $G=\BR_-^N$ or $G=\BR^{N-1}\times(-L,0)$ for $L>0$.
It follows from \cite[Subsection 2.4.2]{Triebel78} and \cite[Subsection 7.32]{AF03} that
\begin{align}\label{add-int-prp}
(W_q^{1-1/q}(\BR^{N-1}),W_q^{2-1/q}(\BR^{N-1}))_{1-2/p,p}
&=B_{q,p}^{2-2/p-1/q}(\BR^{N-1}), \notag \\
(H_q^1(G),H_q^2(G))_{1-2/p,p}&=B_{q,p}^{2-2/p}(G).
\end{align}
By Lemma \ref{lem:ext} \eqref{lem:ext-3}, \eqref{lem:ext-2}
\begin{align*}
\CA
&\in\CL(W_q^{1-1/q}(\BR^{N-1}),H_q^1(\BR^{N-1}\times(-L,0))) \\
&\cap \CL(W_q^{2-1/q}(\BR^{N-1}),H_q^2(\BR^{N-1}\times(-L,0)))
\end{align*}
and
\begin{equation*}
\CB\in\CL(W_q^{1-1/q}(\BR^{N-1}),H_q^1(\BR_-^{N}))
\cap  \CL(W_q^{2-1/q}(\BR^{N-1}),H_q^2(\BR_-^{N}).
\end{equation*}
Combining these with \eqref{add-int-prp} and the real interpolation method furnishes that
for any $f\in B_{q,p}^{2-2/p-1/q}(\BR^{N-1})$ 
\begin{align*}
&\CA f \in B_{q,p}^{2-2/p}(\BR^{N-1}\times (-L,0)) \text{ for any $L>0$ with} \\
&\|\CA f\|_{B_{q,p}^{2-2/p}(\BR^{N-1}\times (-L,0))}
\leq C_L \|f\|_{B_{q,p}^{2-2/p-1/q}(\BR^{N-1})}
\end{align*}
and 
\begin{align*}
\CB f \in B_{q,p}^{2-2/p}(\BR_-^{N}) \text{ with }
\|\CB f\|_{B_{q,p}^{2-2/p}(\BR_-^{N})}\leq C\|f\|_{B_{q,p}^{2-2/p-1/q}(\BR^{N-1})},
\end{align*}
where $C_L$ and $C$ are positive constants and $C_L$ depends on $L$.
These properties together with Lemma \ref{lem:imbed-bexov-conti} \eqref{lem:imbed-bexov-conti-3}
yield the desired properties of \eqref{lem:extra-embed-2} and \eqref{lem:extra-embed-3} immediately.
This completes the proof of Lemma \ref{lem:extra-embed}.
\end{proof}

We finally estimate $\|\CA f\|_{L_2(\BR_-^N)}$. 

\begin{lem}\label{lem:ext-interp}
Suppose  $N\geq 3$ and $s\in(1,4/3)$.
Then, for any $f\in L_s(\BR^{N-1})\cap L_2(\BR^{N-1})$,
$\CA f\in L_2(\BR_-^N)$ with
\begin{equation*}
\|\CA f\|_{L_2(\BR_-^N)}\leq C\Big(\|f\|_{L_{s}(\BR^{N-1})}+\|f\|_{L_2(\BR^{N-1})}\Big),
\end{equation*}
where $C$ is a positive constant independent of $f$.
\end{lem}

\begin{proof}
Parseval's identity yields
\begin{equation*}
\|[\CA f](\cdot,x_N)\|_{L_2(\BR^{N-1})}^2=(2\pi)^{-(N-1)}\|e^{|\xi'|x_N}\wht f(\xi')\|_{L_2(\BR^{N-1})}^2,
\end{equation*}
which implies 
\begin{align*}
\int_{-\infty}^0 \|[\CA f](\cdot,x_N)\|_{L_2(\BR^{N-1})}^2\intd x_N
&=(2\pi)^{-(N-1)}\int_{\BR^{N-1}}\frac{|\wht f(\xi')|^2}{2|\xi'|}\,d\xi' \\
&=(2\pi)^{-(N-1)}\Big(\int_{|\xi'|\leq 1}+\int_{|\xi'|\geq 1}\Big) \frac{|\wht f(\xi')|^2}{2|\xi'|}\intd \xi \\
&=:I_1+I_2.
\end{align*}

We first consider $I_1$. Let $s'=s/(s-1)$ and notice $s'>4$.
We set $p=q=s'$ and $1/r=1-1/p-1/q$.
Since $r\in(1,2)$ and $N\geq 3$, it holds that
\begin{align*}
I_1 
&\leq C
\Big(\int_{|\xi'|\leq 1 } |\wht f(\xi')|^p\intd \xi'\Big)^{1/p}
\Big(\int_{|\xi'|\leq 1 } |\wht f(\xi')|^q\intd \xi'\Big)^{1/q}
\Big(\int_{|\xi'|\leq 1 } \frac{d\xi'}{|\xi'|^r}\Big)^{1/r} \\
&\leq C\|\wht f\|_{L_{s'}(\BR^{N-1})}^2.
\end{align*}
Combining this with Hausdorff-Young's inequality $\|\wht f\|_{L_{s'}(\BR^{N-1})}\leq C\|f\|_{L_s(\BR^{N-1})}$
furnishes
\begin{equation}\label{ineq:ext-I1}
I_1 \leq C\|f\|_{L_s(\BR^{N-1})}^2.
\end{equation}

We next consider $I_2$. There holds
\begin{equation*}
I_2 \leq C\int_{|\xi'|\geq 1}|\wht f(\xi')|^2\intd \xi' \leq C\int_{\BR^{N-1}}|\wht f(\xi')|^2\intd \xi',
\end{equation*}
which, combined with Parseval's identity, furnishes
\begin{equation*}
I_2\leq C\|f\|_{L_2(\BR^{N-1})}^2.
\end{equation*}
This inequality and \eqref{ineq:ext-I1} give us the desired inequality.
This completes the proof of Lemma \ref{lem:ext-interp}.
\end{proof}

\subsection{Properties of $\wht H_{q,0}^1(\BR_-^N)$}
Recall \eqref{dfn:hat-0} and define
\begin{equation*}
H_{q,0}^1(\BR_-^N)=\{f\in H_{q}^1(\BR_-^N) : f=0 \text{ on $\BR_0^N$}\}
\end{equation*}
for $q\in(1,\infty)$. The following lemma then holds.

\begin{lem}\label{lem:dense}
Let $q\in(1,\infty)$. Then the following assertions hold.
\begin{enumerate}[$(1)$]
\item\label{lem:dense-1}
$H_{q,0}^1(\BR_-^N)$ is dense in $\wht H_{q,0}^1(\BR_-^N)$, i.e.,
for any $\varphi\in \wht H_{q,0}^1(\BR\-^N)$ there exists a sequence $\{\varphi_j\}_{j=1}^\infty\subset H_{q,0}^1(\BR_-^N)$
such that $\lim_{j\to \infty}\|\nabla(\varphi_j-\varphi)\|_{L_q(\BR_-^N)}=0$.
\item\label{lem:dense-2}
$C_0^\infty(\BR_-^N)$ is dense in $\wht H_{q,0}^1(\BR\-^N)$, 
i.e.,
for any $\varphi\in \wht H_{q,0}^1(\BR\-^N)$ there exists a sequence $\{\varphi_j\}_{j=1}^\infty\subset C_0^\infty(\BR_-^N)$
such that $\lim_{j\to \infty}\|\nabla(\varphi_j-\varphi)\|_{L_q(\BR_-^N)}=0$.
\end{enumerate}
\end{lem}

\begin{proof}
(1) See \cite[Theorem A.3]{Shibata13}. 

(2) Combining \eqref{lem:dense-1} with the fact that $C_0^\infty(\BR_-^N)$ is dense in $H_{q,0}^1(\BR_-^N)$
shows the desired property.
This completes the proof of Lemma \ref{lem:dense}.
\end{proof}

\section{An estimate for Duhamel's integral}\label{sec:duhamel}

This section proves the following proposition, which plays a crucial role in establishing
our linear theory in Section \ref{sec:lin-theory} below.

\begin{prp}\label{prp:duha}
Let $X$ be a Banach space and $(T(t))_{t\geq 0}$ be a 
(not necessarily bounded)
$C_0$-analytic semigroup on $X$.
Let $A$ be the infinitesimal generator  of $(T(t))_{t\geq 0}$ and $D(A)$ the domain of $A$ with the graph norm
$\|x\|_{D(A)}=\|x\|_X+\|Ax\|_X$.
Suppose that the following conditions $({\rm a})$, $({\rm b})$, and $({\rm c})$ hold.
\begin{enumerate}[{\rm (a)}]
\item
Let $\delta >0$ and $s_0\in [1,\infty)$ such that $s_0\delta>1$.
\item
There exists a positive constant $M$ such that 
\begin{equation}\label{assum:duha-2}
\|T(t)x\|_{D(A)}\leq M  t^{-a-\delta}\|x\|_Y \quad (t\geq 1)
\end{equation}
for an $a\in(0,1]$ and a Banach space $Y\subset X$. 
\item
Let $p\in [s_0,\infty)$ and $\langle t \rangle f\in L_p(\BR_+,Y\cap D(A))$ with $t=\sqrt{1+t^2}$.
\end{enumerate}
Then 
\begin{equation*}
u(t)=\int_0^t T(t-\tau)f(\tau)\intd \tau \quad (t>0)
\end{equation*}
satisfies
\begin{equation}\label{est-duhamel}
\|\langle t \rangle^a \pd_t u\|_{L_p((2,\infty),X)}+\|\langle t \rangle^a u\|_{L_p((2,\infty),D(A))}
\leq C\|\langle t \rangle f\|_{L_p(\BR_+,Y\cap D(A))}
\end{equation}
for some positive constant $C$.
\end{prp}

\begin{proof}
We consider the case $p>1$ only in this proof.
The case $p=1$ can be proved similarly to the case $p>1$.\footnote{
In the case $p=1$, $s_0=1$ and $\delta>1$.
}

Let $C_0^\infty(\BR_+,Y\cap D(A))$
be the set of all $C^\infty$ functions with values in $Y\cap D(A)$
whose supports are compact and contained in $\BR_+$.
Since there exists $\{f_j\}_{j=1}^\infty \subset C_0^\infty(\BR_+,Y\cap D(A))$ such that
\begin{equation*}
\lim_{j\to \infty}\|\langle t\rangle(f_j-f)\|_{L_p(\BR_+,Y\cap D(A))}=0,
\end{equation*}
it suffices to show \eqref{est-duhamel} for $f\in C_0^\infty(\BR_+,Y\cap D(A))$ in what follows.

Let us write $u(t)=[\CU f](t)$ and let $t\geq 2$.  Then
\begin{align*}
[\CU f](t)
&=\Big(\int_0^{t/2}+\int_{t/2}^{t-1}+\int_{t-1}^t\Big)T(t-\tau)f(\tau)\intd \tau \\
&=:I_1(t)+I_2(t)+I_3(t).
\end{align*}

{\bf Step 1}: Estimate $I_1(t)$.
By \eqref{assum:duha-2}, we see that
\begin{align*}
\|I_1(t)\|_{D(A)}
&\leq C\int_0^{t/2}( t-\tau)^{-a-\delta}\|f(\tau)\|_Y\intd \tau \\
&\leq C t ^{-a-\delta}\int_0^{t/2} \|f(\tau)\|_Y\intd\tau \\
&\leq C t ^{-a-\delta}
\bigg(\int_0^\infty\langle \tau\rangle^{-p'}\intd\tau\bigg)^{1/p'}
\bigg(\int_0^\infty\big(\langle\tau\rangle\|f(\tau)\|_Y\big)^p\intd\tau\bigg)^{1/p},
\end{align*}
where $p'=p/(p-1)$. Thus
\begin{equation}\label{duha:eq-1}
\langle t \rangle^{a}\|I_1(t)\|_{D(A)}
\leq C\langle t \rangle^a t^{-a-\delta}\|\langle t \rangle f\|_{L_p(\BR_+,Y)}.
\end{equation}
Since $\langle t\rangle^a t^{-a-\delta}\leq C t^{-\delta}$ for $t\geq 2$,
it holds by the assumption $s_0\delta>1$ that 
\begin{equation*}
\int_2^\infty  (\langle t\rangle^a t^{-a-\delta})^p\intd t
\leq C \int_2^\infty t^{-p\delta}\intd t
\leq \int_2^\infty  t ^{- s_0\delta}\intd t\leq C.
\end{equation*}
Combining this inequality with \eqref{duha:eq-1} furnishes
\begin{equation}\label{duha:eq-2}
\|\langle t \rangle^{a} I_1\|_{L_p((2,\infty),D(A))}
\leq C\|\langle t \rangle f\|_{L_p(\BR_+,Y)}.
\end{equation}

{\bf Step 2}: Estimate $I_2(t)$.
By \eqref{assum:duha-2}, we see that
\begin{align*}
\|I_2(t)\|_{D(A)}
&\leq C\int_{t/2}^{t-1} (t-\tau)^{-a-\delta}\|f(\tau)\|_Y\intd \tau \\
&\leq C\langle t \rangle^{-1}\int_{t/2}^{t-1} ( t-\tau)^{-a-\delta} \langle \tau\rangle \|f(\tau)\|_Y\intd\tau.
\end{align*}
Since $(t -\tau)^{-a-\delta}=(t -\tau)^{-(a+\delta)(1/p+1/p')}$, 
\begin{align*}
\|I_2(t)\|_{D(A)}&\leq C\langle t \rangle^{-1}
\bigg(\int_{t/2}^{t-1} (t-\tau)^{-(a+\delta)}\intd\tau\bigg)^{1/p'} \\
&\times
\bigg(\int_{t/2}^{t-1}(t-\tau)^{-(a+\delta)}\big(\langle \tau\rangle \|f(\tau)\|_Y\big)^p\intd\tau\bigg)^{1/p}.
\end{align*}
Combining this inequality with
\begin{equation*}
\int_{t/2}^{t-1} (t-\tau)^{-(a+\delta)}\intd\tau\leq CJ(t), \quad 
J(t)=\left\{\begin{aligned}
&t^{1-(a+\delta)} &&(a+\delta<1), \\
&\log t &&(a+\delta =1), \\
&1 && (a+\delta>1),
\end{aligned}\right.
\end{equation*}
furnishes
\begin{align*}
&\langle t \rangle^a\|I_2(t)\|_{D(A)} \\
&\leq C\langle t \rangle^{-(1-a)} J(t)^{1/p'} 
\bigg(\int_{t/2}^{t-1}(t-\tau)^{-(a+\delta)}\big(\langle \tau\rangle \|f(\tau)\|_Y\big)^p\intd\tau\bigg)^{1/p}.
\end{align*}
It thus follows from $J(t)^{1/p'}=J(t)^{1-1/p}$ that 
\begin{align*}
&\|\langle t \rangle^a I_2\|_{L_p((2,\infty),D(A))}^p \\
&\leq C\int_2^\infty \langle t\rangle^{-p(1-a)}J(t)^{p-1}
\bigg(\int_{t/2}^{t-1} (t-\tau)^{-(a+\delta)}\big(\langle \tau\rangle \|f(\tau)\|_Y\big)^p\intd\tau\bigg) dt \\
&=:\text{(RHS)},
\end{align*}
which, combined with Fubini's theorem, yields
\begin{equation*}
\text{(RHS)}=C\int_1^\infty\big(\langle \tau\rangle \|f(\tau)\|_Y\big)^p
\bigg(\int_{\tau+1}^{2\tau} \langle t\rangle^{-p(1-a)}J(t)^{p-1}( t-\tau)^{-(a+\delta)}\intd t\bigg)
\intd\tau.
\end{equation*}

Let us define
\begin{align*}
K(\tau)=\int_{\tau+1}^{2\tau} \langle t\rangle^{-p(1-a)}J(t)^{p-1}( t-\tau)^{-(a+\delta)}\intd t \quad (\tau\geq 1),
\end{align*}
and then
\begin{equation}\label{K-form-est}
\|\langle t \rangle^a I_2\|_{L_p((2,\infty),D(A))}^p
\leq C\int_1^\infty\big(\langle \tau\rangle \|f(\tau)\|_Y\big)^p K(\tau)\intd\tau.
\end{equation}
From this viewpoint, we estimate $K(\tau)$ in what follows.

We first consider the case $a+\delta<1$. In this case, it holds for $t\in[\tau+1,2\tau]$ that
\begin{align*}
\langle t\rangle^{-p(1-a)}J(t)^{p-1}( t-\tau)^{-(a+\delta)}
&\leq t^{-p\delta-\{1-(a+\delta)\}}(t-\tau)^{-(a+\delta)} \\
&\leq (t-\tau)^{-p\delta-1},
\end{align*}
and thus
\begin{equation*}
K(\tau)\leq \int_{\tau+1}^{2\tau}(t-\tau)^{-p\delta-1}\intd t \leq C.
\end{equation*}

We next consider the case $a+\delta=1$. In this case,
\begin{equation*}
K(\tau)=\int_{\tau+1}^{2\tau}\langle t\rangle^{-p\delta}(\log t)^{p-1}(t-\tau)^{-1}\intd t,
\end{equation*}
and also
\begin{equation*}
\langle t\rangle^{-p\delta}(\log t)^{p-1}
\leq \langle t\rangle^{-p\delta/2}\sup_{t\geq 2}\Big(\langle t\rangle^{-p\delta/2}(\log t)^{p-1}\Big).
\end{equation*}
By these properties, we see that
\begin{equation*}
K(\tau)\leq C\int_{\tau+1}^{2\tau}\langle t\rangle^{-p\delta/2}(t-\tau)^{-1}\intd t
\leq C\int_{\tau+1}^{2\tau}(t-\tau)^{-p\delta/2-1}\intd t\leq C.
\end{equation*}

Finally, we consider the case $a+\delta>1$. In this case,
\begin{equation*}
K(\tau)=\int_{\tau+1}^{2\tau}\langle t\rangle^{-p(1-a)}(t-\tau)^{-(a+\delta)}\intd t
\leq \int_{\tau+1}^{2\tau}(t-\tau)^{-(a+\delta)}\intd t\leq C.
\end{equation*}

Summing up the estimates of $K(\tau)$ above,
we have obtained from \eqref{K-form-est}
\begin{equation}\label{duha:eq-3}
\|\langle t \rangle^a I_2\|_{L_p((2,\infty),D(A))}
\leq C\|\langle t \rangle f\|_{L_p(\BR_+,Y)}.
\end{equation}

{\bf Step 3}: Estimate $I_3(t)$.
Notice the following property of analytic $C_0$-semigroups:
\begin{equation*}
\|T(t)x\|_{D(A)}\leq Ce^{\gamma t}\|x\|_{D(A)} \quad (t> 0)
\end{equation*}
for some positive constant $\gamma$. Then
\begin{align*}
\|I_3(t)\|_{D(A)}\leq C\int_{t-1}^t e^{\gamma (t-\tau)}\|f(\tau)\|_{D(A)}\intd \tau
\leq Ce^{\gamma}\int_{t-1}^t \|f(\tau)\|_{D(A)}\intd \tau.
\end{align*}
Since $\langle t \rangle \leq C\langle \tau\rangle $ for $\tau\in[t-1,t]$, it follows from the last inequality that
\begin{align*}
\langle t \rangle ^a \|I_3(t)\|_{D(A)}
&\leq C\int_{t-1}^t\langle \tau\rangle^a\|f(\tau)\|_{D(A)}\intd \tau 
\leq C\int_{t-1}^t\langle \tau\rangle\|f(\tau)\|_{D(A)}\intd \tau \\
&\leq C\bigg(\int_{t-1}^t\intd\tau\bigg)^{1/p'}
\bigg(\int_{t-1}^t \big(\langle \tau\rangle\|f(\tau)\|_{D(A)}\big)^p\intd\tau\bigg)^{1/p}.
\end{align*}
Hence, 
\begin{align*}
\|\langle t\rangle^a I_3\|_{L_p((2,\infty),D(A))}^p
\leq C\int_2^\infty\int_{t-1}^t\big(\langle \tau\rangle\|f(\tau)\|_{D(A)}\big)^p\intd\tau dt.
\end{align*}
By Fubini's theorem
\begin{align*}
&\int_2^\infty\int_{t-1}^t\big(\langle \tau\rangle\|f(\tau)\|_{D(A)}\big)^p\intd\tau dt \\
&=\bigg(\int_{1}^2\int_2^{\tau+1}+\int_{2}^\infty\int_{\tau}^{\tau+1}\bigg)
\big(\langle \tau\rangle\|f(\tau)\|_{D(A)}\big)^p\intd t d\tau \\
&=\int_1^2(\tau-1)\big(\langle \tau\rangle\|f(\tau)\|_{D(A)}\big)^p\intd\tau
+\int_2^\infty\big(\langle \tau\rangle\|f(\tau)\|_{D(A)}\big)^p\intd\tau,
\end{align*}
and thus
\begin{equation*}
\|\langle t\rangle^a I_3\|_{L_p((2,\infty),D(A))}
\leq C\|\langle t\rangle f\|_{L_p(\BR_+,D(A))}.
\end{equation*}

Summing up the last inequality, \eqref{duha:eq-2}, and  \eqref{duha:eq-3}, we have obtained
\begin{equation}\label{duha:eq-5}
\|\langle t\rangle^a\CU f\|_{L_p((2,\infty),D(A))}
\leq C\|\langle t\rangle f\|_{L_p(\BR_+,Y\cap D(A))}.
\end{equation}

{\bf Step 4}: Estimate the time derivative of $\CU f$.
We observe that
\begin{align*}
\pd_t [\CU f](t)
&=f(t)+\int_0^t\pd_t T(t-\tau)f(\tau)\intd \tau \\
&=f(t)+A\int_0^t T(t-\tau)f(\tau)\intd \tau \\
&=f(t)+A [\CU  f](t).
\end{align*}
Then
\begin{equation*}
\|\pd_t[\CU f](t)\|_X \leq \|f(t)\|_X+\|A[\CU  f](t)\|_X 
\leq  \|f(t)\|_{D(A)}+\|[\CU f](t)\|_{D(A)},
\end{equation*}
which furnishes
\begin{align*}
\|\langle t\rangle^a\pd_t\,\CU f\|_{L_p((2,\infty),X)}
&\leq
\|\langle t\rangle^a f\|_{L_p((2,\infty),D(A))}
+\|\langle t\rangle^a \CU f\|_{L_p((2,\infty),D(A))}.
\end{align*}
Combining this inequality with \eqref{duha:eq-5} yields
\begin{equation*}
\|\langle t\rangle^a\pd_t\,\CU f\|_{L_p((2,\infty),X)}
\leq C\|\langle t\rangle f\|_{L_p(\BR_+,Y\cap D(A))}.
\end{equation*}
This completes the proof of Proposition \ref{prp:duha}.
\end{proof}

\section{A $C_0$-analytic semigroup}\label{sec:sg}

This section introduces a $C_0$-analytic semigroup
associated with the linearized system of \eqref{nonl:fix1} and its properties.
All of the results of Subsection \ref{subsec:sg} below hold for $N\geq 2$.

\subsection{Generation of the semigroup}\label{subsec:sg}
Let us start with the following resolvent problem independent of time $t$:
\begin{equation}\label{rp1}
\left\{\begin{aligned}
\lambda \eta -u_N&=d  && \text{on $\BR_0^N$,} \\
\lambda \Bu-\mu\Delta\Bu+\nabla\Fp &=\Bf  && \text{in $\BR_-^N$, } \\
\dv\Bu &= 0 && \text{in $\BR_-^N$, } \\
(\mu\BD(\Bu)-\Fp\BI)\Be_N+(c_g-c_\sigma\Delta')\eta\Be_N &=0  && \text{on $\BR_0^N$,} 
\end{aligned}\right.
\end{equation}
where $\lambda$ is the resolvent parameter varying in the sector 
\begin{align*}
\Sigma_{\omega,\gamma}=\Sigma_\omega+\gamma=\{\lambda+\gamma :\lambda \in \Sigma_\omega\}, \quad
\Sigma_{\omega}=\{\lambda\in\BC\setminus\{0\} : |\arg\lambda|<\pi-\omega \}
\end{align*}
for $\gamma\in\BR$ and $\omega\in(0,\pi/2)$.
By \cite[Theorem 1.3]{SS12}, we have

\begin{lem}\label{lem:SS12}
Let $\omega\in(0,\pi/2)$ and $q\in(1,\infty)$.
Then there exists a positive constant $\gamma_0=\gamma_0(\omega)\geq 1$ such that,
for any $\lambda\in\Sigma_{\omega,\gamma_0}$ and $(d,\Bf)\in W_q^{2-1/q}(\BR^{N-1})\times L_q(\BR_-^N)^N$,
\eqref{rp1} admits a unique solution
$(\eta,\Bu,\Fp)$ with
\begin{equation*}
(\eta,\Bu,\Fp)\in W_q^{3-1/q}(\BR^{N-1})\times H_q^2(\BR_-^N)^N\times \wht H_q^1(\BR_-^N).
\end{equation*}
In addition,
\begin{align*}
&|\lambda|\|(\eta,\Bu)\|_{W_q^{2-1/q}(\BR^{N-1})\times L_q(\BR_-^N)^N}
+\|(\eta,\Bu)\|_{W_q^{3-1/q}(\BR^{N-1})\times H_q^2(\BR_-^N)^N} \\
&\leq C\|(d,\Bf)\|_{W_q^{2-1/q}(\BR^{N-1})\times L_q(\BR_-^N)^N}
\end{align*}
for some positive constant $C=C(N,q,\omega,\gamma_0)$ independent of $\lambda$.
\end{lem}

Let us eliminate the pressure $\Fp$ from \eqref{rp1}
in order to get the semirgoup setting.
To this end, we recall \eqref{dfn:hat-0} and consider the following weak problem:
for a given $\Bf\in L_q(\BR_-^N)^N$ find $u\in\wht H_{q,0}^1(\BR_-^N)$ such that 
\begin{equation}\label{weak-p}
(\nabla u,\nabla\varphi)_{\BR_-^N}=(\Bf,\nabla\varphi)_{\BR_-^N}
\quad \text{for any $\varphi\in\wht H_{q',0}^1(\BR_-^N)$,}
\end{equation}
where $q\in(1,\infty)$ and $q'=q/(q-1)$.

The next lemma has been proved in \cite[Theorem 3.2.12]{Shibata20}.

\begin{lem}\label{lem:weak-D}
Let $q\in(1,\infty)$ and $\Bf\in L_q(\BR_-^N)^N$.
Then \eqref{weak-p} admits a unique solution $u\in\wht H_{q,0}^1(\BR_-^N)$,
which satisfies $\|\nabla u\|_{L_q(\BR_-^N)}\leq C\|\Bf\|_{L_q(\BR_-^N)}$ for some positive constant $C=C(N,q)$.
\end{lem}

Let $\Bf\in L_q(\BR_-^N)^N$ and $g\in W_q^{1-1/q}(\BR^{N-1})$.
Let $\wtd g\in H_q^1(\BR_-^N)$ be an extension of $g$ such that
$\wtd g = g$ on $\BR_0^N$ and
$\|\wtd g \|_{H_q^1(\BR_-^N)}\leq C\|g\|_{W_q^{1-1/q}(\BR^{N-1})}$.
Lemma \ref{lem:weak-D} then furnishes that 
there exists $\wtd u\in \wht H_{q,0}^1(\BR_-^N)$ such that
\begin{equation*}
(\nabla \wtd u,\nabla\varphi)_{\BR_-^N}=(\Bf-\nabla \wtd g,\nabla\varphi)_{\BR_-^N}
\quad \text{for any $\varphi\in\wht H_{q',0}^1(\BR_-^N)$.}
\end{equation*}
Setting $u=\wtd u+\wtd g \in \wht H_{q,0}^1(\BR_-^N)+H_q^1(\BR_-^N)$ yields 
\begin{equation}\label{weakD}
\left\{\begin{aligned}
&(\nabla u,\nabla\varphi)_{\BR_-^N} =(\Bf,\nabla\varphi)_{\BR_-^N}
&& \text{for any $\varphi\in \wht H_{q',0}^1(\BR_-^N)$,} \\
 &u=g \quad \text{on $\BR_0^N$.}
\end{aligned}\right.
\end{equation}

Let $(\eta,\Bu)\in W_q^{3-1/q}(\BR^{N-1})\times H_q^2(\BR_-^N)^N$.
We choose in \eqref{weakD}
\begin{equation*}
\Bf=\mu\Delta\Bu, \quad 
g=\Be_N\cdot(\mu\BD(\Bu)\Be_N)+(c_g-c_\sigma\Delta')\eta.
\end{equation*}
Then the mapping 
\begin{equation*}
\CK:W_q^{3-1/q}(\BR^{N-1})\times H_q^2(\BR_-^N)^N
\ni(\eta,\Bu)\mapsto \CK(\eta,\Bu)\in \wht H_{q,0}^1(\BR_-^N)+H_q^1(\BR_-^N)
\end{equation*}
can be defined as follows:\footnote{
This $\CK$ satisfies $\|\nabla \CK(\eta,\Bu)\|_{L_q(\BR_-^N)}
\leq C\|(\eta,\Bu)\|_{W_q^{3-1/q}(\BR^{N-1})\times H_q^2(\BR_-^N)^N }.$}
\begin{equation*}
\left\{\begin{aligned}
&(\nabla \CK(\eta,\Bu),\nabla\varphi)_{\BR_-^N}=(\mu\Delta\Bu,\nabla\varphi)_{\BR_-^N} 
\quad \text{for any $\varphi\in \wht H_{q',0}^1(\BR_-^N)$,} \\
&\CK(\eta,\Bu)=\Be_N\cdot (\mu\BD(\Bu)\Be_N)+(c_g-c_\sigma\Delta')\eta \quad \text{on $\BR_0^N$.}
\end{aligned}\right.
\end{equation*}

We now consider the so-called  reduced Stokes resolvent problem:
\begin{equation}\label{reduced:rp}
\left\{\begin{aligned}
\lambda \eta-u_N&= d && \text{on $\BR_0^N$,} \\
\lambda\Bu-\mu\Delta \Bu+\nabla\CK(\eta,\Bu)&=\Bf && \text{in $\BR_-^N$,} \\
(\mu\BD(\Bu)-\CK(\eta,\Bu)\BI)\Be_N+(c_g-c_\sigma\Delta')\eta\Be_N&=0 && \text{on $\BR_0^N$.} 
\end{aligned}\right.
\end{equation}
For the space $J_q(\BR_-^N)$ given by \eqref{dfn:solenoidal}, we define
\begin{align*}
X_q&=W_q^{2-1/q}(\BR^{N-1})\times J_q(\BR_-^N), \\
\|(\eta,\Bu)\|_{X_q}&=\|\eta\|_{W_q^{2-1/q}(\BR_-^N)}+\|\Bu\|_{L_q(\BR_-^N)}.
\end{align*}
The following lemma then holds.

\begin{lem}\label{lem:equiv}
Let $q\in(1,\infty)$, $(d,\Bf)\in X_q$, and $\lambda\in\BC\setminus\{0\}$.
Then \eqref{rp1} is equivalent to \eqref{reduced:rp}, i.e.,
the following assertions hold.
\begin{enumerate}[$(1)$]
\item
Assume that $(\eta,\Bu,\Fp)\in W_q^{3-1/q}(\BR^{N-1})\times H_q^2(\BR_-^N)^N\times \wht H_q^1(\BR_-^N)$
is a solution to \eqref{rp1}.
Then $(\eta,\Bu)$ solves \eqref{reduced:rp}.
\item
Assume that $(\eta,\Bu)\in W_q^{3-1/q}(\BR^{N-1})\times H_q^2(\BR_-^N)^N$ is a solution to \eqref{reduced:rp}.
Then $(\eta,\Bu)$ solves \eqref{rp1} with $\Fp=\CK(\eta,\Bu)$.
\end{enumerate}
\end{lem}

\begin{proof}
(1) It suffices to show $\Fp=\CK(\eta,\Bu)$.
Let $q'=q/(q-1)$ and $\varphi\in \wht H_{q',0}^1(\BR_-^N)$.
By Lemma \ref{lem:dense},
there exists a sequence $\{\varphi_j\}_{j=1}^\infty \subset C_0^\infty(\BR_-^N)$ such that
\begin{equation*}
\lim_{j\to \infty}\|\nabla(\varphi_j-\varphi)\|_{L_q(\BR_-^N)}=0.
\end{equation*}

We take the inner product between 
the second equation of \eqref{rp1} and $\nabla \varphi$,
and integrate the resultant formula over $\BR_-^N$ in order to obtain
\begin{equation}\label{equiv:eq1}
(\lambda\Bu-\mu\Delta\Bu+\nabla\Fp,\nabla \varphi)_{\BR_-^N}=(\Bf,\nabla\varphi)_{\BR_-^N}.
\end{equation}
The right hand side of this equation vanishes because $\Bf\in J_q(\BR_-^N)$,
and also integration by parts with the third equation of \eqref{rp1} furnishes
\begin{equation*}
(\lambda\Bu,\nabla\varphi)_{\BR_-^N}
=\lim_{j\to \infty}(\lambda\Bu,\nabla\varphi_j)_{\BR_-^N} 
=-\lim_{j\to \infty}(\lambda\dv\Bu,\varphi_j)_{\BR_-^N}=0.
\end{equation*}
Thus \eqref{equiv:eq1} yields
\begin{equation*}
(-\mu\Delta\Bu+\nabla\Fp,\nabla\varphi)_{\BR_-^N}=0,
\end{equation*}
which, combined with the definition of $\CK(\eta,\Bu)$, furnishes
\begin{equation*}
(\nabla(-\CK(\eta,\Bu)+\Fp),\nabla\varphi)_{\BR_-^N}=0.
\end{equation*}
On the other hand, taking the inner product between the fourth equation of \eqref{rp1} and $\Be_N$ yields
\begin{equation*}
\Fp=\Be_N\cdot(\mu\BD(\Bu)\Be_N)+(c_g-c_\sigma\Delta')\eta=\CK(\eta,\Bu) \quad \text{on $\BR_0^N$.}
\end{equation*}
The uniqueness in Lemma \ref{lem:weak-D} then implies $\Fp=\CK(\eta,\Bu)$.

(2) It suffices to show $\dv\Bu=0$.
Let $\varphi\in C_0^\infty(\BR_-^N)$ and
notice that $C_0^\infty(\BR_-^N)\subset \wht H_{q',0}^1(\BR_-^N)$ with $q'=q/(q-1)$.
We take the inner product between the second equation of \eqref{reduced:rp} and $\nabla\varphi$,
and integrate the resultant formula over $\BR_-^N$ in order to obtain
\begin{equation*}
(\lambda\Bu-\mu\Delta\Bu+\nabla\CK(\eta,\Bu),\nabla\varphi)_{\BR_-^N}=(\Bf,\nabla\varphi)_{\BR_-^N}.
\end{equation*}
The definition of $\CK(\eta,\Bu)$ and $\Bf\in J_q(\BR_-^N)$ then give us 
$(\lambda\Bu,\nabla\varphi)_{\BR_-^N}=0$,
which, combined with integration by parts, yields
$-(\lambda\dv\Bu,\varphi)_{\BR_-^N}=0$.
Thus $\dv\Bu=0$ because $\lambda\neq 0$.
This completes the proof of Lemma \ref{lem:equiv}.
\end{proof}

The next lemma follows from Lemmas \ref{lem:SS12} and \ref{lem:equiv} immediately.

\begin{lem}\label{lem:reduced}
Let $\omega\in(0,\pi/2)$ and $q\in(1,\infty)$,
and let $\gamma_0$ be as in Lemma $\ref{lem:SS12}$.
Then, for any $\lambda\in\Sigma_{\omega,\gamma_0}$ and $(d,\Bf)\in X_q$,
\eqref{reduced:rp} admits a unique solution $(\eta,\Bu)\in (W_q^{3-1/q}(\BR^{N-1})\times H_q^2(\BR_-^N)^N) \cap X_q$,
which satisfies
\begin{equation*}
|\lambda|\|(\eta,\Bu)\|_{X_q}+\|(\eta,\Bu)\|_{W_q^{3-1/q}(\BR^{N-1})\times H_q^2(\BR_-^N)^N}
\leq C\|(d,\Bf)\|_{X_q}.
\end{equation*}
\end{lem}

Define the operator $A_q$ by
\begin{equation*}
A_q(\eta,\Bu)=(u_N,\mu\Delta\Bu-\nabla\CK(\eta,\Bu))
\end{equation*}
with the domain
\begin{multline*}
D(A_q)=\{(\eta,\Bu)\in (W_q^{3-1/q}(\BR^{N-1})\times H_q^2(\BR_-^N)^N)\cap X_q
: \\ (\mu\BD(\Bu)\Be_N)_\tau =0 \text{ on $\BR_0^N$}\},
\end{multline*}
where the subscript $\tau$ is defined by \eqref{dfn:tng}.
By Lemma \ref{lem:reduced} and the standard theory of operator semigroups,
we have

\begin{prp}\label{prp:sg}
Let $q\in(1,\infty)$. Then the following assertions hold.
\begin{enumerate}[$(1)$]
\item\label{prp:sg-1}
$A_q$ is a densely defined closed operator on $X_q$.
\item\label{prp:sg-2}
$A_q$ generates a $C_0$-analytic semigroup $(S(t))_{t\geq 0}$ on $X_q$.
\item\label{prp:sg-3}
There exists a positive constant $\gamma_1\geq \gamma_0$ such that 
\begin{equation*}
\|S(t)(\eta_0,\Bu_0)\|_{D(A_q)} 
\leq C e^{\gamma_1 t}\|(\eta_0,\Bu_0)\|_{D(A_q)}
\end{equation*}
for any $(\eta_0,\Bu_0)\in D(A_q)$ and $t>0$.
Here $\gamma_0$ is the positive constant as in Lemma $\ref{lem:SS12}$ and
\begin{equation}\label{norm-dom}
\|(\eta_0,\Bu_0)\|_{D(A_q)}=\|(\eta_0,\Bu_0)\|_{W_q^{3-1/q}(\BR^{N-1})\times H_q^2(\BR_-^N)^N},
\end{equation}
which is equivalent to the graph norm $\|(\eta_0,\Bu_0)\|_{X_q}+\|A_q(\eta_0,\Bu_0)\|_{X_q}$.
\end{enumerate}
\end{prp}

\subsection{Decay properties of the semigroup}
Let $P_i$, $i=1,2$, be the projections defined by
\begin{equation*}
P_1:X_q\to W_q^{2-1/q}(\BR^{N-1}), \quad 
P_2:X_q\to J_q(\BR_-^N),
\end{equation*}
and set
\begin{equation}\label{dfn:S1S2}
S_i(t)(\eta_0,\Bu_0)=P_i S(t)(\eta_0,\Bu_0) \quad \text{for $i=1,2$.}
\end{equation}
In addition, we set 
\begin{align*}
L_q&=L_q(\BR^{N-1})\times J_q(\BR_-^N), \notag \\
\|(\eta,\Bu)\|_{L_q}
&=\|\eta\|_{L_q(\BR^{N-1})}+\|\Bu\|_{L_q(\BR_-^N)}. 
\end{align*}

Recall $\CA$ given by \eqref{dfn:KL}.
From \cite[Theorem 1.1]{SaS16} and \cite[Theorem 3.1.3]{Saito15},
we have

\begin{lem}\label{lem:decay-sg}
Let $N\geq 2$ and $1< p \leq 2 \leq q<\infty$ with $(p,q)\neq (2,2)$.
Then there exist positive constants $\gamma_2$ and $C=C(N,p,q,\gamma_2)$
such that the following assertions hold
for any $(\eta_0,\Bu_0)\in L_p\cap X_q$.
\begin{enumerate}[$(1)$]
\item
Let $t\geq 1$. Then $S_1(t)$ satisfies
\begin{align}
&\|S_1(t)(\eta_0,\Bu_0)\|_{W_q^{3-1/q}(\BR^{N-1})}\notag \\
&\leq C 
\Big(t ^{-\frac{N-1}{2}\left(\frac{1}{p}-\frac{1}{q}\right)} \|(\eta_0,\Bu_0)\|_{L_p} 
+e^{-\gamma_2 t}\|(\eta_0,\Bu_0)\|_{X_q}\Big), \label{SaS:est4} \\
&\|\nabla \CA S_1(t)(\eta_0,\Bu_0)\|_{H_q^2(\BR_-^N)} \notag \\
&\leq C 
\Big(t^{-\frac{N-1}{2}\left(\frac{1}{p}-\frac{1}{q}\right)-\frac{1}{2}\left(\frac{1}{2}-\frac{1}{q}\right)
-\frac{1}{4}} \|(\eta_0,\Bu_0)\|_{L_p} 
+e^{-\gamma_2 t}\|(\eta_0,\Bu_0)\|_{X_q}\Big), \label{SaS:est3} \\
&\|\pd_t \CA S_1(t)(\eta_0,\Bu_0)\|_{H_q^2(\BR_-^N)} \notag \\
&\leq C 
\Big(t^{-\frac{N-1}{2}\left(\frac{1}{p}-\frac{1}{q}\right)-\frac{1}{2}\left(\frac{1}{2}-\frac{1}{q}\right)} \|(\eta_0,\Bu_0)\|_{L_p} 
+e^{-\gamma_2 t}\|(\eta_0,\Bu_0)\|_{X_q}\Big). \label{SaS:est1} 
\end{align}
\item
Let $t\geq 1$. Then $S_2(t)$ satisfies
\begin{align}
&\|S_2(t)(\eta_0,\Bu_0)\|_{H_q^2(\BR_-^N)} \notag \\
& \leq C
\Big(t ^{-\frac{N-1}{2}\left(\frac{1}{p}-\frac{1}{q}\right)-\frac{1}{2}\left(\frac{1}{2}-\frac{1}{q}\right)}
\|(\eta_0,\Bu_0)\|_{L_p} 
+e^{-\gamma_2 t}\|(\eta_0,\Bu_0)\|_{X_q}\Big),  \label{SaS:est6} \\
&\|\nabla S_2(t)(\eta_0,\Bu_0)\|_{H_q^1(\BR_-^N)} \notag \\
& \leq C
\Big( t ^{-\frac{N-1}{2}\left(\frac{1}{p}-\frac{1}{q}\right)-\frac{1}{8}-\delta(p,q)}
\|(\eta_0,\Bu_0)\|_{L_p} 
+e^{-\gamma_2 t}\|(\eta_0,\Bu_0)\|_{X_q}\Big), \label{SaS:est7}
\end{align}
where   
\begin{equation*}
\delta(p,q)=\min\left\{\frac{1}{2}\left(\frac{1}{p}-\frac{1}{q}\right),\frac{1}{8}\left(2-\frac{1}{q}\right)\right\}.
\end{equation*}
\end{enumerate}
\end{lem}

Recall $\CE_N$ and $\|\cdot\|_{D(A_q)}$ given by \eqref{ext:eta} and \eqref{norm-dom}, respectively.
Lemma \ref{lem:decay-sg} enables us to prove the following proposition.

\begin{prp}\label{prp:decay}
Let $N\geq 3$, $q_0=2/9$, and $\delta_0=1/30$.
Let $q_1$ satisfy $2<q_1\leq 2+q_0$.
Then the following assertions hold for any $q\geq q_1$ and $(\eta_0,\Bu_0)\in L_{q_1/2}\cap X_{q}$.
\begin{enumerate}[$(1)$]
\item\label{prp:decay-2}
Suppose $N\geq 4$ and $t\geq 1$. 
Then
\begin{equation}\label{sg-est-4d}
\|S(t)(\eta_0,\Bu_0)\|_{D(A_q)}
\leq C  t ^{-\frac{1}{2}-\delta_0}\|(\eta_0,\Bu_0)\|_{L_{q_1/2}\cap X_q},
\end{equation}
with some positive constant $C$ independent of $t$, $\eta_0$, and $\Bu_0$.
\item\label{prp:decay-1}
Suppose $N=3$ and $t\geq 1$. Then 
\begin{equation}\label{sg-est-3d}
\|S(t)(\eta_0,\Bu_0)\|_{D(A_q)}
\leq C  t ^{-\frac{1}{3}-\delta_0}\|(\eta_0,\Bu_0)\|_{L_{q_1/2}\cap X_q}.
\end{equation}
In addition, 
\begin{align}
\|\nabla \CE_N S_1(t)(\eta_0,\Bu_0)\|_{H_q^2(\BR_-^N)}
&\leq C t ^{-\frac{2}{3}-\delta_0}\|(\eta_0,\Bu_0)\|_{L_{q_1/2}\cap X_{q}}, \label{sg-est-3d-4} \\
\|\nabla S_2(t)(\eta_0,\Bu_0)\|_{H_q^1(\BR_-^N)}
&\leq  C  t ^{-\frac{2}{3}-\delta_0}\|(\eta_0,\Bu_0)\|_{L_{q_1/2}\cap X_{q}}. \label{sg-est-3d-5}
\end{align}
Here $C$ is a positive constant independent of $t$, $\eta_0$, and $\Bu_0$.
\end{enumerate}
\end{prp}

\begin{proof}
In this proof, we use Lemma \ref{lem:decay-sg} with $p=q_1/2$.

(1) Since $N\geq 4$ and $q\geq q_1$, one observes that
\begin{equation*}
\frac{N-1}{2}\left(\frac{2}{q_1}-\frac{1}{q}\right)
\geq \frac{3}{2}\left(\frac{2}{q_1}-\frac{1}{q_1}\right)=\frac{3}{2q_1}\geq \frac{3}{2(2+q_0)}=\frac{27}{40} \geq \frac{1}{2}+\delta_0.
\end{equation*}
Thus \eqref{SaS:est4} and \eqref{SaS:est6} yields \eqref{sg-est-4d}.

(2) 
Since $N=3$ and $q\geq q_1$, one observes that
\begin{equation*}
\frac{N-1}{2}\left(\frac{2}{q_1}-\frac{1}{q}\right)=\frac{2}{q_1}-\frac{1}{q}\geq \frac{2}{q_1}-\frac{1}{q_1}
=\frac{1}{q_1}
\geq \frac{1}{2+q_0}=\frac{9}{20}.
\end{equation*}
This implies
\begin{equation}\label{est:expo-decay}
\frac{N-1}{2}\left(\frac{2}{q_1}-\frac{1}{q}\right)+\frac{1}{4}\geq \frac{9}{20}+\frac{1}{4}=\frac{7}{10}= \frac{2}{3}+\delta_0.
\end{equation}
It thus holds for any $q\geq q_1$ that
\begin{equation*}
\frac{N-1}{2}\left(\frac{2}{q_1}-\frac{1}{q}\right)\geq \frac{1}{3}+\delta_0, \quad 
\frac{N-1}{2}\left(\frac{2}{q_1}-\frac{1}{q}\right)+\frac{1}{4}\geq \frac{2}{3}+\delta_0,
\end{equation*}
which, combined with \eqref{SaS:est4}, \eqref{SaS:est3}, and  \eqref{SaS:est6}, furnishes
\eqref{sg-est-3d} and \eqref{sg-est-3d-4}.

We next prove \eqref{sg-est-3d-5}. Let us observe that
\begin{align*}
\frac{1}{2}\left(\frac{2}{q_1}-\frac{1}{q}\right) &\geq 
\frac{1}{2}\left(\frac{2}{q_1}-\frac{1}{q_1}\right)=\frac{1}{2q_1}\geq \frac{1}{2(2+q_0)}=\frac{9}{40}, \\
\frac{1}{8}\left(2-\frac{1}{q}\right)&\geq \frac{1}{8}\left(2-\frac{1}{q_1}\right)>\frac{1}{8}\left(2-\frac{1}{2}\right)=\frac{3}{16}.
\end{align*}
Thus $\delta(q_1/2,q)\geq 1/8$, which, combined with \eqref{est:expo-decay}, furnishes
\begin{equation*}
\frac{N-1}{2}\left(\frac{2}{q_1}-\frac{1}{q}\right)+\frac{1}{8}+\delta(q_1/2,q)
\geq \frac{N-1}{2}\left(\frac{2}{q_1}-\frac{1}{q}\right)+\frac{1}{4}\geq \frac{2}{3}+\delta_0.
\end{equation*}
Combining this with \eqref{SaS:est7} furnishes \eqref{sg-est-3d-5},
which completes the proof of Proposition \ref{prp:decay}.
\end{proof}

Furthermore, we have the following $L_2$-type estimate for $\pd_t\CE_N S_1(t)(\eta_0,\Bu_0)$.

\begin{prp}\label{prp:decay-L2}
Let $N\geq 3$ and $q_0,\delta_0$ be as in Proposition $\ref{prp:decay}$.
Let $q_1$ satisfy $2<q_1\leq 2+q_0$.
Then for any $(\eta_0,\Bu_0)\in L_{q_1/2}\cap X_{2}$ and $t\geq 1$
\begin{equation*}
\|\pd_t\CE_N S_1(t)(\eta_0,\Bu_0)\|_{H_2^2(\BR_-^N)}
\leq C t^{-\frac{1}{3}-\delta_0}\|(\eta_0,\Bu_0)\|_{L_{q_1/2}\cap X_2}, \label{sg-est-3d-2}
\end{equation*}
with some positive constant $C$ independent of $t$, $\eta_0$, and $\Bu_0$.
\end{prp}

\begin{proof}
Since $N\geq 3$, one observes that
\begin{equation*}
\frac{N-1}{2}\left(\frac{2}{q_1}-\frac{1}{2}\right)\geq \frac{2}{q_1}-\frac{1}{2}
\geq \frac{2}{2+q_0}-\frac{1}{2}=\frac{2}{5}\geq \frac{1}{3}+\delta_0.
\end{equation*}
The desired inequality thus follows from \eqref{SaS:est1}.
This completes the proof of Proposition \ref{prp:decay-L2}
\end{proof}

\section{Linear theory}\label{sec:lin-theory}

This section introduces our linear theory in order to prove Theorem \ref{thm:transformed}.
All of the results of Subsection \ref{subsec:MR} below hold for $N\geq 2$.

\subsection{Maximal regularity}\label{subsec:MR}
Recall $\BQ_-$ and $\BQ_0$ given by \eqref{dfn:Q-Q0}.
Let us consider
\begin{equation}\label{lin-eq:1}
\left\{\begin{aligned}
\pd_t\eta -u_N &=d && \text{on $\BQ_0$,} \\
\pd_t\Bu-\mu\Delta\Bu+\nabla\Fp&=\Bf && \text{in $\BQ_-$,} \\
\dv\Bu= g&=\dv\Fg && \text{in $\BQ_-$,} \\
(\mu\BD(\Bu)-\Fp\BI)\Be_N
+(c_g-c_\sigma\Delta')\eta\Be_N&=\Bh && \text{on $\BQ_0$,} \\
\eta|_{t=0}=\eta_0 \quad \text{on $\BR^{N-1}$,} \quad
\Bu|_{t=0}&=\Bu_0 && \text{in $\BR_-^N$.}
\end{aligned}\right.
\end{equation}

To state the maximal regularity of \eqref{lin-eq:1},
we introduce some norms.
Let $w(t)$ be a function of time $t$. For solutions of \eqref{lin-eq:1}, we set
\begin{align*}
\|(\eta,\Bu)\|_{E_{p,q}(w(t))}
&=\|w(t)\pd_t\eta\|_{L_p(\BR_+,W_q^{2-1/q}(\BR^{N-1}))}
+\|w(t)\eta\|_{L_p(\BR_+,W_q^{3-1/q}(\BR^{N-1}))} \\
&+\|w(t)\pd_t\Bu\|_{L_p(\BR_+,L_q(\BR_-^N)^N)} 
+\|w(t)\Bu\|_{L_p(\BR_+,H_q^2(\BR_-^N)^N)}.
\end{align*}
For right members of \eqref{lin-eq:1}, we set
\begin{align*}
\|(d,\Bf,\Fg,g,\Bh)\|_{F_{p,q}(w(t))}
&=\|w(t)d\|_{L_p(\BR_+,W_q^{2-1/q}(\BR^{N-1}))}
+\|w(t)\Bf\|_{L_p(\BR_+,L_q(\BR_-^N)^N)} \\
&+\|w(t)\pd_t\Fg\|_{L_p(\BR_+,L_q(\BR_-^N)^N)}
+\|w(t)\Fg\|_{L_p(\BR_+,L_q(\BR_-^N)^N)} \\
&+\|w(t) g\|_{H_p^{1/2}(\BR_+,L_q(\BR_-^N))}
+\|w(t) g\|_{ L_p(\BR_+,H_q^1(\BR_-^N))} \\
&+\|w(t) \Bh\|_{H_p^{1/2}(\BR_+,L_q(\BR_-^N)^N)}
+\|w(t) \Bh\|_{L_p(\BR_+,H_q^1(\BR_-^N)^N)}.
\end{align*}

Recall $I_{q,p}$ given by \eqref{dfn:ini-sp} for the initial data.
The maximal regularity for \eqref{lin-eq:1} has been proved in \cite[Theorem 3.4.2]{Shibata20},
see also \cite[Theorem 1.4]{SS12}.

\begin{lem}\label{lem:mr}
Let $p,q\in(1,\infty)$ with $2/p+1/q\neq 1$.
Then there exists a positive constant $\gamma_3$ such that 
for any $(\eta_0,\Bu_0)\in I_{q,p}$
and for any 
\begin{align*}
e^{-\gamma_3 t} d 
&\in L_p(\BR_+,W_q^{2-1/q}(\BR^{N-1})), \\
e^{-\gamma_3 t} \Bf 
&\in L_p(\BR_+,L_q(\BR_-^N)^N), \\
e^{-\gamma_3 t} \Fg 
&\in H_p^1(\BR_+,L_q(\BR_-^N)^N), \\
e^{-\gamma_3 t} g
&\in H_p^{1/2}(\BR_+,L_q(\BR_-^N))\cap L_p(\BR_+,H_q^1(\BR_-^N)), \\
e^{-\gamma_3 t}\Bh 
&\in H_p^{1/2}(\BR_+,L_q(\BR_-^N)^N)\cap L_p(\BR_+,H_q^1(\BR_-^N)^N),
\end{align*}
with the following compatibility conditions {\rm (a)} and {\rm (b):}
\begin{enumerate}[{\rm (a)}]
\item
$\Bu_0-\Fg|_{t=0}\in J_q(\BR_-^N)$ and $\dv\Bu_0=g|_{t=0}$ in $\BR_-^N$,
\item
$(\mu\BD(\Bu_0)\Be_N)_\tau=(\Bh|_{t=0})_\tau$ in $B_{q,p}^{1-2/p-1/q}(\BR_0^N)$ if $2/p+1/q<1$,
\end{enumerate}
\eqref{lin-eq:1} admits a unique solution $(\eta,\Bu,\Fp)$ with
\begin{align*}
\eta &\in H_{p,\lc}^1(\BR_+,W_q^{2-1/q}(\BR^{N-1}))\cap L_{p,\lc}(\BR_+,W_q^{3-1/q}(\BR^{N-1})), \notag \\
\Bu &\in H_{p,\lc}^1(\BR_+,L_q(\BR_-^N)^N) \cap L_{p,\lc}(\BR_+,H_q^2(\BR_-^N)^N), \notag \\
\Fp&\in L_{p,\lc}(\BR_+,\wht H_{q,0}^1(\BR_-^N)+H_q^1(\BR_-^N)).
\end{align*}
In addition, the following estimates hold.
\begin{enumerate}[$(1)$]
\item
There exists a positive constant $C=C(N,p,q,\gamma_3)$ such that
\begin{equation*}
\|(\eta,\Bu)\|_{E_{p,q}(e^{-\gamma_3 t})}\leq C\Big(\|(\eta_0,\Bu_0)\|_{I_{q,p}}
+\|(d,\Bf,\Fg,g,\Bh)\|_{F_{p,q}(e^{-\gamma_3 t})}\Big).
\end{equation*}
\item
For any $T>0$
\begin{align*}
&\|(\pd_t\eta,\pd_t\Bu)\|_{L_p((0,T),W_q^{2-1/q}(\BR^{N-1})\times L_q(\BR_-^N)^N)} \\
&+\|(\eta,\Bu)\|_{L_p((0,T),W_q^{3-1/q}(\BR^{N-1})\times H_q^2(\BR_-^N)^N)} \\
&\leq C e^{\gamma_3 T}\Big(\|(\eta_0,\Bu_0)\|_{I_{q,p}}
+\|(d,\Bf,\Fg,g,\Bh)\|_{F_{p,q}(e^{-\gamma_3 t})}\Big),
\end{align*}
where $C$ is a positive constant independent of $T$.
\end{enumerate}
\end{lem}

We next consider the following system with $\gamma_3$ given by Lemma \ref{lem:mr}:
\begin{equation}\label{t-shift:eq1}
\left\{\begin{aligned}
\pd_t\eta +\gamma_3\eta-u_N&=d && \text{on $\BQ_0$,} \\
\pd_t\Bu+\gamma_3\Bu-\mu\Delta\Bu+\nabla\Fp&=\Bf && \text{in $\BQ_-$,} \\
\dv\Bu= g &=\dv\Fg && \text{in $\BQ_-$,} \\
\mu(\BD(\Bu)-\Fp\BI)\Be_N+(c_g-c_\sigma\Delta')\eta\Be_N&=\Bh && \text{on $\BQ_0$,} \\
\eta|_{t=0}=\eta_0 \quad \text{on $\BR^{N-1}$,} \quad
\Bu|_{t=0}&=\Bu_0 && \text{in $\BR_-^N$.}
\end{aligned}\right.
\end{equation}
The next lemma follows from Lemma \ref{lem:mr} immediately.

\begin{lem}\label{lem:t-shift}
Let $p,q\in(1,\infty)$ with $2/p+1/q\neq 1$.
Then for any $(\eta_0,\Bu_0)\in I_{q,p}$
and for any
\begin{align*}
d &\in L_p(\BR_+,W_q^{2-1/q}(\BR^{N-1})), \\
\Bf &\in L_p(\BR_+, L_q(\BR_-^N)^N), \\
\Fg &\in H_p^1(\BR_+,L_q(\BR_-^N)^N), \\
g&\in H_p^{1/2}(\BR_+,L_q(\BR_-^N))\cap L_p(\BR_+,H_q^1(\BR_-^N)), \\
\Bh &\in H_p^{1/2}(\BR_+,L_q(\BR_-^N)^N)\cap L_p(\BR_+,H_q^1(\BR_-^N)^N),
\end{align*}
with the compatibility conditions {\rm (a)} and {\rm (b)} stated in Lemma $\ref{lem:mr}$,
\eqref{t-shift:eq1} admits a unique solution $(\eta,\Bu,\Fp)$ with
\begin{align*}
\eta &\in H_{p}^1(\BR_+,W_q^{2-1/q}(\BR^{N-1}))\cap L_{p}(\BR_+,W_q^{3-1/q}(\BR^{N-1})), \notag \\
\Bu &\in H_{p}^1(\BR_+,L_q(\BR_-^N)^N) \cap L_{p}(\BR_+,H_q^2(\BR_-^N)^N), \notag \\
\Fp&\in L_{p}(\BR_+,\wht H_{q,0}^1(\BR_-^N)+H_q^1(\BR_-^N)).
\end{align*}
In addition, 
\begin{equation*}
\|(\eta,\Bu)\|_{E_{p,q}(1)}
\leq C\Big(\|(\eta_0,\Bu_0)\|_{I_{q,p}}+\|(d,\Bf,\Fg,g,\Bh)\|_{F_{p,q}(1)}\Big)
\end{equation*}
for some positive constant $C=C(N,p,q,\gamma_3)$.
\end{lem}

Recall $\langle t \rangle = \sqrt{1+t^2}$.
Similarly to the argumentation in \cite[page 436]{Shibata20},
we obtain the following proposition from Lemma \ref{lem:t-shift}.

\begin{prp}\label{prp:t-shift}
Let $p,q\in(1,\infty)$ with $2/p+1/q\neq 1$ and let $b>0$.
Then for any $(\eta_0,\Bu_0)\in I_{q,p}$
and for any
\begin{align*}
\langle t\rangle^b d  &\in L_p(\BR_+,W_q^{2-1/q}(\BR^{N-1})), \\
\langle t\rangle^b\Bf &\in L_p(\BR_+, L_q(\BR_-^N)^N), \\
\langle t\rangle^b \Bg &\in H_p^1(\BR_+,L_q(\BR_-^N)^N), \\
\langle t\rangle^b g &\in H_p^{1/2}(\BR_+,L_q(\BR_-^N))\cap L_p(\BR_+,H_q^1(\BR_-^N)), \\
\langle t\rangle^b \Bh &\in H_p^{1/2}(\BR_+,L_q(\BR_-^N)^N)\cap L_p(\BR_+,H_q^1(\BR_-^N)^N),
\end{align*}
with the compatibility conditions {\rm (a)} and {\rm (b)} stated in Lemma $\ref{lem:mr}$,
the unique solution $(\eta,\Bu,\Fp)$ of \eqref{t-shift:eq1} obtained in Lemma $\ref{lem:t-shift}$ satisfies
\begin{equation*}
\|(\eta,\Bu)\|_{E_{p,q}(\langle t \rangle^b)}
\leq C\Big(\|(\eta_0,\Bu_0)\|_{I_{q,p}}+\|(d,\Bf,\Fg,g,\Bh)\|_{F_{p,q}(\langle t \rangle^b)}\Big)
\end{equation*}
for some positive constant $C=C(N,p,q,\gamma_3,b)$.
\end{prp}

\subsection{Time-weighed estimates}\label{subsec:6-2}
This subsection is concerned with 
time-weighted estimates for solutions of \eqref{lin-eq:1} given by Lemma \ref{lem:mr}.
Let us start with the case $N\geq 4$.

\begin{prp}\label{prp:linear4d}
Suppose $N\geq 4$. Let $q_0$, $\delta_0$ be as in Proposition $\ref{prp:decay}$ and let $p_0=31$.
Then  for any $p$, $q_1$, and $q_2$ satisfying
\begin{equation}\label{cond-max-pq}
p_0\leq p <\infty, \quad 
2<q_1\leq 2+q_0, \quad N < q_2<\infty, \quad \frac{2}{p}+\frac{1}{(q_1/2)}\neq 1
\end{equation}
and for any $(\eta_0,\Bu_0)\in \bigcap_{r\in\{q_1/2,q_2\}} I_{r,p}$,
\begin{align*}
\langle t \rangle d 
&\in \bigcap_{r\in\{q_1/2,q_2\}} L_p(\BR_+,W_r^{2-1/r}(\BR^{N-1})), \quad \\
\langle t \rangle \Bf
&\in \bigcap_{r\in\{q_1/2,q_2\}} L_p(\BR_+,L_{r}(\BR_-^N)^N),  \\
\langle t \rangle \Fg
& \in \bigcap_{r\in\{q_1/2,q_2\}} H_p^1(\BR_+,L_{r}(\BR_-^N)^N), \\
\langle t \rangle g
&\in \bigcap_{r\in\{q_1/2,q_2\}} H_p^{1/2}(\BR_+,L_r(\BR_-^N))\cap L_p(\BR_+,H_r^1(\BR_-^N)), \\
\langle t \rangle \Bh
&\in \bigcap_{r\in\{q_1/2,q_2\}} H_p^{1/2}(\BR_+,L_r(\BR_-^N)^N)\cap L_p(\BR_+,H_r^1(\BR_-^N)^N),
\end{align*}
with the compatibility conditions 
\begin{alignat}{2}
\dv\Bu_0&=g|_{t=0}  && \quad \text{in $\BR_-^N$,} \label{comp:t-wegiht1} \\
(\mu\BD(\Bu_0)\Be_N)_\tau &=(\Bh|_{t=0})_\tau && \quad \text{on $\BR_0^N$,} \label{comp:t-wegiht2}
\end{alignat}
the unique solution $(\eta,\Bu,\Fp)$ of \eqref{lin-eq:1} obtained in Lemma $\ref{lem:mr}$ satisfies
for $q\in\{q_1,q_2\}$
\begin{equation*}
\|(\eta,\Bu)\|_{E_{p,q}(\langle t \rangle^{1/2})}
\leq L_1\sum_{r\in\{q_1/2,q_2\}} \Big(\|(\eta_0,\Bu_0)\|_{I_{r,p}}+\|(d,\Bf,\Fg,g,\Bh)\|_{F_{p,r}(\langle t \rangle)}\Big)
\end{equation*}
with some positive constant  $L_1$.
\end{prp}

\begin{rmk}\label{rmk:compati}
\begin{enumerate}[$(1)$]
\item
\eqref{comp:t-wegiht1} implies $({\rm a})$ of Lemma $\ref{lem:mr}$ for $q\in\{q_1/2,q_2\}$
as follows.

Let  $q\in\{q_1/2,q_2\}$ and $\varphi\in \wht H_{q,0}^1(\BR_-^N)$.
Lemma $\ref{lem:dense}$ gives us a sequence  $\{\varphi_j\}_{j=1}^\infty \subset C_0^\infty(\BR_-^N)$ such that
$\lim_{j\to \infty}\|\nabla(\varphi_j-\varphi)\|_{L_q(\BR_-^N)}=0$.  
Then 
\begin{align*}
(\Bu_0-\Fg|_{t=0},\nabla\varphi)_{\BR_-^N}
&=\lim_{j\to \infty}(\Bu_0-\Fg|_{t=0},\nabla\varphi_j)_{\BR_-^N} \\
&=-\lim_{j\to \infty}(\dv\Bu_0-g|_{t=0},\varphi_j)_{\BR_-^N} =0,
\end{align*}
which yields $\Bu_0-\Fg|_{t=0}\in J_q(\BR_-^N)$. 

\item
\eqref{comp:t-wegiht2} implies $({\rm b})$ of Lemma $\ref{lem:mr}$ for $q=q_2$
and for $q=q_1/2$ if $2/p+1/(q_1/2)<1$.
\end{enumerate}
\end{rmk}

\begin{proof}
Let $(\eta^i,\Bu^i,\Fp^i)$, $i=1,2,3$, be the solutions to the following systems:
\begin{align}
&\left\{\begin{aligned}
\pd_t\eta^1+\gamma_3\eta^1 -u_N^1&=d && \text{on $\BQ_0$, } \\
\pd_t\Bu^1+\gamma_3\Bu^1-\mu\Delta\Bu^1+\nabla\Fp^1&=\Bf && \text{in $\BQ_-$, } \\
\dv\Bu^1= g&=\dv\Fg && \text{in $\BQ_-$,} \\
(\mu\BD(\Bu^1)-\Fp^1\BI)\Be_N
+(c_g-c_\sigma\Delta')\eta^1\Be_N&=\Bh && \text{on $\BQ_0$, } \\
\eta^1|_{t=0}=\eta_0 \quad \text{on $\BR^{N-1}$,} \quad
\Bu^1|_{t=0}&=\Bu_0 && \text{in $\BR_-^N$,}
\end{aligned}\right. \label{decomp:1} \\
&\left\{\begin{aligned}
\pd_t\eta^2 +\gamma_3 \eta^2 -u_N^2&=\gamma_3\eta^1 && \text{on $\BQ_0$,} \\
\pd_t\Bu^2+\gamma_3\Bu^2-\mu\Delta\Bu^2+\nabla\Fp^2&=\gamma_3\Bu^1 && \text{in $\BQ_-$,} \\
\dv\Bu^2&=0 && \text{in $\BQ_-$,} \\
(\mu\BD(\Bu^2)-\Fp^2\BI)\Be_N
+(c_g-c_\sigma\Delta')\eta^2\Be_N&=0 && \text{on $\BQ_0$,} \\
\eta^2|_{t=0}=0 \quad \text{on $\BR^{N-1}$,} \quad 
\Bu^2|_{t=0}&=0 && \text{in $\BR_-^N$,}
\end{aligned}\right. \label{decomp:2} \\
&\left\{\begin{aligned}
\pd_t\eta^3 -u_N^3&=\gamma_3\eta^2 && \text{on $\BQ_0$, } \\
\pd_t\Bu^3-\mu\Delta\Bu^3+\nabla\Fp^3&=\gamma_3\Bu^2 && \text{in $\BQ_-$,} \\
\dv\Bu^3&=0 && \text{in $\BQ_-$,} \\
(\mu\BD(\Bu^3)-\Fp^3\BI)
+(c_g-c_\sigma\Delta')\eta^3\Be_N&=0 && \text{on $\BQ_0$,} \\
\eta|_{t=0}=0 \quad  \text{on $\BR^{N-1}$,} \quad 
\Bu|_{t=0}&=0 && \text{in $\BR_-^N$,}
\end{aligned}\right. \label{decomp:3} 
\end{align} 
where $\gamma_3$ is given by Lemma \ref{lem:mr}.
Then the solution $(\eta,\Bu,\Fp)$ of \eqref{lin-eq:1} is given by
\begin{equation*}
\eta= \eta^1+\eta^2+\eta^3, \quad
\Bu=\Bu^1+\Bu^2+\Bu^3, \quad 
\Fp=\Fp^1+\Fp^2+\Fp^3.
\end{equation*}
From this viewpoint,
we estimate $(\eta^i,\Bu^i,\Fp^i)$ for $i=1,2,3$ in what follows.
Let $r\in\{q_1/2,q_2\}$ and $s\in\{q_1/2, q_1,q_2\}$ in the following argumentation.

{\bf Step 1}: Estimate $(\eta^1,\Bu^1)$.
Noting Remark \ref{rmk:compati}, we use Proposition \ref{prp:t-shift} with $b=1$ in order to obtain
\begin{equation}\label{est:case1}
\|(\eta^1,\Bu^1)\|_{E_{p,r}(\langle t \rangle)}
\leq C\Big(\|(\eta_0,\Bu_0)\|_{I_{r,p}}
+\|(d,\Bf,\Bg,g,\Bh)\|_{F_{p,r}(\langle t \rangle)}\Big).
\end{equation}
On the other hand, $1/q_1=(1-\theta)/(q_1/2)+\theta/q_2$ for some $\theta\in(0,1)$,
and thus 
Lemma \ref{lem:int-p-v2}
together with $a^{1-\theta}b^\theta\leq C(a+b)$ for $a,b\geq 0$ furnishes
\begin{equation*}
\|(\eta^1,\Bu^1)\|_{E_{p,q_1}(\langle t \rangle)}\leq C\sum_{r\in\{q_1/2,q_2\}}\|(\eta^1,\Bu^1)\|_{E_{p,r}(\langle t \rangle)}.
\end{equation*}
Combining this inequality with \eqref{est:case1} yields
\begin{align*}
\|(\eta^1,\Bu^1)\|_{E_{p,q_1}(\langle t \rangle)} 
\leq C\sum_{r\in\{q_1/2,q_2\}}\Big(\|(\eta_0,\Bu_0)\|_{I_{r,p}}
+\|(d,f,\Bg,g,\Bh)\|_{F_{p,r}(\langle t \rangle)}\Big).
\end{align*}

Summing up the last inequality and \eqref{est:case1},
we have obtained for  $s\in\{q_1/2,q_1,q_2\}$
\begin{align}\label{est:case1-2}
&\|(\eta^1,\Bu^1)\|_{E_{p,s}(\langle t \rangle)} \notag  \\
&\leq C \sum_{r\in \{q_1/2,q_2\}}\Big(\|(\eta_0,\Bu_0)\|_{I_{r,p}}
+\|(d,\Bf,\Bg,g,\Bh)\|_{F_{p,r}(\langle t \rangle)}\Big).
\end{align}

{\bf Step 2}: Estimate $(\eta^2,\Bu^2)$.
Proposition \ref{prp:t-shift} with $b=1$ yields
\begin{equation*}
\|(\eta^2,\Bu^2)\|_{E_{p,s}(\langle t \rangle)}
\leq C\|(\eta^1,\Bu^1,0,0,0)\|_{F_{p,s}(\langle t \rangle)},
\end{equation*}
which, combined with $\|(\eta^1,\Bu^1,0,0,0)\|_{F_{p,s}(\langle t \rangle)}
\leq \|(\eta^1,\Bu^1)\|_{E_{p,s}(\langle t \rangle)}$, furnishes
\begin{equation*}
\|(\eta^2,\Bu^2)\|_{E_{p,s}(\langle t \rangle)}
\leq C \|(\eta^1,\Bu^1)\|_{E_{p,s}(\langle t \rangle)}.
\end{equation*}
Thus \eqref{est:case1-2} yields for $s\in\{q_1/2,q_1,q_2\}$
\begin{align}\label{est:case2-2}
&\|(\eta^2,\Bu^2)\|_{E_{p,s}(\langle t \rangle)} \notag \\
&\leq C \sum_{r\in \{q_1/2,q_2\}}\Big(\|(\eta_0,\Bu_0)\|_{I_{r,p}}
+\|(d,\Bf,\Bg,g,\Bh)\|_{F_{p,r}(\langle t \rangle)}\Big).
\end{align}

Let $q\in\{q_1,q_2\}$ and recall $X_q$, $A_q$, $D(A_q)$, and $L_{q_1/2}$ defined in the previous section.
Notice that Lemma \ref{lem:dense} and $\dv\Bu^2=0$ imply 
\begin{equation*}
\Bu^2\in J_{s}(\BR_-^N) \quad \text{for $s\in\{q_1/2,q_1,q_2\}$}.
\end{equation*}
This shows that $(\eta^2,\Bu^2)\in W_q^{3-1/q}(\BR^{N-1})\times (H_q^2(\BR_-^N)\cap J_{q}(\BR_-^N))$,
which, combined with the fourth equation of \eqref{decomp:2}, furnishes $(\eta^2,\Bu^2)\in D(A_q)$.
Thus \eqref{est:case2-2} and $\Bu^2\in J_{q_1/2}(\BR_-^N)$ give us
\begin{equation}\label{case2-domain}
\langle t \rangle (\eta^2,\Bu^2)\in L_p(\BR_+,(L_{q_1/2}\cap X_q)\cap D(A_{q})) 
\end{equation}
with
\begin{align}\label{est:domain-A}
&\|\langle t\rangle (\eta^2,\Bu^2)\|_{L_p(\BR_+,(L_{q_1/2}\cap X_q)\cap D(A_{q}))} \notag \\
&\leq C \sum_{r\in \{q_1/2,q_2\}}\Big(\|(\eta_0,\Bu_0)\|_{I_{r,p}}
+\|(d,\Bf,\Bg,g,\Bh)\|_{F_{p,r}(\langle t \rangle)}\Big).
\end{align}

{\bf Step 3}: Estimate $(\eta^3,\Bu^3)$.
Recall $q\in\{q_1,q_2\}$ and consider the case $t\in (0,3)$.
By Lemma \ref{lem:mr}, we have
\begin{align*}
&\|(\pd_t\eta^3,\pd_t \Bu^3)\|_{L_p((0,3),W_q^{2-1/q}(\BR^{N-1})\times L_q(\BR_-^N)^N)} \\
&+\|(\eta^3,\Bu^3)\|_{L_p((0,3),W_q^{3-1/q}(\BR^{N-1})\times H_q^2(\BR_-^N)^N)}   \notag \\
&\leq 
C\|(\eta^2,\Bu^2,0,0,0)\|_{F_{p,q}(e^{-\gamma_3 t})}
\end{align*}
for some positive constant $C$.
Notice that
\begin{align*}
\|(\eta^2,\Bu^2,0,0,0)\|_{F_{p,q}(e^{-\gamma_3 t})}
&\leq 
\|(\eta^2,\Bu^2,0,0,0)\|_{F_{p,q}(1)} \\
&\leq 
\|(\eta^2,\Bu^2,0,0,0)\|_{F_{p,q}(\langle t \rangle)},
\end{align*}
which, combined with $\|(\eta^2,\Bu^2,0,0,0)\|_{F_{p,q}(\langle t \rangle)}\leq \|(\eta^2,\Bu^2)\|_{E_{p,q}(\langle t \rangle)}$
and \eqref{est:case2-2}, furnishes that
\begin{align*}
&\|(\eta^2,\Bu^2,0,0,0)\|_{F_{p,q}(e^{-\gamma_3 t})} \\
&\leq C \sum_{r\in \{q_1/2,q_2\}}\Big(\|(\eta_0,\Bu_0)\|_{I_{r,p}}
+\|(d,\Bf,\Bg,g,\Bh)\|_{F_{p,r}(\langle t \rangle)}\Big).
\end{align*}
Thus
\begin{align}\label{est:local}
&\|(\pd_t\eta^3,\pd_t \Bu^3)\|_{L_p((0,3),W_q^{2-1/q}(\BR^{N-1})\times L_q(\BR_-^N)^N)} \notag \\
&+\|(\eta^3,\Bu^3)\|_{L_p((0,3),W_q^{3-1/q}(\BR^{N-1})\times H_q^2(\BR_-^N)^N)}   \notag \\
&\leq 
C
\sum_{r\in\{q_1/2,q_2\}}\Big(\|(\eta_0,\Bu_0)\|_{I_{r,p}}
+\|(d,\Bf,\Bg,g,\Bh)\|_{F_{p,r}(\langle t \rangle)}\Big).
\end{align}

We next consider the case $t\geq 2$. Let us write
\begin{equation*}
(\eta^3(t),\Bu^3(t))
=\int_0^t S(t-\tau)(\gamma_3\eta^2(\tau),\gamma_3\Bu^2(\tau)) \intd\tau.
\end{equation*}
In view of Proposition \ref{prp:decay} \eqref{prp:decay-2},
we apply Proposition \ref{prp:duha} to the above formula of  $(\eta^3,\Bu^3)$.
More precisely, we use Proposition \ref{prp:duha} with
 $T(t)=S(t)$, $X=X_q$, $A=A_q$, and $D(A)=D(A_q)$
under the condition that
\begin{equation*}
a=1/2, \quad \delta=\delta_0, \quad s_0=p_0, \quad Y=L_{q_1/2}\cap X_q,
\end{equation*}
where we note that $s_0\delta=p_0\delta_0=31/30>1$.
Then \eqref{case2-domain} and \eqref{est:domain-A} show that
\begin{align*}
&\|\langle t\rangle^{1/2}(\pd_t\eta^3,\pd_t\Bu^3)\|_{L_p((2,\infty),X_q)} 
+\|\langle t\rangle^{1/2}(\eta^3,\Bu^3)\|_{L_p((2,\infty),D(A_q))} \\
&\leq C\|\langle t \rangle(\eta^2,\Bu^2)\|_{L_p(\BR_+,(L_{q_1/2}\cap X_q) \cap D(A_q))} \\
&\leq 
C
\sum_{r\in\{q_1/2,q_2\}}\Big(\|(\eta_0,\Bu_0)\|_{I_{r,p}}
+\|(d,\Bf,\Bg,g,\Bh)\|_{F_{p,r}(\langle t \rangle)}\Big).
\end{align*}
Combining this inequality with \eqref{est:local} yields
\begin{align*}
\|(\eta^3,\Bu^3)\|_{E_{p,q}(\langle t\rangle^{1/2})} 
\leq 
C
\sum_{r\in\{q_1/2,q_2\}}\Big(\|(\eta_0,\Bu_0)\|_{I_{r,p}}
+\|(d,\Bf,\Bg,g,\Bh)\|_{F_{p,r}(\langle t \rangle)}\Big).
\end{align*}

Summing up the last inequality with \eqref{est:case1-2} and \eqref{est:case2-2},
we have obtained the desired estimate of the solution $(\eta,\Bu,\Fp)$ of \eqref{lin-eq:1}.
This completes the proof of Proposition \ref{prp:linear4d}.
\end{proof}

We next consider the case $N=3$.

\begin{prp}\label{prp:linear3d}
Let $N= 3$ and let $q_0$, $\delta_0$ be as in Proposition $\ref{prp:decay}$.
Let $p_0$ be as in Proposition  $\ref{prp:linear4d}$.
Then for any $p$, $q_1$, and $q_2$ satisfying \eqref{cond-max-pq}
and for any $(\eta_0,\Bu_0)\in \bigcap_{r\in\{q_1/2,q_2\}} I_{r,p}$,
\begin{align*}
\langle t \rangle d 
&\in \bigcap_{r\in\{q_1/2,q_2\}} L_p(\BR_+,W_r^{2-1/r}(\BR^{N-1})), \quad \\
\langle t \rangle \Bf
&\in \bigcap_{r\in\{q_1/2,q_2\}} L_p(\BR_+,L_{r}(\BR_-^N)^N),  \\
\langle t \rangle \Fg
& \in \bigcap_{r\in\{q_1/2,q_2\}} H_p^1(\BR_+,L_{r}(\BR_-^N)^N), \\
\langle t \rangle g
&\in \bigcap_{r\in\{q_1/2,q_2\}} H_p^{1/2}(\BR_+,L_r(\BR_-^N))\cap L_p(\BR_+,H_r^1(\BR_-^N)), \\
\langle t \rangle \Bh
&\in \bigcap_{r\in\{q_1/2,q_2\}} H_p^{1/2}(\BR_+,L_r(\BR_-^N)^N)\cap L_p(\BR_+,H_r^1(\BR_-^N)^N),
\end{align*}
which satisfy  \eqref{comp:t-wegiht1} and \eqref{comp:t-wegiht2},
the unique solution $(\eta,\Bu,\Fp)$ of \eqref{lin-eq:1} obtained in Lemma $\ref{lem:mr}$ satisfies
for $q\in \{q_1,q_2\}$
\begin{align}\label{3d-est:1}
&\|(\eta,\Bu)\|_{E_{p,q}(\langle t \rangle^{1/3})} \notag \\
&\leq L_2\sum_{r\in\{q_1/2,q_2\}} \Big(\|(\eta_0,\Bu_0)\|_{I_{r,p}}+\|(d,\Bf,g,\Fg,\Bh)\|_{F_{p,r}(\langle t \rangle)}\Big)
\end{align}
and
\begin{align}\label{3d-est:2}
&\|\langle t\rangle^{2/3}  \nabla \pd_t\CE_N\eta\|_{L_p(\BR_+,H_q^1(\BR_-^N)^N)} 
+\|\langle t\rangle^{2/3}\nabla\CE_N\eta\|_{L_p(\BR_+,H_q^2(\BR_-^N)^N)} \notag \\
&+\|\langle t\rangle^{2/3}\nabla\Bu\|_{L_p(\BR_+,H_q^1(\BR_-^N)^{N\times N})} \notag \\
&\leq L_3\sum_{r\in\{q_1/2,q_2\}} \Big(\|(\eta_0,\Bu_0)\|_{I_{r,p}}+\|(d,\Bf,g,\Fg,\Bh)\|_{F_{p,r}(\langle t \rangle)}\Big)
\end{align}
with positive constants $L_2$ and $L_3$,
where $\CE_N$ is given by \eqref{ext:eta}.
\end{prp}

\begin{proof}
The proof of \eqref{3d-est:1} is similar to one of Proposition \ref{prp:linear4d},
so that the detailed proof of \eqref{3d-est:1} may be omitted.

In what follows, we suppose  $q\in\{q_1,q_2\}$ and prove \eqref{3d-est:2}.
We first estimate $\nabla\CE_N\eta$.
Let us use the decomposition given by \eqref{decomp:1}--\eqref{decomp:3} in the following.
Notice that \eqref{est:case1-2}--\eqref{est:local} 
holds also for $N=3$.
Lemma \ref{lem:ext} \eqref{lem:ext-3} with $m=3$ yields
\begin{equation*}
\|\nabla\CE_N\eta^i(t)\|_{H_q^2(\BR_-^N)}
\leq C\|\eta^i(t)\|_{W_q^{3-1/q}(\BR^{N-1})} \quad (i=1,2,3).
\end{equation*}
This implies
\begin{align*}
\|\langle t \rangle \nabla\CE_N\eta^i\|_{L_p(\BR_+,H_q^2(\BR_-^N)^N)}
&\leq C\|\langle t \rangle \eta^i\|_{L_p(\BR_+,W_q^{3-1/q}(\BR^{N-1}))} \quad (i=1,2), \\
\|\nabla\CE_N\eta^3\|_{L_p((0,3),H_q^2(\BR_-^N)^N)}
&\leq C\|\eta^3\|_{L_p((0,3),W_q^{3-1/q}(\BR^{N-1}))},
\end{align*}
which, combined with \eqref{est:case1-2}, \eqref{est:case2-2}, and \eqref{est:local},
furnishes  for $i=1,2$
\begin{align}
&\|\langle t \rangle \nabla\CE_N\eta^i\|_{L_p(\BR_+,H_q^2(\BR_-^{N})^N)} \notag \\
&\leq C \sum_{r\in \{q_1/2,q_2\}}\Big(\|(\eta_0,\Bu_0)\|_{I_{r,p}}
+\|(d,\Bf,\Bg,g,\Bh)\|_{F_{p,r}(\langle t \rangle)}\Big), \label{est-3d-eta12} \\
&\|\nabla\CE_N\eta^3\|_{L_p((0,3),H_q^2(\BR_-^{N})^N)} \notag \\
&\leq C \sum_{r\in \{q_1/2,q_2\}}\Big(\|(\eta_0,\Bu_0)\|_{I_{r,p}}
+\|(d,\Bf,\Bg,g,\Bh)\|_{F_{p,r}(\langle t \rangle)}\Big). \label{est-3d-eta3}
\end{align}

We next estimate $\nabla\CE_N\eta^3$ for $t\geq 2$.
Recall $S_1(t)$, $S_2(t)$ given by \eqref{dfn:S1S2}.
We have the following formula:
\begin{equation*}
\eta^3(t)=\int_0^t S_1(t-\tau)(\gamma_3\eta^2(\tau),\gamma_3\Bu^2(\tau))\intd \tau.
\end{equation*}
This gives us
\begin{align*}
\nabla\CE_N\eta^3(t)
&=\left(\int_0^{t/2}+\int_{t/2}^{t-1}+\int_{t-1}^t\right)
\nabla\CE_N S_1(t-\tau)
(\gamma_3\eta^2(\tau),\gamma_3\Bu^2(\tau))\intd \tau \\
&=:J_1(t)+J_2(t)+J_3(t).
\end{align*}
Similarly to the estimates of $I_1(t)$ and $I_2(t)$ in the proof of Proposition \ref{prp:duha},
we see from \eqref{sg-est-3d-4} that  
\begin{equation}\label{est-J-12}
\|\langle t\rangle^{2/3}J_k\|_{L_p((2,\infty),H_q^2(\BR_-^N)^N)}
\leq C\|\langle t \rangle(\eta^2,\Bu^2)\|_{L_p(\BR_+,L_{q_1/2}\cap X_q)} \quad (k=1,2).
\end{equation}
Concerning $J_3(t)$, we have by Lemma \ref{lem:ext} \eqref{lem:ext-3} with $m=3$
\begin{equation*}
\|J_3(t)\|_{H_q^2(\BR_-^N)}
\leq C\int_{t-1}^t \|S_1(t-\tau)(\gamma_3\eta^2(\tau),\gamma_3\Bu^2(\tau))\|_{W_q^{3-1/q}(\BR^{N-1})}\intd\tau.
\end{equation*}
Combining this inequality with Proposition \ref{prp:sg} furnishes
\begin{align*}
\|J_3(t)\|_{H_q^2(\BR_-^N)}
&\leq C\int_{t-1}^t e^{\gamma_1(t-\tau)}\|(\eta^2(\tau),\Bu^2(\tau))\|_{D(A_q)}\intd\tau \\
&\leq Ce^{\gamma_1}\int_{t-1}^t \|(\eta^2(\tau),\Bu^2(\tau))\|_{D(A_q)}\intd\tau.
\end{align*}
We thus obtain
\begin{equation*}
\|\langle t \rangle^{2/3} J_3\|_{L_p((2,\infty),H_q^2(\BR\-^N)^N)}
\leq C\|\langle t \rangle (\eta^2,\Bu^2)\|_{L_p(\BR_+,D(A_q))}
\end{equation*}
similarly to the estimate of $I_3(t)$ in the proof of Proposition \ref{prp:duha}.

Summing up the last inequality and \eqref{est-J-12}, we have obtained
\begin{equation*}
\|\langle t \rangle^{2/3}\nabla\CE_N\eta^3\|_{L_p((2,\infty),H_q^2(\BR_-^N)^N)}
\leq C\|\langle t\rangle(\eta^2,\Bu^2)\|_{L_p(\BR_+,(L_{q_1/2}\cap X_q)\cap D(A_q))}.
\end{equation*}
Combining this inequality with \eqref{est:domain-A} and \eqref{est-3d-eta3} yields
\begin{align*}
&\|\langle t \rangle^{2/3}\nabla\CE_N\eta^3\|_{L_p(\BR_+,H_q^2(\BR_-^N)^N)} \\
&\leq C \sum_{r\in \{q_1/2,q_2\}}\Big(\|(\eta_0,\Bu_0)\|_{I_{r,p}}
+\|(d,\Bf,\Bg,g,\Bh)\|_{F_{p,r}(\langle t \rangle)}\Big),
\end{align*}
which, combined with \eqref{est-3d-eta12}, furnishes
\begin{align}\label{est-nab-height}
&\|\langle t \rangle^{2/3}\nabla \CE_N\eta\|_{L_p(\BR_+,H_q^2(\BR_-^N)^N)}  \notag \\
&\leq C \sum_{r\in \{q_1/2,q_2\}}\Big(\|(\eta_0,\Bu_0)\|_{I_{r,p}}
+\|(d,\Bf,\Bg,g,\Bh)\|_{F_{p,r}(\langle t \rangle)}\Big).
\end{align}
Analogously, we have from \eqref{sg-est-3d-5} the desired estimate of $\nabla\Bu$, i.e.,
\begin{align}\label{est-nab-u}
&\|\langle t \rangle^{2/3}\nabla \Bu\|_{L_p(\BR_+,H_q^1(\BR_-^N)^{N\times N})}  \notag \\
&\leq C \sum_{r\in \{q_1/2,q_2\}}\Big(\|(\eta_0,\Bu_0)\|_{I_{r,p}}
+\|(d,\Bf,\Bg,g,\Bh)\|_{F_{p,r}(\langle t \rangle)}\Big).
\end{align}

Let us consider $\nabla\pd_t\CE_N\eta$.
We apply $\CE_N$ to the both side of the first equation of \eqref{lin-eq:1} and use $\CE_N\pd_t \eta=\pd_t\CE_N\eta $
in order to obtain 
\begin{equation*}
\pd_t \CE_N\eta=\CE_Nd+\CE_N (u_N|_{\BR_0^N}) \quad \text{in $\BQ_-$.}
\end{equation*}
This implies
\begin{equation*}
\nabla \pd_t \CE_N\eta=\nabla\CE_Nd+\nabla\CE_N (u_N|_{\BR_0^N}) \quad \text{in $\BQ_-$.}
\end{equation*}
Lemma \ref{lem:ext} \eqref{lem:ext-3} with $m=2$ yields
\begin{equation*}
\|\nabla\CE_N d\|_{H_q^1(\BR_-^N)} \leq C\|d\|_{W_q^{2-1/q}(\BR^{N-1})}, 
\end{equation*}
while Lemma \ref{lem:ext} \eqref{lem:ext-1} with $m\in\{1,2\}$ yields
\begin{align*}
\|\nabla\CE_N (u_N|_{\BR_0^N})\|_{H_q^1(\BR_-^N)}
\leq C\|\nabla u_N\|_{H_q^1(\BR_-^N)}.
\end{align*}
Thus
\begin{align*}
&\|\langle t \rangle^{2/3}\nabla \pd_t \CE_N\eta\|_{L_p(\BR_+,H_q^1(\BR_-^N)^N)} \\
&\leq C\Big( \|\langle t \rangle^{2/3} d\|_{L_p(\BR_+,W_q^{2-1/q}(\BR^{N-1}))}
+\|\langle t \rangle^{2/3} \nabla u_N\|_{L_p(\BR_+,H_q^1(\BR_-^N)^N)}\Big),
\end{align*}
which, combined with \eqref{est-nab-u}, furnishes
\begin{align*}
&\|\langle t \rangle^{2/3}\nabla\pd_t\CE_N \eta\|_{L_p(\BR_+,H_q^1(\BR_-^N)^N)}  \notag \\
&\leq C \sum_{r\in \{q_1/2,q_2\}}\Big(\|(\eta_0,\Bu_0)\|_{I_{r,p}}
+\|(d,\Bf,\Bg,g,\Bh)\|_{F_{p,r}(\langle t \rangle)}\Big).
\end{align*}
Combining this inequality with \eqref{est-nab-height} and  \eqref{est-nab-u}
yields \eqref{3d-est:2}.
This completes the proof of Proposition \ref{prp:linear3d}.
\end{proof}

In the case $N=3$, we further obtain

\begin{prp}\label{prp:linear3d-v2}
Suppose that the same assumption as in Proposition $\ref{prp:linear3d}$ holds.
Let 
\begin{equation*}
\langle t \rangle^{1/3} d\in \bigcap_{r\in\{q_1/2,q_2\}}L_\infty(\BR_+,W_r^{1-1/r}(\BR^{N-1}))
\end{equation*}
additionally. Then the solution $\eta$ of \eqref{lin-eq:1} satisfies
\begin{align*}
&\|\langle t \rangle^{1/3}\pd_t\CE_N\eta\|_{L_\infty(\BR_+,H_2^1(\BR_-^N))} \notag \\
&\leq L_4 
\sum_{r\in\{q_1/2,q_2\}} \Big(\|(\eta_0,\Bu_0)\|_{I_{r,p}}+\|(d,\Bf,g,\Fg,\Bh)\|_{F_{p,r}(\langle t \rangle)} \notag \\
&+\|\langle t \rangle^{1/3}d\|_{L_\infty(\BR_+,W_r^{1-1/r}(\BR^{N-1}))}\Big)
\end{align*}
with some positive constant $L_4$,
where $\CE_N$ is given by \eqref{ext:eta}.
\end{prp}

\begin{proof}
We divide the proof into four steps
and use the decomposition given by \eqref{decomp:1}--\eqref{decomp:3} in the following.

{\bf Step 1}.
In this step, we prove for $z\in \bigcap_{r\in\{q_1/2,q_2\}}W_r^{1-1/r}(\BR^{N-1})$
\begin{equation}\label{z-est-2}
\|\CE_N z\|_{H_2^1(\BR_-^N)} \leq C\sum_{r\in\{q_1/2,q_2\}}\|z\|_{W_r^{1-1/r}(\BR^{N-1})}.
\end{equation}

We see by Lemma \ref{lem:ext-interp} with $s=q_1/2$ that 
\begin{equation}\label{Ez-est-1}
\|\CE_N z\|_{L_2(\BR_-^N)}\leq C\Big(\|z\|_{L_{q_1/2}(\BR^{N-1})}+\|z\|_{L_2(\BR^{N-1})}\Big).
\end{equation}
Since $q_1/2<2<q_2$, the standard interpolation inequality yields
\begin{equation*}
\|z\|_{L_2(\BR^{N-1})}\leq C\|z\|_{L_{q_1/2}(\BR^{N-1})}^{1-\theta}\|z\|_{L_{q_2}(\BR^{N-1})}^\theta
\end{equation*}
for some $\theta\in(0,1)$,
which, combined with $a^{1-\theta}b^\theta\leq C(a+b)$ for $a,b\geq 0$, furnishes
\begin{equation*}
\|z\|_{L_2(\BR^{N-1})}\leq C\Big(\|z\|_{L_{q_1/2}(\BR^{N-1})}+\|z\|_{L_{q_2}(\BR^{N-1})}\Big).
\end{equation*}
This inequality together with \eqref{Ez-est-1} gives us
\begin{equation*}
\|\CE_N z\|_{L_2(\BR_-^N)}\leq C\sum_{r\in\{q_1/2,q_2\}}\|z\|_{L_r(\BR^{N-1})},
\end{equation*}
which, combined with $\|z\|_{L_r(\BR^{N-1})}\leq \|z\|_{W_r^{1-1/r}(\BR^{N-1})}$, furnishes
\begin{equation}\label{z-est-1}
\|\CE_N z\|_{L_2(\BR_-^N)}\leq C\sum_{r\in\{q_1/2,q_2\}}\|z\|_{W_r^{1-1/r}(\BR^{N-1})}.
\end{equation}

We use Lemma \ref{lem:ext} \eqref{lem:ext-3} with $m=1$ in order to obtain
\begin{equation*}
\|\nabla \CE_N z\|_{L_2(\BR_-^N)}\leq C\|z\|_{W_2^{1-1/2}(\BR^{N-1})}.
\end{equation*}
Since $q_1/2<2<q_2$, it follows from Lemma \ref{lem:int-p} that
\begin{equation*}
\|z\|_{W_2^{1-1/2}(\BR^{N-1})}
\leq C\|z\|_{W_{q_1/2}^{1-1/(q_1/2)}(\BR^{N-1})}^{1-\theta}\|z\|_{W_{q_2}^{1-1/q_2}(\BR^{N-1})}^\theta
\end{equation*}
for some $\theta\in(0,1)$. Thus
\begin{equation*}
\|\nabla \CE_N z\|_{L_2(\BR_-^N)}\leq C\|z\|_{W_{q_1/2}^{1-1/(q_1/2)}(\BR^{N-1})}^{1-\theta}\|z\|_{W_{q_2}^{1-1/q_2}(\BR^{N-1})}^\theta,
\end{equation*}
which, combined with $a^{1-\theta}b^\theta\leq C(a+b)$ for $a,b\geq 0$,
furnishes
\begin{equation*}
\|\nabla \CE_N z\|_{L_2(\BR_-^N)}\leq C\Big(\|z\|_{W_{q_1/2}^{1-1/(q_1/2)}(\BR^{N-1})}
+\|z\|_{W_{q_2}^{1-1/q_2}(\BR^{N-1})}\Big).
\end{equation*}
This inequality together with \eqref{z-est-1} gives us \eqref{z-est-2}.

Let us consider the case $z\in\{u_N^1|_{\BR_0^N},u_N^2|_{\BR_0^N},u_N^3|_{\BR_0^N}\}$ in \eqref{z-est-2}.
By the trace theorem, we obtain  for $i=1,2,3$
\begin{equation}\label{z-est-3}
\|\CE_N(u_N^i|_{\BR_0^N})\|_{H_2^1(\BR_-^N)} \leq C\sum_{r\in\{q_1/2,q_2\}}\|u_N^i\|_{H_r^{1}(\BR_-^{N})}.
\end{equation}

{\bf Step 2}: Estimate $\eta^4:=\eta^1+\eta^2$.
By the first equations of \eqref{decomp:1} and \eqref{decomp:2}, we see that $\eta^4$ satisfies
\begin{equation*}
\pd_t\eta^4=d-\gamma_3\eta^2+u_N^1+u_N^2 \quad \text{on $\BQ_0$.}
\end{equation*}
One applies $\CE_N$ to the both side of this equation and uses $\CE_N\pd_t\eta^4=\pd_t\CE_N\eta^4$
in order to obtain
\begin{equation*}
\pd_t\CE_N\eta^4 = 
\CE_N d -\gamma_3\CE_N \eta^2+\CE_N(u_N^1|_{\BR_0^N})+\CE_N(u_N^2|_{\BR_0^N}),
\end{equation*}
which, combined with \eqref{z-est-2} for $z\in\{d,\eta^2\}$
and \eqref{z-est-3} for $i=1,2$, furnishes
\begin{align}\label{est-eta1-1}
&\|\langle t\rangle^{1/3}\pd_t\CE_N\eta^4\|_{L_\infty(\BR_+,H_2^1(\BR_-^N))} \notag \\
&\leq C\sum_{r\in\{q_1/2,q_2\}}
\Big(\|\langle t \rangle ^{1/3}d\|_{L_\infty(\BR_+,W_r^{1-1/r}(\BR^{N-1}))} 
+\|\langle t \rangle ^{1/3}\eta^2 \|_{L_\infty(\BR_+,W_r^{1-1/r}(\BR^{N-1}))} \notag \\
& +\|\langle t \rangle ^{1/3}u_N^1\|_{L_\infty(\BR_+,H_r^{1}(\BR_-^{N}))}
+\|\langle t \rangle ^{1/3}u_N^2\|_{L_\infty(\BR_+,H_r^{1}(\BR_-^{N}))}\Big).
\end{align}

By Lemma \ref{lem:H^1-embed}
\begin{equation*}
\|\langle t \rangle^{1/3} \eta^2\|_{L_\infty(\BR_+,W_r^{1-1/r}(\BR^{N-1}))}
\leq C\|\langle t \rangle^{1/3} \eta^2\|_{H_p^1(\BR_+,W_r^{1-1/r}(\BR^{N-1}))}
\end{equation*}
for $r\in\{q_1/2,q_2\}$, which, combined with \eqref{est:case2-2}, furnishes
\begin{align}\label{est-eta1-1/3-weight}
&\sum_{r\in \{q_1/2,q_2\}}\|\langle t \rangle^{1/3} \eta^2\|_{L_\infty(\BR_+,W_r^{1-1/r}(\BR^{N-1}))}  \notag  \\
&\leq C \sum_{r\in \{q_1/2,q_2\}}\Big(\|(\eta_0,\Bu_0)\|_{I_{r,p}}
+\|(d,\Bf,\Bg,g,\Bh)\|_{F_{p,r}(\langle t \rangle)}\Big).
\end{align}
On the other hand, 
Lemma \ref{lem:fund-embed} \eqref{lem:fund-embed-2} and Lemma \ref{lem:time-sp} \eqref{lem:time-sp-2}  show that
\begin{align*}
&\|\langle t \rangle^{1/3} u_N^i\|_{L_\infty(\BR_+,H_r^{1}(\BR_-^{N}))} 
\leq C\|\langle t \rangle^{1/3} u_N^i\|_{L_\infty(\BR_+,B_{r,p}^{2-2/p}(\BR_-^{N}))} \notag \\
&\leq C\Big( \|\langle t\rangle^{1/3}\pd_t u_N^i\|_{L_p(\BR_+,L_r(\BR_-^{N}))}
+\|\langle t\rangle^{1/3} u_N^i\|_{L_p(\BR_+,H_r^{2}(\BR_-^{N}))} \Big) 
\end{align*}
for $i=1,2$ and $r\in\{q_1/2,q_2\}$, which, combined with \eqref{est:case1-2} and \eqref{est:case2-2}, furnishes
\begin{align*}
&\sum_{r\in \{q_1/2,q_2\}}\|\langle t \rangle^{1/3} u_N^i\|_{L_\infty(\BR_+,H_r^{1}(\BR_-^{N}))}  \notag  \\
&\leq C \sum_{r\in \{q_1/2,q_2\}}\Big(\|(\eta_0,\Bu_0)\|_{I_{r,p}}
+\|(d,\Bf,\Bg,g,\Bh)\|_{F_{p,r}(\langle t \rangle)}\Big).
\end{align*}
This inequality together with \eqref{est-eta1-1} and \eqref{est-eta1-1/3-weight} yields
\begin{align}\label{est-eta1-add1}
&\|\langle t \rangle^{1/3}\pd_t\CE_N\eta^4\|_{L_\infty(\BR_+,H_2^1(\BR_-^N))} \notag \\
&\leq C\sum_{r\in\{q_1/2,q_2\}}
\Big(\|\langle t \rangle^{1/3} d\|_{L_\infty(\BR_+,W_r^{1-1/r}(\BR^{N-1}))}  \notag \\
&+\|(\eta_0,\Bu_0)\|_{I_{r,p}}
+\|(d,\Bf,\Bg,g,\Bh)\|_{F_{p,r}(\langle t \rangle)}\Big).
\end{align}

{\bf Step 3}: Estimate $\eta^3$ local in time.
Let $r\in\{q_1/2,q_2\}$.
Lemma \ref{lem:mr} yields
\begin{align*}
&\|e^{-\gamma_3 t}\pd_t\Bu^3\|_{L_p(\BR_+,L_r(\BR_-^N)^N)}
+\|e^{-\gamma_3 t}\Bu^3\|_{L_p(\BR_+,H_r^2(\BR_-^N)^N)} \\
&\leq C\|(\eta^2,\Bu^2,0,0,0)\|_{F_{p,r}(e^{-\gamma_3 t})} \\
&\leq C\Big(\|\eta^2\|_{L_p(\BR_+,W_r^{2-1/r}(\BR^{N-1}))}+\|\Bu^2\|_{L_p(\BR_+,L_r(\BR_-^N)^N)}\Big),
\end{align*}
which, combined with \eqref{est:case2-2}, furnishes
\begin{align}\label{est-eta3-u3}
&\|e^{-\gamma_3 t}\pd_t\Bu^3\|_{L_p(\BR_+,L_r(\BR_-^N)^N)}+\|e^{-\gamma_3 t}\Bu^3\|_{L_p(\BR_+,H_r^2(\BR_-^N)^N)} \notag \\
&\leq C \sum_{r\in \{q_1/2,q_2\}}\Big(\|(\eta_0,\Bu_0)\|_{I_{r,p}}
+\|(d,\Bf,\Bg,g,\Bh)\|_{F_{p,r}(\langle t \rangle)}\Big).
\end{align}

Lemma \ref{lem:fund-embed} \eqref{lem:fund-embed-2} and
Lemma \ref{lem:time-sp} \eqref{lem:time-sp-2} show that
\begin{align*}
&\|e^{-\gamma_3 t}\Bu^3\|_{L_{\infty}(\BR_+,H_r^1(\BR_-^N)^N)} 
\leq C \|e^{-\gamma_3 t}\Bu^3\|_{L_{\infty}(\BR_+,B_{r,p}^{2-2/p}(\BR_-^N)^N)}\\
&\leq C\Big( \|e^{-\gamma_3 t}\pd_t \Bu^3\|_{L_p(\BR_+,L_r(\BR_-^{N})^N)}
+\|e^{-\gamma_3 t} \Bu^3\|_{L_p(\BR_+,H_r^{2}(\BR_-^{N})^N)} \Big),
\end{align*}
which, combined with \eqref{est-eta3-u3}, furnishes
\begin{align}\label{local-u3-est}
&\|e^{-\gamma_3 t}\Bu^3\|_{L_{\infty}(\BR_+,H_r^1(\BR_-^N)^N)} \notag \\
&\leq C \sum_{r\in \{q_1/2,q_2\}}\Big(\|(\eta_0,\Bu_0)\|_{I_{r,p}}
+\|(d,\Bf,\Bg,g,\Bh)\|_{F_{p,r}(\langle t \rangle)}\Big).
\end{align}

We apply $\CE_N$ to the both side of the first equation of \eqref{decomp:3}
and use $\CE_N\pd_t\eta^3=\pd_t\CE_N\eta^3$ in order to obtain
\begin{equation*}
\pd_t\CE_N\eta^3=\gamma_3\CE_N\eta^2+\CE_N u_N^3 \quad \text{in $\BQ_-$.}
\end{equation*}
Combining this equation with \eqref{z-est-2} for $z=\eta^2$ and \eqref{z-est-3} for $i=3$
furnishes
\begin{align*}
&\|\pd_t\CE_N\eta^3\|_{L_\infty((0,3),H_2^1(\BR_-^N))} \\
&\leq C\sum_{r\in\{q_1/2,q_2\}}
\Big(\|\eta^2\|_{L_\infty((0,3),W_r^{1-1/r}(\BR^{N-1}))}+\|u_N^3\|_{L_\infty((0,3),H_r^1(\BR_-^N))}\Big).
\end{align*}
From this estimate, \eqref{est-eta1-1/3-weight}, and \eqref{local-u3-est}, we obtain
\begin{align}\label{est:local-in-time}
&\|\pd_t\CE_N\eta^3\|_{L_\infty((0,3),H_2^1(\BR_-^N))} \notag \\
&\leq C \sum_{r\in \{q_1/2,q_2\}}\Big(\|(\eta_0,\Bu_0)\|_{I_{r,p}}
+\|(d,\Bf,\Bg,g,\Bh)\|_{F_{p,r}(\langle t \rangle)}\Big).
\end{align}

{\bf Step 4}: Estimate $\eta^3$ for large time.
We use the formula
\begin{equation*}
(\eta^3(t),\Bu^3(t)) =\int_0^t S(t-\tau)(\gamma_3 \eta^2(\tau),\gamma_3 \Bu^2(\tau))\intd \tau, 
\end{equation*}
which gives us
\begin{equation*}
\pd_t\eta^3(t)
=\gamma_3\eta^2(t)+\int_0^t \pd_t S_1(t-\tau)(\gamma_3\eta^2(\tau),\gamma_3\Bu^2(\tau))\intd \tau.
\end{equation*}
Applying the extension operator $\CE_N$ to the both side of this equation
and using $\CE_N\pd_t=\pd_t\CE_N$ yield
\begin{align}\label{decomp:eta3-J}
\pd_t\CE_N\eta^3(t)
&=\gamma_3\CE_N \eta^2(t)+\int_0^t \pd_t \CE_N  S_1(t-\tau)(\gamma_3 \eta^2(\tau),\gamma_3 \Bu^2(\tau))\intd \tau \notag \\
&=:\gamma_3\CE_N \eta^2(t)+J(t).
\end{align}

By \eqref{z-est-2} with $z=\eta^2$ 
\begin{align*}
&\|\langle t \rangle^{1/3}\CE_N\eta^2\|_{L_\infty((2,\infty),H_2^1(\BR_-^N))} \\
&\leq C\sum_{r\in\{q_1/2,q_2\}}\|\langle t \rangle^{1/3}\eta^2 \|_{L_\infty((2,\infty),W_r^{1-1/r}(\BR^{N-1}))},
\end{align*}
which, combined with \eqref{est-eta1-1/3-weight}, furnishes
\begin{align}\label{est:eta2-weight}
&\|\langle t \rangle^{1/3}\CE_N\eta^2\|_{L_\infty((2,\infty),H_2^1(\BR_-^N))} \notag \\
&\leq C \sum_{r\in \{q_1/2,q_2\}}\Big(\|(\eta_0,\Bu_0)\|_{I_{r,p}}
+\|(d,\Bf,\Bg,g,\Bh)\|_{F_{p,r}(\langle t \rangle)}\Big).
\end{align}

Let $t\geq 2$ and let us write
\begin{align*}
J(t)
&=\left(\int_0^{t/2}+\int_{t/2}^{t-1}+\int_{t-1}^t\right)
\pd_t \CE_N S_1(t-\tau)(\gamma_3\eta^2(\tau),\gamma_3\Bu^2(\tau))\intd \tau \\
&=:J_1(t)+J_2(t)+J_3(t).
\end{align*}

We first estimate $J_1(t)$.
By Proposition \ref{prp:decay-L2}, we see that
\begin{align*}
&\|J_1(t)\|_{H_2^1(\BR_-^N)}\leq C
\int_0^{t/2}(t-\tau)^{-\frac{1}{3}-\delta_0}\|(\eta^2(\tau),\Bu^2(\tau))\|_{L_{q_1/2}\cap X_2}\intd \tau \\
&\leq C t^{-\frac{1}{3}-\delta_0}\int_0^{t/2}\|(\eta^2(\tau),\Bu^2(\tau))\|_{L_{q_1/2}\cap X_2}\intd \tau \\
&\leq Ct^{-\frac{1}{3}-\delta_0} \left(\int_0^{t/2}\langle \tau\rangle^{-p'}\intd \tau\right)^{1/p'}
\left(\int_0^{t/2}\left(\langle \tau\rangle \|(\eta^2(\tau),\Bu^2(\tau))\|_{L_{q_1/2}\cap X_2} \right)^p\intd\tau\right)^{1/p},
\end{align*}
where $p'=p/(p-1)$. Thus
\begin{equation}\label{est:J1-3d}
\|\langle t\rangle^{1/3}J_1\|_{L_\infty((2,\infty),H_2^1(\BR_-^N))}
\leq C\|\langle t \rangle (\eta^2,\Bu^2)\|_{L_p(\BR_+,L_{q_1/2}\cap X_2)}.
\end{equation}

We next estimate $J_2(t)$.
By Proposition \ref{prp:decay-L2}, we see that
\begin{align}\label{est:J2-3d-1}
&\|J_2(t)\|_{H_2^1(\BR_-^N)}\leq C
\int_{t/2}^{t-1}(t-\tau)^{-\frac{1}{3}-\delta_0}\|(\eta^2(\tau),\Bu^2(\tau))\|_{L_{q_1/2}\cap X_2}\intd \tau \notag \\
&\leq C t^{-1}\int_{t/2}^{t-1}(t-\tau)^{-\frac{1}{3}-\delta_0}\langle \tau\rangle\|(\eta^2(\tau),\Bu^2(\tau))
\|_{L_{q_1/2}\cap X_2}\intd \tau \notag \\
&\leq C t^{-1}\left(\int_{t/2}^{t-1}(t-\tau)^{-(\frac{1}{3}+\delta_0)p'}\intd\tau\right)^{1/p'} \notag \\
&\times \left(\int_{t/2}^{t-1} 
\left(\langle \tau\rangle\|(\eta^2(\tau),\Bu^2(\tau))\|_{L_{q_1/2}\cap X_2}\right)^p\intd \tau \right)^{1/p}.
\end{align}
Since $1/p'=1-1/p\geq 1-1/p_0=30/31$, it holds that
\begin{equation*}
\left(\frac{1}{3}+\delta_0\right)p'=\left(\frac{1}{3}+\frac{1}{30}\right)p'\leq \left(\frac{1}{3}+\frac{1}{30}\right)\frac{31}{30}<1.
\end{equation*}
Thus
\begin{equation*}
\int_{t/2}^{t-1}(t-\tau)^{-(\frac{1}{3}+\delta_0)p'}\intd\tau
\leq C t^{1-(\frac{1}{3}+\delta_0)p'},
\end{equation*}
which implies
\begin{equation*}
\left(\int_{t/2}^{t-1}(t-\tau)^{-(\frac{1}{3}+\delta_0)p'}\intd\tau\right)^{1/p'}
\leq C t^{1/p'-(\frac{1}{3}+\delta_0)}\leq Ct^{1-(\frac{1}{3}+\delta_0)}.
\end{equation*}
Combining this with \eqref{est:J2-3d-1} furnishes 
\begin{equation*}
\|J_2(t)\|_{H_2^1(\BR_-^N)}\leq C
t^{-\frac{1}{3}-\delta_0}\|\langle t \rangle (\eta^2,\Bu^2)\|_{L_p(\BR_+,L_{q_1/2}\cap X_2)},
\end{equation*}
and thus
\begin{equation}\label{est:J2-3d-2}
\|\langle t\rangle^{1/3}J_2\|_{L_\infty((2,\infty),H_2^1(\BR_-^N))}
\leq C\|\langle t \rangle (\eta^2,\Bu^2)\|_{L_p(\BR_+,L_{q_1/2}\cap X_2)}.
\end{equation}

We finally estimate $J_3(t)$. Noting $\pd_t\CE_N=\CE_N\pd_t$, we have
\begin{equation*}
\|J_3(t)\|_{H_2^1(\BR_-^N)}
\leq C\int_{t-1}^t \|\CE_N
\pd_t S_1(t-\tau)(\eta^2(\tau),\Bu^2(\tau))\|_{H_2^1(\BR_-^N)}\intd\tau.
\end{equation*}
Combining this with \eqref{z-est-2} furnishes
\begin{equation*}
\|J_3(t)\|_{H_2^1(\BR_-^N)}
\leq C\sum_{r\in\{q_1/2,q_2\}}\int_{t-1}^t \|\pd_t S_1(t-\tau)(\eta^2(\tau),\Bu^2(\tau))\|_{W_r^{1-1/r}(\BR^{N-1})}\intd\tau,
\end{equation*}
and thus Proposition \ref{prp:sg} shows that
\begin{align*}
\|J_3(t)\|_{H_2^1(\BR_-^N)}
&\leq C\sum_{r\in\{q_1/2,q_2\}}\int_{t-1}^te^{\gamma_3 (t-\tau)}\|(\eta_2(\tau),\Bu^2(\tau))\|_{D(A_r)}\intd\tau \\
&\leq Ce^{\gamma_3 }\sum_{r\in\{q_1/2,q_2\}}\int_{t-1}^t\|(\eta_2(\tau),\Bu^2(\tau))\|_{D(A_r)}\intd\tau.
\end{align*}
Since $\langle t \rangle \leq C\langle \tau \rangle$ for $t-1\leq \tau\leq t$,
it follows from the last inequality that
\begin{align*}
&\langle t \rangle^{1/3}\|J_3(t)\|_{H_2^1(\BR_-^N)} \\
&\leq C\sum_{r\in\{q_1/2,q_2\}}
\int_{t-1}^t\langle \tau\rangle^{1/3}\|(\eta_2(\tau),\Bu^2(\tau))\|_{D(A_r)}\intd\tau \\ 
&\leq C\langle t \rangle^{-2/3}\sum_{r\in\{q_1/2,q_2\}}
\int_{t-1}^t\langle \tau\rangle\|(\eta_2(\tau),\Bu^2(\tau))\|_{D(A_r)}\intd\tau \\ 
&\leq C\langle t \rangle^{-2/3}\sum_{r\in\{q_1/2,q_2\}}
\left(\int_{t-1}^t\intd\tau\right)^{1/p'}
\left(\int_{t-1}^t\big(\langle \tau\rangle\|(\eta_2(\tau),\Bu^2(\tau))\|_{D(A_r)}\big)^p\intd\tau\right)^{1/p}.
\end{align*}
This gives us
\begin{equation}\label{est-J3-3d}
\|\langle t \rangle^{1/3}J_3\|_{L_\infty((2,\infty),H_2^1(\BR_-^N))}
\leq C\sum_{r\in\{q_1/2,q_2\}}\|\langle t\rangle (\eta^2,\Bu^2)\|_{L_p(\BR_+,D(A_r))}.
\end{equation}

Since $q_1/2<2<q_2$, it holds by Lemma \ref{lem:int-p-v2} that
\begin{equation*}
\|\langle t \rangle(\eta^2,\Bu^2)\|_{L_p(\BR_+,X_2)}
\leq C\|\langle t \rangle(\eta^2,\Bu^2)\|_{L_p(\BR_+,X_{q_1/2})}^{1-\theta}
\|\langle t \rangle(\eta^2,\Bu^2)\|_{L_p(\BR_+,X_{q_2})}^\theta
\end{equation*}
for some $\theta\in(0,1)$.
Thus Young's inequality, $a^{1-\theta}b^\theta\leq C(a+b)$ for $a,b\geq 0$, yields
\begin{equation*}
\|\langle t \rangle(\eta^2,\Bu^2)\|_{L_p(\BR_+,X_2)}
\leq C\sum_{r\in\{q_1/2,q_2\}} \|\langle t \rangle(\eta^2,\Bu^2)\|_{L_p(\BR_+, X_r)}.
\end{equation*}
From this, \eqref{est:J1-3d}, \eqref{est:J2-3d-2}, and \eqref{est-J3-3d}, we have
\begin{align*}
\|\langle t\rangle^{1/3}J\|_{L_\infty((2,\infty),H_2^1(\BR_-^N))}
\leq C\sum_{r\in\{q_1/2,q_2\}} \|\langle t \rangle(\eta^2,\Bu^2)\|_{L_p(\BR_+,D(A_r))},
\end{align*}
where we have used the fact that $D(A_{q_1/2})$ and $D(A_r)$ are continuously embeded into
 $L_{q_1/2}$ and $X_r$, respectively, see the previous section for the function spaces.
Combining the last inequality with \eqref{est:case2-2} furnishes
\begin{align*}
&\|\langle t \rangle^{1/3}J\|_{L_\infty((2,\infty),H_2^1(\BR_-^N))} \\
&\leq C \sum_{r\in \{q_1/2,q_2\}}\Big(\|(\eta_0,\Bu_0)\|_{I_{r,p}}
+\|(d,\Bf,\Bg,g,\Bh)\|_{F_{p,r}(\langle t \rangle)}\Big).
\end{align*}

Summing up the last inequality and \eqref{est:eta2-weight}, we obtain from \eqref{decomp:eta3-J}
\begin{align*}
&\|\langle t \rangle^{1/3}\pd_t\CE_N\eta^3\|_{L_\infty((2,\infty),H_2^1(\BR_-^N))} \\
&\leq 
C\sum_{r\in \{q_1/2,q_2\}}\Big(\|(\eta_0,\Bu_0)\|_{I_{r,p}}
+\|(d,\Bf,\Bg,g,\Bh)\|_{F_{p,r}(\langle t \rangle)}\Big).
\end{align*}
From this and \eqref{est:local-in-time}, we have
\begin{align*}
&\|\langle t \rangle^{1/3}\pd_t\CE_N\eta^3\|_{L_\infty(\BR_+,H_2^1(\BR_-^N))} \\
&\leq 
C\sum_{r\in \{q_1/2,q_2\}}\Big(\|(\eta_0,\Bu_0)\|_{I_{r,p}}
+\|(d,\Bf,\Bg,g,\Bh)\|_{F_{p,r}(\langle t \rangle)}\Big).
\end{align*}
This inequality together with \eqref{est-eta1-add1} shows that the desired inequality holds.
This completes the proof of Proposition \ref{prp:linear3d-v2}.
\end{proof}

\section{Estimates of nonlinear terms}\label{sec:nonl-terms}

\subsection{Function spaces}\label{subsec:fspaces}
Let us define
\begin{align*}
Z_1(p,q)
&=H_p^1(\BR_+,W_q^{2-1/q}(\BR^{N-1}))\cap L_p(\BR_+,W_q^{3-1/q}(\BR^{N-1})), \\
Z_2(p,q)
&=H_p^1(\BR_+,L_q(\BR_-^N)^N)\cap L_p(\BR_+,H_q^2(\BR_-^N)^N).
\end{align*}
Let $\delta$ be a positive number and recall $\langle t\rangle = \sqrt{1+t^2}$. We define
\begin{multline*}
Z_1^\delta (p,q) = \{\eta\in Z_1(p,q) :
\langle t \rangle^\delta \pd_t \eta\in L_p(\BR_+,W_q^{2-1/q}(\BR^{N-1})), \\
\langle t \rangle^\delta \eta \in L_p(\BR_+,W_q^{3-1/q}(\BR^{N-1}))\}
\end{multline*}
with
\begin{equation*}
\|\eta\|_{Z_1^\delta(p,q)}=\|\langle t \rangle^\delta \pd_t \eta\|_{L_p(\BR_+,W_q^{2-1/q}(\BR^{N-1}))}
+\|\langle t \rangle^\delta \eta\|_{L_p(\BR_+,W_q^{3-1/q}(\BR^{N-1}))},
\end{equation*}
while we define
\begin{multline*}
Z_2^\delta(p,q)=\{\Bu\in Z_2(p,q) :
\langle t \rangle^\delta \pd_t \Bu\in L_p(\BR_+,L_q(\BR_-^N)^N),  \\
\langle t \rangle^\delta \Bu\in L_p(\BR_+,H_q^2(\BR_-^N)^N)\}
\end{multline*}
with 
\begin{equation*}
\|\Bu\|_{Z_2^\delta(p,q)}
=\|\langle t \rangle^\delta \pd_t \Bu\|_{L_p(\BR_+,L_q(\BR_-^N)^N)}
+\|\langle t \rangle^\delta \Bu\|_{L_p(\BR_+,H_q^2(\BR_-^N)^N)}. 
\end{equation*}
Recall $\CE_N$ given by \eqref{ext:eta}. Then
\begin{multline*}
A_1^\delta(p,q)=\{\eta : \langle t \rangle^\delta\nabla\pd_t\CE_N \eta \in  L_p(\BR_+,H_q^1(\BR_-^N)^N), \\
\langle t \rangle^\delta \nabla\CE_N \eta \in L_p(\BR_+,H_q^2(\BR_-^N)^N)\}
\end{multline*}
with
\begin{align*}
\|\eta\|_{A_1^\delta(p,q)}
=\|\langle t \rangle^\delta \nabla\pd_t\CE_N \eta\|_{L_p(\BR_+,H_q^1(\BR_-^N)^N)}
+\|\langle t \rangle^\delta\nabla\CE_N\eta\|_{L_p(\BR_+,H_q^2(\BR_-^N)^N)},
\end{align*}
and also
\begin{align*}
A_2^\delta (p,q)
&=\{\Bu : \langle t \rangle^\delta\nabla\Bu \in L_p(\BR_+,H_q^1(\BR_-^N)^{N\times N})\}, \\
\|\Bu\|_{A_2^\delta(p,q)}
&=\|\langle t \rangle^\delta \nabla\Bu\|_{L_p(\BR_+,H_q^1(\BR_-^N)^{N\times N})}.
\end{align*}
We finally set
\begin{align*}
B^\delta(p,q)
&=\{\eta : \langle t \rangle^\delta\pd_t\CE_N\eta\in L_p(\BR_+,H_q^1(\BR_-^N))\}, \\
\|\eta\|_{B^\delta(p,q)}
&=\|\langle t \rangle^\delta \pd_t\CE\eta\|_{L_p(\BR_+,H_q^1(\BR_-^N))}.
\end{align*}

Let us recall the assumption for $p$ and $q$.

\begin{asm}\label{asm:p-q}
Let $q_0$ and $p_0$ be as in Propositions $\ref{prp:decay}$ and $\ref{prp:linear4d}$, respectively.
The $p$, $q_1$, and $q_2$ satisfy
\begin{equation*}
p_0\leq p<\infty, \quad 2<q_1\leq 2+q_0, \quad N<q_2<\infty,
\end{equation*}
where $N=3$ or $N\geq 4$.
\end{asm}

For $p$, $q_1$, and $q_2$ satisfying Assumption \ref{asm:p-q},
we define the underlying space
\begin{align*}
K_{p,q_1,q_2}^N &= K_{p,q_1,q_2;1}^N\times K_{p,q_1,q_2:2}^N, \\
\|(\eta,\Bu)\|_{K_{p,q_1,q_2}^N} &= \|\eta\|_{K_{p,q_1,q_2;1}^N}+\|\Bu\|_{K_{p,q_1,q_2:2}^N},
\end{align*}
in the following manner.
\begin{itemize}
\item
Let $N\geq 4$. Then
\begin{align*}
K_{p,q_1,q_2;1}^N 
&=Z_1^{1/2}(p,q_1)\cap Z_1^{1/2}(p,q_2), \\
K_{p,q_1,q_2;2}^N
&=Z_2^{1/2}(p,q_1)\cap  Z_2^{1/2}(p,q_2).
\end{align*}
\item
Let $N=3$. Then
\begin{align*}
K_{p,q_1,q_2;1}^N
&=Z_1^{1/3}(p,q_1)\cap Z_1^{1/3}(p,q_2) \cap A_1^{2/3}(p,q_1)\cap A_1^{2/3}(p,q_2)\\ 
&\cap B^{1/3}(\infty,2), \\
K_{p,q_1,q_2;2}^N
&=Z_2^{1/3}(p,q_1)\cap  Z_2^{1/3}(p,q_2) \cap A_2^{2/3}(p,q_1)\cap A_2^{2/3}(p,q_2).
\end{align*}
\end{itemize}

\begin{rmk}\label{rmk:equi-norm-4d}
In the case $N\geq 4$, Lemma $\ref{lem:ext}$ \eqref{lem:ext-2} shows that for $q\in\{q_1,q_2\}$
\begin{align}
C_0^{-1}\|\eta\|_{K_{p,q_1,q_2;1}^N}
&\leq \|\langle t\rangle^{1/2} \pd_t \CE_N\eta\|_{L_p(\BR_+,H_q^2(\BR_-^N))}
\leq C_0\|\eta\|_{K_{p,q_1,q_2;1}^N}, \notag \\ 
C_0^{-1}\|\eta\|_{K_{p,q_1,q_2;1}^N} &\leq 
\|\langle t\rangle^{1/2} \CE_N\eta\|_{L_p(\BR_+,H_q^3(\BR_-^N))}
\leq C_0\|\eta\|_{K_{p,q_1,q_2;1}^N}, \label{rmk:equi-norm-4d-2}
\end{align}
with some positive constant $C_0$.
Here we have used $\pd_t\CE_N\eta=\CE_N\pd_t\eta$
to obtain the first line of \eqref{rmk:equi-norm-4d-2}.
\end{rmk}

\subsection{Fundamental inequalities}

From the embeddings stated in Subsection \ref{subsec:embed},
we have the following two lemmas.

\begin{lem}\label{lem:main-embed3d}
Let $N=3$ and suppose that Assumption $\ref{asm:p-q}$ holds.
Then there exists a positive constant $C$
such that for any $(\eta,\Bu)\in K_{p,q_1,q_2}^N$ 
and for $q\in\{q_1,q_2\}$ the following inequalities hold.
\begin{enumerate}[$(1)$]
\item\label{3d-em-1}
$\|\langle t \rangle^{2/3}\nabla\CE_N\eta\|_{L_\infty(\BR_+,H_q^1(\BR_-^N)^N)}
\leq C\|\eta\|_{K_{p,q_1,q_2;1}^N}$.
\item\label{3d-em-2}
$\|\langle t \rangle^{2/3}\nabla\CE_N\eta\|_{L_\infty(\BR_+,L_\infty(\BR_-^N)^N)}
\leq C\|\eta\|_{K_{p,q_1,q_2;1}^N}$.
\item\label{3d-em-3}
$\|\langle t \rangle^{2/3}\nabla\CE_N\eta\|_{H_p^1(\BR_+,L_\infty(\BR_-^N)^N)}
\leq C\|\eta\|_{K_{p,q_1,q_2;1}^N}$.
\item\label{3d-em-4}
$\|\langle t \rangle^{2/3}\pd_j \nabla \CE_N\eta\|_{H_p^{1/2}(\BR_+,L_q(\BR_-^N)^N)}
\leq C\|\eta\|_{K_{p,q_1,q_2;1}^N}$ for $j=1,\dots,N$.
\item\label{3d-em-5}
$\|\langle t \rangle^{1/3}\Bu\|_{L_\infty(\BR_+,H_q^1(\BR_-^N)^N)}
\leq C\|\Bu\|_{K_{p,q_1,q_2;2}^N}$.
\item\label{3d-em-6}
$\|\langle t \rangle^{1/3}\Bu\|_{L_\infty(\BR_+,L_\infty(\BR_-^N)^N)}
\leq C\|\Bu\|_{K_{p,q_1,q_2;2}^N}$.
\item\label{3d-em-7}
$\|\langle t \rangle^{1/3}\pd_j \Bu\|_{H_p^{1/2}(\BR_+,L_q(\BR_-^N)^{N})}
\leq C\|\Bu\|_{K_{p,q_1,q_2;2}^N}$ for $j=1,\dots,N$.
\end{enumerate}
\end{lem}

\begin{proof}
\eqref{3d-em-1}
Since $\langle t \rangle^{2/3}\nabla\CE_N\eta \in H_p^1(\BR_+,H_q^1(\BR_-^N)^N)$ with
\begin{equation}\label{3d-em-1:eq-1}
\|\langle t \rangle^{2/3}\nabla\CE_N\eta\|_{H_p^1(\BR_+,H_q^1(\BR_-^N)^N)}
\leq \|\eta\|_{K_{p,q_1,q_2;1}^N},
\end{equation}
the desired inequality follows from Lemma \ref{lem:H^1-embed}.

\eqref{3d-em-2}
Combining \eqref{3d-em-1} with Lemma \ref{lem:fund-embed} \eqref{lem:fund-embed-3} 
yields the desired inequality.

\eqref{3d-em-3}
Combining \eqref{3d-em-1:eq-1} with Lemma \ref{lem:fund-embed} \eqref{lem:fund-embed-3} 
yields the desired inequality.

\eqref{3d-em-4}
By Lemma \ref{int-H-1/2} \eqref{int-H-1/2-2}
\begin{align*}
&\|\langle t \rangle^{2/3}\pd_j\nabla\CE_N\eta\|_{H_p^{1/2}(\BR_+,L_q(\BR_-^N)^N)} \notag \\
&\leq C
\Big(\|\langle t \rangle^{2/3}\nabla\CE_N\eta\|_{H_p^{1}(\BR_+,L_q(\BR_-^N)^N)}
+\|\langle t \rangle^{2/3}\nabla\CE_N\eta\|_{L_p(\BR_+,H_q^2(\BR_-^N)^N)}
\Big),
\end{align*}
which furnishes the desired inequality.

\eqref{3d-em-5}
By Lemma \ref{lem:fund-embed} \eqref{lem:fund-embed-2} and Lemma \ref{lem:time-sp} \eqref{lem:time-sp-2},
we see that
\begin{align*}
&\|\langle t \rangle^{1/3}\Bu\|_{L_\infty(\BR_+,H_q^1(\BR_-^N)^N)} 
\leq 
C\|\langle t \rangle^{1/3}\Bu\|_{L_\infty(\BR_+,B_{q,p}^{2-2/p}(\BR_-^N)^N)}  \\
&\leq C\Big(
\|\langle t \rangle^{1/3}\Bu\|_{H_p^1(\BR_+,L_q(\BR_-^N)^N)}
+\|\langle t \rangle^{1/3}\Bu\|_{L_p(\BR_+,H_q^2(\BR_-^N)^N)}\Big).
\end{align*}
This yields the desired inequality.

\eqref{3d-em-6}
Combining \eqref{3d-em-5} with  Lemma \ref{lem:fund-embed} \eqref{lem:fund-embed-3} yields the desired inequality.

\eqref{3d-em-7}
We can prove the desired inequality in the same manner as in \eqref{3d-em-4}.
This completes the proof of Lemma \ref{lem:main-embed3d}.
\end{proof}

\begin{lem}\label{lem:main-embed4d}
Let $N\geq 4$ and suppose that Assumption $\ref{asm:p-q}$ holds. 
Then there exists a positive constant $C$
such that for any $(\eta,\Bu)\in K_{p,q_1,q_2}^N$ and for $q\in\{q_1,q_2\}$ 
the following inequalities hold.
\begin{enumerate}[$(1)$]
\item\label{4d-em-1}
$\|\langle t \rangle^{1/2}\CE_N\eta\|_{L_\infty(\BR_+,H_q^2(\BR_-^N))}
\leq C\|\eta\|_{K_{p,q_1,q_2;1}^N}$.
\item\label{4d-em-2}
$\|\langle t \rangle^{1/2}\CE_N\eta\|_{L_\infty(\BR_+,H_\infty^1(\BR_-^N))}
\leq C\|\eta\|_{K_{p,q_1,q_2;1}^N}$.
\item\label{4d-em-3}
$\|\langle t \rangle^{1/2}\CE_N\eta\|_{H_p^1(\BR_+,H_\infty^1(\BR_-^N))}
\leq C\|\eta\|_{K_{p,q_1,q_2;1}^N}$.
\item\label{4d-em-4}
$\|\langle t \rangle^{1/2}\pd_j\nabla \CE_N\eta\|_{H_p^{1/2}(\BR_+,L_q(\BR_-^N)^N)}
\leq C\|\eta\|_{K_{p,q_1,q_2;1}^N}$ for $j=1,\dots,N$.
\item\label{4d-em-5}
$\|\langle t \rangle^{1/2}\Bu\|_{L_\infty(\BR_+,H_q^1(\BR_-^N)^N)}
\leq C\|\Bu\|_{K_{p,q_1,q_2;2}^N}$.
\item\label{4d-em-6}
$\|\langle t \rangle^{1/2}\Bu\|_{L_\infty(\BR_+,L_\infty(\BR_-^N)^N)}
\leq C\|\Bu\|_{K_{p,q_1,q_2;2}^N}$.
\item\label{4d-em-7}
$\|\langle t \rangle^{1/2}\pd_j \Bu\|_{H_p^{1/2}(\BR_+,L_q(\BR_-^N)^{N})}
\leq C\|\Bu\|_{K_{p,q_1,q_2;2}^N}$ for $j=1,\dots,N$.
\end{enumerate}
\end{lem}

\begin{proof}
We can prove \eqref{4d-em-1}--\eqref{4d-em-3}
by means of \eqref{rmk:equi-norm-4d-2} instead of \eqref{3d-em-1:eq-1}
in the same manner as in the proof of  \eqref{3d-em-1}--\eqref{3d-em-3} of Lemma \ref{lem:main-embed3d}.
By Lemma \ref{int-H-1/2} \eqref{int-H-1/2-2}
\begin{align*}
&\|\langle t \rangle^{1/2}\pd_j\nabla \CE_N\eta\|_{H_p^{1/2}(\BR_+,L_q(\BR_-^N)^N)} \notag \\
&\leq C
\Big(\|\langle t \rangle^{1/2}\nabla \CE_N\eta\|_{H_p^{1}(\BR_+,L_q(\BR_-^N)^N)}
+\|\langle t \rangle^{1/2}\nabla\CE_N\eta\|_{L_p(\BR_+,H_q^2(\BR_-^N)^N)}
\Big),
\end{align*}
which, combined with \eqref{rmk:equi-norm-4d-2}, furnishes the desired inequality of \eqref{4d-em-4}.
The proof of \eqref{4d-em-5}--\eqref{4d-em-7} is similar to 
one of \eqref{3d-em-5}--\eqref{3d-em-7} of Lemma \ref{lem:main-embed3d}.
This completes the proof of Lemma \ref{lem:main-embed4d}.
\end{proof}

The next lemma is often used in the following subsections.

\begin{lem}\label{lem:holder-type-ineq}
Suppose that Assumption $\ref{asm:p-q}$ holds and let $r\in\{q_1/2,q_2\}$. Define
\begin{equation*}
(r_1,r_2)=
\left\{\begin{aligned}
& (q_1,q_1) && \text{when $r=q_1/2$,} \\
& (\infty,q_2) && \text{when $r=q_2$.}
\end{aligned}\right.
\end{equation*}
Then 
\begin{align}
\|fg\|_{L_r(\BR_-^N)}
&\leq C\|f\|_{L_{r_1}(\BR_-^N)}\|g\|_{L_{r_2}(\BR_-^N)}, \label{r-holder-1} \\
\|fg\|_{H_r^1(\BR_-^N)}
&\leq C\|f\|_{H_{r_2}^1(\BR_-^N)}\|g\|_{H_{r_2}^1(\BR_-^N)},  \label{r-holder-2}
\end{align}
with some positive constant $C$.
\end{lem}

\begin{proof}
The  inequality \eqref{r-holder-2} with $r=q_2$ follows from the fact that $H_{q_2}^1(\BR_-^N)$ is a Banach algebra.
This completes the proof of Lemma \ref{lem:holder-type-ineq}.
\end{proof}

\subsection{Estimates of $\SSD(\eta,\Bu)$}

This subsection estimates $\SSD(\eta,\Bu)$ given by \eqref{def:D}.

\begin{prp}\label{prp:nonl-D}
Suppose that Assumption $\ref{asm:p-q}$ holds.
Then  the following assertions hold.
\begin{enumerate}[$(1)$]
\item
There exists a positive constant $M_1$ such that 
for any $(\eta^i,\Bu^i)\in K_{p,q_1,q_2}^N$, $i=1,2$,
\begin{align*}
&\sum_{r\in\{q_1/2,q_2\}}\|\langle t \rangle(\SSD(\eta^1,\Bu^1)-\SSD(\eta^2,\Bu^2))\|_{L_p(\BR_+,W_r^{2-1/r}(\BR_0^{N}))} \\
&\leq M_1
\|(\eta^1-\eta^2,\Bu^1-\Bu^2)\|_{K_{p,q_1,q_2}^N}
\Big(\|(\eta^1,\Bu^1)\|_{K_{p,q_1,q_2}^N}+\|(\eta^2,\Bu^2)\|_{K_{p,q_1,q_2}^N}\Big).
\end{align*}
\item
There exists a positive constant $M_2$ such that 
for any $(\eta^i,\Bu^i)\in K_{p,q_1,q_2}^N$, $i=1,2$,
\begin{align*}
&\sum_{r\in\{q_1/2,q_2\}}\|\langle t \rangle(\SSD(\eta^1,\Bu^1)-\SSD(\eta^2,\Bu^2))\|_{L_\infty(\BR_+,W_r^{1-1/r}(\BR_0^{N}))} \\
&\leq  M_2
\|(\eta^1-\eta^2,\Bu^1-\Bu^2)\|_{K_{p,q_1,q_2}^N}
\Big(\|(\eta^1,\Bu^1)\|_{K_{p,q_1,q_2}^N}+\|(\eta^2,\Bu^2)\|_{K_{p,q_1,q_2}^N}\Big).
\end{align*}
\end{enumerate}
\end{prp}

\begin{proof}
Let us write 
\begin{align*}
\SSD(\eta^1,\Bu^1)-\SSD(\eta^2,\Bu^2)
&=-(\Bu^1-\Bu^2)'\cdot\nabla'\CE_N\eta^1+(\Bu^2)'\cdot\nabla'\CE_N(\eta^1-\eta^2) \\
& =:I_1+I_2,
\end{align*}
where $\Ba'\cdot \nabla'=\sum_{j=1}^{N-1}a_j \pd_j$ for an $N$-vector $\Ba=(a_1,\dots,a_{N-1},a_N)^\SST$.
Let $r=q_1/2$ or $r=q_2$. We use $(r_1,r_2)$ defined in Lemma \ref{lem:holder-type-ineq} in what follows.

(1) {\bf Case 1}: $N=3$. By the trace theorem
\begin{equation*}
\|I_1(t)\|_{W_{r}^{2-1/r}(\BR_0^N)} \leq C \|I_1(t)\|_{H_r^2(\BR_-^N)},
\end{equation*}
which, combined with 
\begin{align*}
\|I_1(t)\|_{H_r^2(\BR_-^N)}
&\leq C
\Big(\|(\Bu^1(t)-\Bu^2(t))'\cdot\nabla'\CE_N\eta^1(t)\|_{H_r^1(\BR_-^N)} \\
&+\|(\Bu^1(t)-\Bu^2(t))'\|_{H_{r_2}^2(\BR_-^N)}\|\nabla'\CE_N\eta^1(t)\|_{L_{r_1}(\BR_-^N)} \\
&+\|(\Bu^1(t)-\Bu^2(t))'\|_{L_{r_1}(\BR_-^N)}\|\nabla'\CE_N\eta^1(t)\|_{H_{r_2}^2(\BR_-^N)}\Big)
\end{align*}
and with \eqref{r-holder-2}, furnishes
\begin{align*}
\|I_1(t)\|_{W_r^{2-1/r}(\BR_0^N)}
&\leq C\Big(\|(\Bu^1(t)-\Bu^2(t))'\|_{H_{r_2}^1(\BR_-^N)}\|\nabla'\CE_N\eta^1(t)\|_{H_{r_2}^1(\BR_-^N)} \\
&+\|(\Bu^1(t)-\Bu^2(t))'\|_{H_{r_2}^2(\BR_-^N)}\|\nabla'\CE_N\eta^1(t)\|_{L_{r_1}(\BR_-^N)}\Big) \\
&+\|(\Bu^1(t)-\Bu^2(t))'\|_{L_{r_1}(\BR_-^N)}\|\nabla'\CE_N\eta^1(t)\|_{H_{r_2}^2(\BR_-^N)}\Big).
\end{align*}
Thus
\begin{align*}
&\|\langle t \rangle I_1\|_{L_p(\BR_+,W_{r}^{2-1/r}(\BR_0^N))} \notag \\
&\leq C\Big(
\|\langle t\rangle^{1/3}(\Bu^1-\Bu^2)'\|_{L_\infty(\BR_+,H_{r_2}^1(\BR_-^N)^{N-1})}
\|\langle t \rangle^{2/3} \nabla'\CE_N \eta^1\|_{L_p(\BR_+,H_{r_2}^1(\BR_-^{N})^{N-1})} \notag \\
&+\|\langle t\rangle^{1/3}(\Bu^1-\Bu^2)'\|_{L_p(\BR_+,H_{r_2}^2(\BR_-^N)^{N-1})}
\|\langle t \rangle^{2/3} \nabla'\CE_N \eta^1\|_{L_\infty(\BR_+,L_{r_1}(\BR_-^{N})^{N-1})} \notag \\
&+\|\langle t\rangle^{1/3}(\Bu^1-\Bu^2)'\|_{L_\infty(\BR_+,L_{r_1}(\BR_-^N)^{N-1})}
\|\langle t \rangle^{2/3} \nabla'\CE_N \eta^1\|_{L_p(\BR_+,H_{r_2}^2(\BR_-^{N})^{N-1})}\Big).
\end{align*}
Combining this inequality with Lemma \ref{lem:main-embed3d} yields
\begin{equation*}
\|\langle t \rangle I_1\|_{L_p(\BR_+,W_{r}^{2-1/r}(\BR_0^N))}
\leq C\|\Bu^1-\Bu^2\|_{K_{p,q_1,q_2;2}^N}\|\eta^1\|_{K_{p,q_1,q_2;1}}.
\end{equation*}
Analogously,
\begin{equation*}
\|\langle t \rangle I_2\|_{L_p(\BR_+,W_{r}^{2-1/r}(\BR_0^N))}
\leq C\|\Bu^2\|_{K_{p,q_1,q_2;2}^N}\|\eta^1-\eta^2\|_{K_{p,q_1,q_2;1}^N}.
\end{equation*}
These two inequalities yield the desired inequality for  $N=3$.

{\bf Case 2}: $N\geq 4$.
We replace $\langle t \rangle^{1/3}$ with $\langle t \rangle^{1/2}$
and $\langle t\rangle^{2/3}$ with $\langle t \rangle^{1/2}$ in Case 1,
and then we obtain the desired inequality for $N\geq 4$ from Lemma \ref{lem:main-embed4d}.

(2) By the trace theorem
\begin{equation*}
\|I_1(t)\|_{W_r^{1-1/r}(\BR_0^N)} \leq  C\|I_1(t)\|_{H_r^1(\BR_-^N)},
\end{equation*}
which, combined with \eqref{r-holder-2}, yields
\begin{equation*}
\|I_1(t)\|_{W_r^{1-1/r}(\BR_0^N)}
\leq C\|(\Bu^1(t)-\Bu^2(t))'\|_{H_{r_2}^1(\BR_-^N)}\|\nabla'\CE_N\eta^1(t)\|_{H_{r_2}^1(\BR_-^N)}.
\end{equation*}
Thus
\begin{align*}
&\|\langle t \rangle I_1\|_{L_\infty(\BR_+,W_r^{1-1/r}(\BR_0^N))} \leq C \\
&\times\left\{\begin{aligned}
&\|\langle t \rangle^{1/3}(\Bu^1-\Bu^2)'\|_{L_\infty(\BR_+,H_{r_2}^1(\BR_-^N))}
\|\langle t \rangle^{2/3}\nabla'\CE_N\eta^1\|_{L_\infty(\BR_+,H_{r_2}^1(\BR_-^N))}
&& (N=3), \\
&\|\langle t \rangle^{1/2}(\Bu^1-\Bu^2)'\|_{L_\infty(\BR_+,H_{r_2}^1(\BR_-^N))}
\|\langle t \rangle^{1/2}\nabla'\CE_N\eta^1\|_{L_\infty(\BR_+,H_{r_2}^1(\BR_-^N))}
&& (N\geq 4),
\end{aligned}\right.
\end{align*}
which, combined with Lemmas \ref{lem:main-embed3d} and \ref{lem:main-embed4d}, furnishes
\begin{equation*}
\|\langle t \rangle I_1\|_{L_\infty(\BR_+,W_r^{1-1/r}(\BR_0^N))}
\leq C\|\Bu^1-\Bu^2\|_{K_{p,q_1,q_2;2}^N}
\|\eta^1\|_{K_{p,q_1,q_2;1}^N}.
\end{equation*}
Analogously,
\begin{equation*}
\|\langle t \rangle I_2\|_{L_\infty(\BR_+,W_r^{1-1/r}(\BR_0^N))}
\leq C\|\Bu^2\|_{K_{p,q_1,q_2;2}^N}
\|\eta^1-\eta^2\|_{K_{p,q_1,q_2;1}^N}.
\end{equation*}
These two inequalities yield the desired inequality.
This completes the proof of Proposition \ref{prp:nonl-D}.
\end{proof}

\subsection{Estimates of $\SSF(\eta,\Bu)$}
This subsection estimates $\SSF(\eta,\Bu)$ given by \eqref{dfn:F}.
Let us define 
\begin{alignat}{2}
\SSF_1(\eta,\Bu)&=(\pd_t\CE_N\eta)\pd_N\Bu, \quad
&\SSF_2(\eta,\Bu)&=u_j\pd_k\Bu, \notag  \\ 
\SSF_3(\eta,\Bu)&=(\pd_j\pd_k\CE_N\eta)\pd_l\Bu, \quad 
&\SSF_4(\eta,\Bu)&=(\pd_j\CE_N\eta)\pd_k\pd_l\Bu, \notag \\
\SSF_5(\eta,\Bu)&=(\pd_j\CE_N\eta)\pd_t\Bu, \label{dfn-F5}
\end{alignat}
where $j,k,l=1,\dots,N$ and $u_j$ denotes the $j$th component of $\Bu$.
We first prove the following lemma for $\SSF_1(\eta,\Bu)$.

\begin{lem}\label{lem:nonl-F-0}
Suppose that Assumption $\ref{asm:p-q}$ holds and that $q_2\leq 6$ when $N=3$. 
Let $r\in\{q_1/2,q_2\}$.
Then for any $(\eta^i,\Bu^i)\in K_{p,q_1,q_2}^N$, $i=1,2$,
\begin{align*}
&\|\langle t \rangle (\SSF_1(\eta^1,\Bu^1)-\SSF_1(\eta^2,\Bu^2))\|_{L_p(\BR_+,L_r(\BR_-^N)^N)} \\
&\leq 
C\|(\eta^1-\eta^2,\Bu^1-\Bu^2)\|_{K_{p,q_1,q_2}^N}
\Big(\|(\eta^1,\Bu^1)\|_{K_{p,q_1,q_2}^N}+\|(\eta^2,\Bu^2)\|_{K_{p,q_1,q_2}^N}\Big)
\end{align*}
 with some positive constant $C$.
\end{lem}

\begin{proof}
Let us write
\begin{align*}
\SSF_1(\eta^1,\Bu^1)-\SSF_2(\eta^2,\Bu^2)
&=
(\pd_t\CE_N\eta^1-\pd_t\CE\eta^2)\pd_N\Bu^1+(\pd_t\CE_N\eta^2)(\pd_N\Bu^1-\pd_N\Bu^2) \\
&=:I_1+I_2.
\end{align*}

{\bf Case 1}: $r=q_1/2$ and $N=3$.
We see that
\begin{equation*}
\|I_1(t)\|_{L_{q_1/2}(\BR_-^N)}\leq \|\pd_t\CE_N\eta^1(t)-\pd_t\CE_N\eta^2(t)\|_{L_2(\BR_-^N)}
\|\pd_N\Bu^1(t)\|_{L_s(\BR_-^N)},
\end{equation*}
where $2/q_1=1/2+1/s$. Notice that $q_1<s<q_2$, because
\begin{equation*}
\frac{1}{q_1}-\frac{1}{s} =\frac{1}{2}-\frac{1}{q_1}>0, \quad \frac{1}{s}=\frac{2}{q_1}-\frac{1}{2}
\geq \frac{2}{2+q_0}-\frac{1}{2}=\frac{2}{2+2/9}-\frac{1}{2}=\frac{2}{5}>\frac{1}{3}.
\end{equation*}
By Lemma \ref{lem:int-p} 
\begin{equation*}
\|\pd_N\Bu^1(t)\|_{L_s(\BR_-^N)}\leq C
\|\pd_N\Bu^1(t)\|_{L_{q_1}(\BR_-^N)}^{1-\theta}\|\pd_N\Bu^1(t)\|_{L_{q_2}(\BR_-^N)}^{\theta}
\end{equation*}
for some $\theta\in(0,1)$,
which, combined with Young's inequality $a^{1-\theta}b^\theta\leq C(a+b)$ for $a,b\geq 0$, furnishes
\begin{equation*}
\|\pd_N\Bu^1(t)\|_{L_s(\BR_-^N)}\leq C
\Big(\|\pd_N\Bu^1(t)\|_{L_{q_1}(\BR_-^N)}+ \|\pd_N\Bu^1(t)\|_{L_{q_2}(\BR_-^N)}\Big).
\end{equation*}
Thus
\begin{align*}
&\|\langle t\rangle I_1\|_{L_p(\BR_+,L_{q_1/2}(\BR_-^N)^N)}
\leq C\|\langle t \rangle^{1/3}(\pd_t\CE_N\eta^1-\pd_t\CE_N\eta^2)\|_{L_\infty(\BR_+,L_2(\BR_-^N))} \\
&\times
\Big(\|\langle t \rangle^{2/3}\pd_N\Bu^1\|_{L_p(\BR_+,L_{q_1}(\BR_-^N)^N)}
+\|\langle t \rangle^{2/3}\pd_N\Bu^1\|_{L_p(\BR_+,L_{q_2}(\BR_-^N)^N)}\Big),
\end{align*}
which implies
\begin{equation*}
\|\langle t\rangle I_1\|_{L_p(\BR_+,L_{q_1/2}(\BR_-^N)^N)}
\leq C\|\eta^1-\eta^2\|_{K_{p,q_1,q_2;1}^N}\|\Bu^1\|_{K_{p,q_1,q_2;2}}.
\end{equation*}
Analogously,
\begin{equation*}
\|\langle t\rangle I_2\|_{L_p(\BR_+,L_{q_1/2}(\BR_-^N)^N)}
\leq C\|\eta^2\|_{K_{p,q_1,q_2;1}^N}\|\Bu^1-\Bu^2\|_{K_{p,q_1,q_2;2}}.
\end{equation*}
These two inequalities yield the desired inequality for $r=q_1/2$ and $N=3$.

{\bf Case 2}: $r=q_2$ and $N=3$.
We see that
\begin{equation*}
\|I_1(t)\|_{L_{q_2}(\BR_-^N)}
\leq
\|\pd_t\CE_N\eta^1(t)-\pd_t\CE_N\eta^2(t)\|_{L_{q_2}(\BR_-^N)}\|\pd_N\Bu^1(t)\|_{L_\infty(\BR_-^N)}.
\end{equation*}
Since $H_2^1(\BR_-^N)$ is continuously embedded into $L_{q_2}(\BR_-^N)$ for $N=3$,
\begin{equation*}
\|\pd_t\CE_N\eta^1(t)-\pd_t\CE_N\eta^2(t)\|_{L_{q_2}(\BR_-^N)}
\leq C\|\pd_t\CE_N\eta^1(t)-\pd_t\CE_N\eta^2(t)\|_{H_{2}^1(\BR_-^N)}.
\end{equation*}
On the other hand, Lemma \ref{lem:fund-embed} \eqref{lem:fund-embed-3} yields
\begin{equation*}
\|\pd_N\Bu^1(t)\|_{L_\infty(\BR_-^N)} \leq C\|\pd_N\Bu^1(t)\|_{H_{q_2}^1(\BR_-^N)}.
\end{equation*}
Combining these three inequalities, we have 
\begin{equation*}
\|I_1(t)\|_{L_{q_2}(\BR_-^N)}
\leq  C\|\pd_t\CE_N\eta^1(t)-\pd_t\CE_N\eta^2(t)\|_{H_2^1(\BR_-^N)}\|\pd_N\Bu^1(t)\|_{H_{q_2}^1(\BR_-^N)}.
\end{equation*}
Thus
\begin{align*}
&\|\langle t \rangle I_1\|_{L_p(\BR_+,L_{q_2}(\BR_-^N)^N)} \\
&\leq C
\|\langle t \rangle^{1/3}(\pd_t\CE_N\eta^1-\pd_t\CE_N\eta^2)\|_{L_\infty(\BR_+,H_2^1(\BR_-^N))}
\|\langle t \rangle^{2/3}\pd_N\Bu^1\|_{L_p(\BR_+,H_{q_2}^1(\BR_-^N)^N)} \\
&\leq C
\|\eta^1-\eta^2\|_{K_{p,q_1,q_2;1}^N}\|\Bu^1\|_{K_{p,q_1,q_2;2}^N}.
\end{align*}
Analogously,
\begin{equation*}
\|\langle t \rangle I_2\|_{L_p(\BR_+,L_{q_2}(\BR_-^N)^N)}
\leq C 
\|\eta^2\|_{K_{p,q_1,q_2;1}^N}\|\Bu^1-\Bu^2\|_{K_{p,q_1,q_2;2}^N}.
\end{equation*}
The desired inequality therefore holds for $r=q_2$ and $N=3$.

{\bf Case 3}: $N\geq 4$.
We use $(r_1,r_2)$ defined in Lemma \ref{lem:holder-type-ineq}.
It holds by \eqref{r-holder-1} that
\begin{align*}
\|I_1(t)\|_{L_{r}(\BR_-^N)}
&\leq \| \pd_t \CE_N\eta^1(t)-\pd_t\CE_N\eta^2(t)\|_{L_{r_1}(\BR_-^N)}
\|\pd_N \Bu^1(t)\|_{L_{r_2}(\BR_-^N)},
\end{align*}
and thus  
\begin{align*}
\|\langle t \rangle I_1\|_{L_p(\BR_+,L_{r}(\BR_-^N)^N)} 
&\leq  C
\|\langle t \rangle^{1/2}(\pd_t \CE_N\eta^1 - \pd_t\CE_N \eta^2)\|_{L_p(\BR_+,L_{r_1}(\BR_-^{N} ))} \\
&\times \|\langle t \rangle^{1/2}\pd_N\Bu^1\|_{L_\infty(\BR_+,L_{r_2} (\BR_-^{N} )^N)}.
\end{align*}
Combining this inequality with Lemma \ref{lem:main-embed4d} furnishes
\begin{equation*}
\|\langle t \rangle I_1\|_{L_p(\BR_+,L_{r}(\BR_-^N)^N)} 
\leq C \|\eta^1-\eta^2\|_{K_{p,q_1,q_2;1}^N}\|\Bu^1\|_{K_{p,q_1,q_2;2}^N}.
\end{equation*}
Analogously,
\begin{equation*}
\|\langle t \rangle I_2\|_{L_p(\BR_+,L_{r}(\BR_-^N)^N)} 
\leq  C\|\eta^2\|_{K_{p,q_1,q_2;1}^N}\|\Bu^1-\Bu^2\|_{K_{p,q_1,q_2;2}^N}.
\end{equation*}
These two inequalities yield the desired inequality for $N\geq 4$.
This completes the proof of Lemma \ref{lem:nonl-F-0}.
\end{proof}

We next  prove

\begin{lem}\label{lem:nonl-F-1}
Suppose that Assumption $\ref{asm:p-q}$ holds and let $r\in\{q_1/2,q_2\}$.
Then for any $(\eta^i,\Bu^i)\in K_{p,q_1,q_2}^N$, $i=1,2$,
\begin{align*}
&\|\langle t \rangle (\SSF_m(\eta^1,\Bu^1)-\SSF_m(\eta^2,\Bu^2))\|_{L_p(\BR_+,L_r(\BR_-^N)^N)} \\
&\leq 
C\|(\eta^1-\eta^2,\Bu^1-\Bu^2)\|_{K_{p,q_1,q_2}^N}
\Big(\|(\eta^1,\Bu^1)\|_{K_{p,q_1,q_2}^N}+\|(\eta^2,\Bu^2)\|_{K_{p,q_1,q_2}^N}\Big),
\end{align*}
where $m=2,3,4,5$ and $C$ is a positive constant.
\end{lem}

\begin{proof}
We consider $\SSF_2(\eta,\Bu)$, $\SSF_3(\eta,\Bu)$, and $\SSF_5(\eta,\Bu)$ only.
The $\SSF_4(\eta,\Bu)$ can be treated similarly to $\SSF_5(\eta,\Bu)$.
Let us write
\begin{align*}
\SSF_2(\eta^1,\Bu^1)-\SSF_2(\eta^2,\Bu^2)
&=(u_j^1-u_j^2)\pd_k\Bu^1+ u_j^2(\pd_k\Bu^1-\pd_k\Bu^2) \\
&=:I_1+I_2, \\
\SSF_3(\eta^1,\Bu^1)-\SSF_3(\eta^2,\Bu^2)
&=(\pd_j\pd_k\CE_N\eta^1- \pd_j\pd_k\CE_N\eta^2)\pd_l\Bu^1
+(\pd_j\pd_k\CE_N\eta^2)(\pd_l\Bu^1-\pd_l\Bu^2)\\
&=:I_3+I_4, \\
\SSF_5(\eta^1,\Bu^1)-\SSF_5(\eta^2,\Bu^2)
&=(\pd_j\CE_N\eta^1-\pd_j\CE_N\eta^2)\pd_t\Bu^1+(\pd_j\CE_N\eta^2)(\pd_t\Bu^1-\pd_t\Bu^2) \\
&=:I_5+I_6.
\end{align*}
We use $(r_1,r_2)$ defined in Lemma \ref{lem:holder-type-ineq} in what follows.

{\bf Case 1}: $N=3$.
We see by \eqref{r-holder-1} that
\begin{align*}
\|I_1(t)\|_{L_{r}(\BR_-^N)}
&\leq \|u_j^1(t)-u_j^2(t)\|_{L_{r_1}(\BR_-^N)}\|\pd_k\Bu^1(t)\|_{L_{r_2}(\BR_-^N)}, \\
\|I_3(t)\|_{L_{r}(\BR_-^N)}
&\leq \|\pd_j\pd_k\CE_N\eta^1(t)- \pd_j\pd_k\CE_N\eta^2(t)\|_{L_{r_2}(\BR_-^N)}
\|\pd_l\Bu^1(t)\|_{L_{r_1}(\BR_-^{N})}, \\
\|I_5(t)\|_{L_r(\BR_-^N)}
&\leq \|\pd_j\CE_N\eta^1(t)-\pd_j\CE_N\eta^2(t)\|_{L_{r_1}(\BR_-^N)}
\|\pd_t\Bu^1(t)\|_{L_{r_2}(\BR_-^N)}.
\end{align*}
Lemma \ref{lem:main-embed3d} tells us that
\begin{align*}
\|\langle t \rangle I_1\|_{L_p(\BR_+,L_{r}(\BR_-^N)^N)}
&\leq
\|\langle t \rangle^{1/3}(u_j^1-u_j^2)\|_{L_\infty(\BR_+,L_{r_1}(\BR_-^N))} \\
&\times \|\langle t \rangle^{2/3}\pd_k\Bu^1\|_{L_p(\BR_+,L_{r_2}(\BR_-^N)^N)} \\
&\leq C\|\Bu^1-\Bu^2\|_{K_{p,q_1,q_2;2}^N}\|\Bu^1\|_{K_{p,q_1,q_2;2}^N}, \\
\| \langle t \rangle I_3\|_{L_p(\BR_+,L_{r}(\BR_-^N)^N)}
&\leq \|\langle t \rangle^{2/3}(\pd_j\pd_k\CE_N\eta^1- \pd_j\pd_k\CE_N\eta^2)\|_{L_\infty(\BR_+,L_{r_2}(\BR_-^N))} \\
&\times \|\langle t \rangle^{1/3}\pd_k\Bu^1\|_{L_p(\BR_+,L_{r_1}(\BR_-^N)^N)} \\
&\leq C\|\eta^1-\eta^2\|_{K_{p,q_1,q_2;1}^N}
\|\Bu^1\|_{K_{p,q_1,q_2;2}^N}, \\
\|\langle t \rangle I_5\|_{L_p(\BR_+,L_{r}(\BR_-^N)^N)}
&\leq
\|\langle t\rangle^{2/3}(\pd_j\CE_N\eta^1-\pd_j\CE_N\eta^2)\|_{L_\infty(\BR_+,L_{r_1}(\BR_-^N))} \\
&\times\|\langle t\rangle^{1/3}\pd_t\Bu^1\|_{L_p(\BR_+,L_{r_2}(\BR_-^N)^N)} \\
&\leq 
C\|\eta^1-\eta^2\|_{K_{p,q_1,q_2;1}^N}
\|\Bu^1\|_{K_{p,q_1,q_2;2}^N}.
\end{align*}
Analogously,
\begin{align*}
\|\langle t \rangle I_2\|_{L_p(\BR_+,L_r(\BR_-^N)^N)}
&\leq C\|\Bu^2\|_{K_{p,q_1,q_2;2}^N}\|\Bu^1-\Bu^2\|_{K_{p,q_1,q_2;2}^N}, \\
\|\langle t \rangle I_4\|_{L_p(\BR_+,L_{r}(\BR_-^N)^N)}
&\leq C\|\eta^2\|_{K_{p,q_1,q_2;1}^N}\|\Bu^1-\Bu^2\|_{K_{p,q_1,q_2;2}^N}, \\
\|\langle t \rangle I_6\|_{L_p(\BR_+,L_{r}(\BR_-^N)^N)}
&\leq C\|\eta^2\|_{K_{p,q_1,q_2;1}^N}\|\Bu^1-\Bu^2\|_{K_{p,q_1,q_2;2}^N}.
\end{align*}
This completes the proof of Case 1.

{\bf Case 2}: $N\geq 4$.
We replace $\langle t \rangle^{1/3}$ with $\langle t \rangle^{1/2}$
and $\langle t\rangle^{2/3}$ with $\langle t \rangle^{1/2}$ in Case 1,
and then we obtain the desired inequality for $N\geq 4$ from Lemma \ref{lem:main-embed4d}.
This completes the proof of Lemma \ref{lem:nonl-F-1}.
\end{proof}

Let us define
\begin{equation}\label{dfn:F-eta}
\CF(\eta)=(\pd_1\CE_N\eta,\dots,\pd_N\CE_N\eta).
\end{equation}
Notice that for any multi-index $\alpha=(\alpha_1,\dots,\alpha_N)\in\BN_0^N$
\begin{equation*}
(\CF(\eta))^\alpha=(\pd_1\CE_N\eta)^{\alpha_1}\dots (\pd_N\CE_N\eta)^{\alpha_N}.
\end{equation*}
We then have

\begin{lem}\label{lem:nonl-F-2}
Suppose that Assumption $\ref{asm:p-q}$ holds.
Let $q\in\{q_1,q_2\}$  and $m\in\BN$.
Then for any  $\eta^i\in K_{p,q_1,q_2;1}^N$, $i=1,2$, 
and for any multi-index $\alpha\in\BN_0^N$ with $|\alpha|=m$
\begin{align}\label{est:Z-1-eta}
&\|(\CF(\eta^1))^\alpha-(\CF(\eta^2))^\alpha\|_Z \notag \\
&\leq  C\|\eta^1-\eta^2\|_{K_{p,q_1,q_2;1}^N}\Big(\|\eta^1\|_{K_{p,q_1,q_2;1}^N}+\|\eta^2\|_{K_{p,q_1,q_2;1}^N}\Big)^{m-1}
\end{align}
hold for $Z=H_p^1(\BR_+,H_q^1(\BR_-^N))$,
$Z= H_p^1(\BR_+,L_\infty(\BR_-^N))$,
$Z= L_\infty(\BR_+,H_q^1(\BR_-^N))$,
or $Z=L_\infty(\BR_+, L_\infty(\BR_-^N))$,
where $C$ is a positive constant.
\end{lem}

\begin{proof}
In this proof, we often use the following inequalities:
\begin{align}
\|u\|_{H_p^1(\BR_+,L_\infty(\BR_-^N))}
&\leq C\|u\|_{H_p^1(\BR_+,H_{q_2}^1(\BR_-^N))}, \label{eta-est-kihon-1} \\
\|u\|_{L_\infty(\BR_+,H_q^1(\BR_-^N))}
&\leq C\|u\|_{H_p^1(\BR_+,H_q^1(\BR_-^N))}, \label{eta-est-kihon-2} \\
\|u\|_{L_\infty(\BR_+,L_\infty(\BR_-^N))}
&\leq C\|u\|_{H_p^1(\BR_+,H_{q_2}^1(\BR_-^N))},  \label{eta-est-kihon-3}
\end{align}
which follow from Lemma \ref{lem:fund-embed} \eqref{lem:fund-embed-3} and Lemma \ref{lem:H^1-embed}.

Let us prove \eqref{est:Z-1-eta} by induction in the case $Z=H_p^1(\BR_+,H_q^1(\BR_-^N))$.

We first consider the case $m=1$. In this case,
$(\CF(\eta^1))^\alpha-(\CF(\eta^2))^\alpha =\pd_j\CE_N\eta^1-\pd_j\CE_N\eta^2$
for some $j$. It thus holds that 
\begin{equation*}
\|(\CF(\eta^1))^\alpha-\CF(\eta^2))^\alpha\|_{H_p^1(\BR_+,H_q^1(\BR_-^N))}
\leq C\|\eta^1-\eta^2\|_{K_{p,q_1,q_2;1}^N}.
\end{equation*}
This completes the proof of $m=1$.

We assume \eqref{est:Z-1-eta} holds and consider the case $|\alpha|=m+1$.
In this case, 
\begin{equation*}
\alpha=(0,\dots,\underset{\text{$j$th}}{1},\dots,0)+\beta, \quad |\beta|=m
\end{equation*}
for some $j$. Then
\begin{align*}
(\CF(\eta^1))^\alpha-(\CF(\eta^2))^\alpha
&=(\pd_j\CE_N\eta^1)(\CF(\eta^1))^\beta
-(\pd_j\CE_N\eta^2)(\CF(\eta^2))^\beta \\
&=(\pd_j\CE_N\eta^1-\pd_j\CE_N\eta^2)(\CF(\eta^1))^\beta \\
&+(\pd_j\CE_N\eta^2)((\CF(\eta^1))^\beta-(\CF(\eta^2))^\beta) \\
&=:I_1+I_2.
\end{align*}
One sees that
\begin{align*}
&\|I_1\|_{L_p(\BR_+,H_q^1(\BR_-^N))} \\
&\leq C\|\pd_j\CE_N\eta^1-\pd_j\CE_N\eta^2\|_{L_p(\BR_+,H_\infty^1((\BR_-^N))}
\|(\CF(\eta^1))^\beta\|_{L_\infty(\BR_+,H_q^1(\BR_-^N))},
\end{align*}
which, combined with Lemma \ref{lem:fund-embed} \eqref{lem:fund-embed-3} and \eqref{eta-est-kihon-2}, furnishes
\begin{align*}
&\|I_1\|_{L_p(\BR_+,H_q^1(\BR_-^N))} \\
&\leq C\|\pd_j\CE_N\eta^1-\pd_j\CE_N\eta^2\|_{L_p(\BR_+,H_{q_2}^2((\BR_-^N))}
\|(\CF(\eta^1))^\beta\|_{H_p^1(\BR_+,H_q^1(\BR_-^N))} \\
&\leq C\|\eta^1-\eta^2\|_{K_{p,q_1,q_2;1}^N}
\|(\CF(\eta^1))^\beta\|_{H_p^1(\BR_+,H_q^1(\BR_-^N))}.
\end{align*}
By \eqref{est:Z-1-eta} with $\eta^2=0$ and $Z=H_p^1(\BR_+,H_q^1(\BR_-^N))$
\begin{equation}\label{ass-induc-1}
\|(\CF(\eta^1))^\beta\|_{H_p^1(\BR_+,H_q^1(\BR_-^N))}
\leq C\|\eta^1\|_{K_{p,q_1,q_2;1}^N}^m,
\end{equation}
and thus 
\begin{equation}\label{I1-est-eta-induc}
\|I_1\|_{L_p(\BR_+,H_q^1(\BR_-^N))} 
\leq 
C\|\eta^1-\eta^2\|_{K_{p,q_1,q_2;1}^N}\|\eta^1\|_{K_{p,q_1,q_2;1}^N}^m.
\end{equation}

Let us consider time derivatives of $I_1(t)$. It holds that
\begin{align*}
\pd_t I_1(t)
&= \pd_t(\pd_j\CE_N\eta^1-\pd_j\CE_N\eta^2)\cdot  (\CF(\eta^1))^\beta \\
&+(\pd_j\CE_N\eta^1-\pd_j\CE_N\eta^2)\cdot \pd_t(\CF(\eta^1))^\beta, \\ 
\pd_k\pd_t I_1(t)
&= \pd_k\pd_t(\pd_j\CE_N\eta^1-\pd_j\CE_N\eta^2)\cdot  (\CF(\eta^1))^\beta \\
&+\pd_t(\pd_j\CE_N\eta^1-\pd_j\CE_N\eta^2)\cdot \pd_k(\CF(\eta^1))^\beta \\ 
&+\pd_k(\pd_j\CE_N\eta^1-\pd_j\CE_N\eta^2)\cdot \pd_t(\CF(\eta^1))^\beta \\
&+(\pd_j\CE_N\eta^1-\pd_j\CE_N\eta^2)\cdot \pd_k\pd_t(\CF(\eta^1))^\beta,
\end{align*}
where $\pd_k=\pd/\pd x_k$ and $k=1,\dots,N$. Then 
\begin{align*}
&\|\pd_t I_1\|_{L_p(\BR_+,L_q(\BR_-^N))} \\
&\leq C
\Big(\|\pd_t(\pd_j\CE_N\eta^1-\pd_j\CE_N\eta^2)\|_{L_p(\BR_+,L_\infty(\BR_-^N))}
\|(\CF(\eta^1))^\beta\|_{L_\infty(\BR_+,L_q(\BR_-^N))} \\
&+
\|\pd_j\CE_N\eta^1-\pd_j\CE_N\eta^2\|_{L_\infty(\BR_+,L_\infty(\BR_-^N))}
\|\pd_t(\CF(\eta^1))^\beta\|_{L_p(\BR_+,L_q(\BR_-^N))}\Big) \\
&\leq C
\Big(\|\pd_j\CE_N\eta^1-\pd_j\CE_N\eta^2\|_{H_p^1(\BR_+,L_\infty(\BR_-^N))}
\|(\CF(\eta^1))^\beta\|_{L_\infty(\BR_+,H_q^1(\BR_-^N))} \\
&+
\|\pd_j\CE_N\eta^1-\pd_j\CE_N\eta^2\|_{L_\infty(\BR_+,L_\infty(\BR_-^N))}
\|(\CF(\eta^1))^\beta\|_{H_p^1(\BR_+,H_q^1(\BR_-^N))}\Big),
\end{align*}
and also 
\begin{align*}
&\|\pd_k\pd_t I_1\|_{L_p(\BR_+,L_q(\BR_-^N))} \\
&\leq 
C\Big( \|\pd_t\pd_k (\pd_j\CE_N\eta^1-\pd_j\CE_N\eta^2)\|_{L_p(\BR_+,L_q(\BR_-^N))}
\|(\CF(\eta^1))^\beta\|_{L_\infty(\BR_+,L_\infty(\BR_-^N))} \\
&+ \|\pd_t(\pd_j\CE_N\eta^1-\pd_j\CE_N\eta^2)\|_{L_p(\BR_+,L_\infty(\BR_-^N))}
\|\pd_k(\CF(\eta^1))^\beta\|_{L_\infty(\BR_+,L_q(\BR_-^N))} \\
&+\|\pd_k(\pd_j\CE_N\eta^1-\pd_j\CE_N\eta^2)\|_{L_\infty(\BR_+,L_q(\BR_-^N))}
\|\pd_t(\CF(\eta^1))^\beta\|_{L_p(\BR_+,L_\infty(\BR_-^N))} \\
&+\|\pd_j\CE_N\eta^1-\pd_j\CE_N\eta^2\|_{L_\infty(\BR_+,L_\infty(\BR_-^N))}
\|\pd_k\pd_t(\CF(\eta^1))^\beta\|_{L_p(\BR_+,L_q(\BR_-^N))}\Big) \\
&\leq 
C\Big( \|\pd_j\CE_N\eta^1-\pd_j\CE_N\eta^2\|_{H_p^1(\BR_+,H_q^1(\BR_-^N))}
\|(\CF(\eta^1))^\beta\|_{L_\infty(\BR_+,L_\infty(\BR_-^N))} \\
&+ \|\pd_j\CE_N\eta^1-\pd_j\CE_N\eta^2\|_{H_p^1(\BR_+,L_\infty(\BR_-^N))}
\|(\CF(\eta^1))^\beta\|_{L_\infty(\BR_+,H_q^1(\BR_-^N))} \\
&+\|\pd_j\CE_N\eta^1-\pd_j\CE_N\eta^2\|_{L_\infty(\BR_+,H_q^1(\BR_-^N))}
\|(\CF(\eta^1))^\beta\|_{H_p^1(\BR_+,L_\infty(\BR_-^N))} \\
&+\|\pd_j\CE_N\eta^1-\pd_j\CE_N\eta^2\|_{L_\infty(\BR_+,L_\infty(\BR_-^N))}
\|(\CF(\eta^1))^\beta\|_{H_p^1(\BR_+,H_q^1(\BR_-^N))}\Big).
\end{align*}
Thus \eqref{eta-est-kihon-1}--\eqref{eta-est-kihon-3} show that
\begin{align*}
\|\pd_t I_1\|_{L_p(\BR_+,H_q^1(\BR_-^N))}
&\leq C\bigg(\sum_{q\in\{q_1,q_2\}}\|\pd_j\CE_N\eta^1-\pd_j\CE_N\eta^2\|_{H_p^1(\BR_+,H_q^1(\BR_-^N))}\bigg) \\
&\times\sum_{q\in\{q_1,q_2\}}\|(\CF(\eta^1))^\beta\|_{H_p^1(\BR_+,H_q^1(\BR_-^N))} \\
&\leq C\|\eta^1-\eta^2\|_{K_{p,q_1,q_2;1}^N}\sum_{q\in\{q_1,q_2\}}\|(\CF(\eta^1))^\beta\|_{H_p^1(\BR_+,H_q^1(\BR_-^N))},
\end{align*}
which, combined with \eqref{ass-induc-1}, furnishes
\begin{equation*}
\|\pd_t I_1\|_{L_p(\BR_+,H_q^1(\BR_-^N))}
\leq C\|\eta^1-\eta^2\|_{K_{p,q_1,q_2;1}^N}\|\eta^1\|_{K_{p,q_1,q_2;1}^N}^m.
\end{equation*}
Combining this with \eqref{I1-est-eta-induc} gives us
\begin{align*}
\|I_1\|_{H_p^1(\BR_+,H_q^1(\BR_-^N))} 
\leq C\|\eta^1-\eta^2\|_{K_{p,q_1,q_2;1}^N}\|\eta^1\|_{K_{p,q_1,q_2;1}^N}^m.
\end{align*}
Analogously,
\begin{align*}
&\|I_2\|_{H_p^1(\BR_+,H_q^1(\BR_-^N))}  \\
&\leq C\|\eta^2\|_{K_{p,q_1,q_2;1}^N}\|\eta^1-\eta^2\|_{K_{p,q_1,q_2;1}^N}
\Big(\|\eta^1\|_{K_{p,q_1,q_2;1}^N}+\|\eta^2\|_{K_{p,q_1,q_2;1}^N}\Big)^{m-1}.
\end{align*}
Hence \eqref{est:Z-1-eta} holds for $|\alpha|=m+1$ and $Z=H_p^1(\BR_+,H_q^1(\BR_-^N))$.

The other cases $Z= H_p^1(\BR_+,L_\infty(\BR_-^N))$,
$Z= L_\infty(\BR_+,H_q^1(\BR_-^N))$,
or $Z=L_\infty(\BR_+, L_\infty(\BR_-^N))$
follow from \eqref{eta-est-kihon-1}--\eqref{eta-est-kihon-3} and 
\eqref{est:Z-1-eta} with $Z=H_p^1(\BR_+,H_q^1(\BR_-^N))$.
This completes the proof of Lemma \ref{lem:nonl-F-2}
\end{proof}

Recall $\BJ(\eta)$ and $\wtd\SSF(\eta,\Bu)$ given by \eqref{matrix:JK} and \eqref{dfn:F-tilde}, respectively.
Lemmas \ref{lem:nonl-F-0}--\ref{lem:nonl-F-2} immediately yield the following lemma.

\begin{lem}\label{lem:nonl-F-3}
Suppose that Assumption $\ref{asm:p-q}$ holds and let $r\in\{q_1/2,q_2\}$.
Then for any $(\eta^i,\Bu^i)\in K_{p,q_1,q_2}^N$, $i=1,2$,
\begin{align*}
&\|\langle t \rangle \{(\BI+\BJ(\eta^1))\wtd \SSF(\eta^1,\Bu^1)-(\BI+\BJ(\eta^2))\wtd\SSF(\eta^2,\Bu^2)\}
\|_{L_p(\BR_+,L_r(\BR_-^N)^N)} \\
&\leq 
C\|(\eta^1-\eta^2,\Bu^1-\Bu^2)\|_{K_{p,q_1,q_2}^N}
\sum_{j=1}^5\Big(\|(\eta^1,\Bu^1)\|_{K_{p,q_1,q_2}^N}+\|(\eta^2,\Bu^2)\|_{K_{p,q_1,q_2}^N}\Big)^j, \\
&\|\langle t \rangle (\BJ(\eta^1)\Delta \Bu^1-\BJ(\eta^2)\Delta\Bu^2)\|_{L_p(\BR_+,L_r(\BR_-^N)^N)} \\
&\leq 
C\|(\eta^1-\eta^2,\Bu^1-\Bu^2)\|_{K_{p,q_1,q_2}^N}
\Big(\|(\eta^1,\Bu^1)\|_{K_{p,q_1,q_2}^N}+\|(\eta^2,\Bu^2)\|_{K_{p,q_1,q_2}^N}\Big), \\
&\|\langle t \rangle (\BJ(\eta^1)\pd_t\Bu^1-\BJ(\eta^2)\pd_t\Bu^2)\|_{L_p(\BR_+,L_r(\BR_-^N)^N)} \\
&\leq 
C\|(\eta^1-\eta^2,\Bu^1-\Bu^2)\|_{K_{p,q_1,q_2}^N}
\Big(\|(\eta^1,\Bu^1)\|_{K_{p,q_1,q_2}^N}+\|(\eta^2,\Bu^2)\|_{K_{p,q_1,q_2}^N}\Big).
\end{align*}
\end{lem}

\begin{proof}
We can prove the desired inequalities by Lemmas \ref{lem:nonl-F-0}--\ref{lem:nonl-F-2} with some calculations,
so that we may omit the detailed proof.
\end{proof}

We next prove

\begin{lem}\label{lem:F-final}
Suppose that Assumption $\ref{asm:p-q}$ holds and let $r\in\{q_1/2,q_2\}$.
Then for any $(\eta^i,\Bu^i)\in K_{p,q_1,q_2}^N$, $i=1,2$, satisfying
\begin{equation}\label{est-1/2}
\sup_{(x,t)\in \BR_-^N\times (0,\infty)}|(\pd_N\CE\eta^i)(x,t)| \leq \frac{1}{2},
\end{equation}
there holds
\begin{align*}
&\bigg\|\langle t\rangle \bigg(\frac{(\BI+\BJ(\eta^1))\wtd\SSF(\eta^1,\Bu^1)}{(1+\pd_N\CE\eta^1)^3}
-\frac{(\BI+\BJ(\eta^2))\wtd\SSF(\eta^2,\Bu^2)}{(1+\pd_N\CE\eta^2)^3}\bigg)\bigg\|_{L_p(\BR_+,L_r(\BR_-^N)^N)} \\
&\leq 
C
\|(\eta^1-\eta^2,\Bu^1-\Bu^2)\|_{K_{p,q_1,q_2}^N}
\sum_{j=1}^8\Big(\|(\eta^1,\Bu^1)\|_{K_{p,q_1,q_2}^N}+\|(\eta^2,\Bu^2)\|_{K_{p,q_1,q_2}^N}\Big)^j
\end{align*}
for some positive constant $C$.
\end{lem}

\begin{proof}
We see that
\begin{align*}
&\frac{(\BI+\BJ(\eta^1))\wtd\SSF(\eta^1,\Bu^1)}{(1+\pd_N\CE\eta^1)^3}
-\frac{(\BI+\BJ(\eta^2))\wtd\SSF(\eta^2,\Bu^2)}{(1+\pd_N\CE\eta^2)^3} \\
&=\bigg(\frac{1}{(1+\pd_N\CE\eta^1)^3}-\frac{1}{(1+\pd_N\CE\eta^2)^3}\bigg)(\BI+\BJ(\eta^1))\wtd\SSF(\eta^1,\Bu^1) \\
&+\frac{1}{(1+\pd_N\CE\eta^2)^3}\Big((\BI+\BJ(\eta^1))\wtd\SSF(\eta^1,\Bu^1)-(\BI+\BJ(\eta^2))\wtd\SSF(\eta^2,\Bu^2)\Big) \\
&=:I_1+I_2.
\end{align*}
It then holds that
\begin{align*}
&\frac{1}{(1+\pd_N\CE\eta^1)^3}-\frac{1}{(1+\pd_N\CE\eta^2)^3}  \\
&=\frac{1}{(1+\pd_N\CE\eta^1)^3(1+\pd_N\CE\eta^2)^3}
\bigg[3(\pd_N\CE\eta^2-\pd_N\CE\eta^1) 
+3\{(\pd_N\CE\eta^2)^2-(\pd_N\CE\eta^1)^2\} \\
&+(\pd_N\CE\eta^2)^3-(\pd_N\CE\eta^1)^3\bigg],
\end{align*}
which, combined with Lemma \ref{lem:nonl-F-2}, furnishes
\begin{align*}
&\Big\|\frac{1}{(1+\pd_N\CE\eta^1)^3}-\frac{1}{(1+\pd_N\CE\eta^2)^3} \Big\|_{L_\infty(\BR_-^N,L_\infty(\BR_-^N))} \\
&\leq C 
\bigg(\sup_{(x,t)\in\BR_-^N\times (0,\infty)}\frac{1}{|1+\pd_N\CE\eta^1||1+\pd_N\CE\eta^2|}\bigg)^3 \\
&\times \|\eta^1-\eta^2\|_{K_{p,q_1,q_2;1}^N}
\sum_{j=0}^2\Big(\|\eta^1\|_{K_{p,q_1,q_2;1}^N}+\|\eta^2\|_{K_{p,q_1,q_2;1}^N}\Big)^j \\
&\leq C\cdot 2^{6}\|\eta^1-\eta^2\|_{K_{p,q_1,q_2;1}^N}
\sum_{j=0}^2\Big(\|\eta^1\|_{K_{p,q_1,q_2;1}^N}+\|\eta^2\|_{K_{p,q_1,q_2;1}^N}\Big)^j.
\end{align*}
On the other hand, Lemma \ref{lem:nonl-F-3} with $(\eta^2,\Bu^2)=(0,0)$ yields
\begin{align*}
&\|\langle t\rangle(\BI+\BJ(\eta^1))\wtd\SSF(\eta^1,\Bu^1)\|_{L_p(\BR_+,L_r(\BR_-^N)^N)} \\
&\leq C\|(\eta^1,\Bu^1)\|_{K_{p,q_1,q_2}^N}\sum_{j=1}^5\|(\eta^1,\Bu^1)\|_{K_{p,q_1,q_2}^N}^j.
\end{align*}
It thus holds that
\begin{align*}
&\|\langle t \rangle I_1\|_{L_p(\BR+,L_r(\BR_-^N)^N)} \\
&\leq C
\Big\|\frac{1}{(1+\pd_N\CE\eta^1)^3}-\frac{1}{(1+\pd_N\CE\eta^2)^3} \Big\|_{L_\infty(\BR_-^N,L_\infty(\BR_-^N))} \\
&\times \|\langle t\rangle(\BI+\BJ(\eta^1))\wtd\SSF(\eta^1,\Bu^1)\|_{L_p(\BR_+,L_r(\BR_-^N)^N)}
\\
&\leq C
\|(\eta^1-\eta^2,\Bu^1-\Bu^2)\|_{K_{p,q_1,q_2}^N}
\sum_{j=1}^8\Big(\|(\eta^1,\Bu^1)\|_{K_{p,q_1,q_2}^N}+\|(\eta^2,\Bu^2)\|_{K_{p,q_1,q_2}^N}\Big)^j.
\end{align*}
Analogously, 
\begin{align*}
&\|\langle t \rangle I_2\|_{L_p(\BR+,L_r(\BR_-^N)^N)} \\
&\leq C
\|(\eta^1-\eta^2,\Bu^1-\Bu^2)\|_{K_{p,q_1,q_2}^N}
\sum_{j=1}^5\Big(\|(\eta^1,\Bu^1)\|_{K_{p,q_1,q_2}^N}+\|(\eta^2,\Bu^2)\|_{K_{p,q_1,q_2}^N}\Big)^j.
\end{align*}
These two inequalities yield the desired inequality. 
This completes the proof of Lemma \ref{lem:F-final}.
\end{proof}

We then have

\begin{prp}\label{prp:nonl-F}
Suppose that Assumption $\ref{asm:p-q}$ holds.
Then for any $(\eta^i,\Bu^i)\in K_{p,q_1,q_2}^N$, $i=1,2$, with \eqref{est-1/2}
\begin{align*}
&\sum_{r\in\{q_1/2,q_2\}}\|\langle t \rangle(\SSF(\eta^1,\Bu^1)-\SSF(\eta^2,\Bu^2))\|_{L_p(\BR_+,L_r(\BR_-^N)^N)} \\
&\leq M_3\|(\eta^1-\eta^2,\Bu^1-\Bu^2)\|_{K_{p,q_1,q_2}^N}
\sum_{j=1}^8\Big(\|(\eta^1,\Bu^1)\|_{K_{p,q_1,q_2}^N}+\|(\eta^2,\Bu^2)\|_{K_{p,q_1,q_2}^N}\Big)^j,
\end{align*}
where $M_3$ is a positive constant.
\end{prp}

\begin{proof}
We can prove the desired inequality by Lemmas \ref{lem:nonl-F-3} and \ref{lem:F-final} with some calculations,
so that we may omit the detailed proof.
\end{proof}

\subsection{Estimates of $\wtd \SSG(\eta,\Bu)$}

This subsection estimates $\widetilde\SSG(\eta,\Bu)$ given by \eqref{def:G-tilde}.
Let us define 
\begin{equation*}
\wtd \SSG_1(\eta,\Bu)=(\pd_j\CE_N\eta)\Bu
\end{equation*}
for $j=1,\dots,N$. The following lemma then holds.

\begin{lem}\label{lem:wtd-g-1}
Suppose that Assumption $\ref{asm:p-q}$ holds and let $r\in\{q_1/2,q_2\}$.
Then for any $(\eta^i,\Bu^i)\in K_{p,q_1,q_2}^N$, $i=1,2$,
\begin{align*}
&\|\langle t\rangle(\pd_t\wtd \SSG_1(\eta^1,\Bu^1)-\pd_t\wtd\SSG_1(\eta^2,\Bu^2))\|_{L_p(\BR_+,L_r(\BR_-^N)^N)} \\
&\leq 
C\|(\eta^1-\eta^2,\Bu^1-\Bu^2)\|_{K_{p,q_1,q_2}^N}
\Big(\|(\eta^1,\Bu^1)\|_{K_{p,q_1,q_2}^N}+\|(\eta^2,\Bu^2)\|_{K_{p,q_1,q_2}^N}\Big), \\
&\|\langle t \rangle( \wtd\SSG_1(\eta^1,\Bu^1)-\wtd\SSG_1(\eta^2,\Bu^2))\|_{L_p(\BR_+,L_r(\BR_-^N)^N)} \\
&\leq 
C\|(\eta^1-\eta^2,\Bu^1-\Bu^2)\|_{K_{p,q_1,q_2}^N}
\Big(\|(\eta^1,\Bu^1)\|_{K_{p,q_1,q_2}^N}+\|(\eta^2,\Bu^2)\|_{K_{p,q_1,q_2}^N}\Big),
\end{align*}
with some positive constant $C$.
\end{lem}

\begin{proof}
Recall $\SSF_5(\eta,\Bu)$ given by \eqref{dfn-F5} and write
\begin{align*}
\pd_t\wtd \SSG_1(\eta^1,\Bu^1)-\pd_t\wtd \SSG_1(\eta^2,\Bu^2)&=
(\pd_t\pd_j\CE_N\eta^1-\pd_t\pd_j\CE_N\eta^2)\Bu^1-(\pd_t\pd_j\CE_N\eta^2)(\Bu^1-\Bu^2) \\
&+\SSF_5(\eta^1,\Bu^1)-\SSF_5(\eta^2,\Bu^2).
\end{align*}
We set
\begin{align*}
I_1&=(\pd_t\pd_j\CE_N\eta^1-\pd_t\pd_j\CE_N\eta^2)\Bu^1, \\
I_2&=-(\pd_t\pd_j\CE_N\eta^2)(\Bu^1-\Bu^2), \quad
I_3=\SSF_5(\eta^1,\Bu^1)-\SSF_5(\eta^2,\Bu^2),
\end{align*}
i.e., 
\begin{equation*}
\pd_t\wtd \SSG_1(\eta^1,\Bu^1)-\pd_t\wtd \SSG_1(\eta^2,\Bu^2)=I_1+I_2+I_3.
\end{equation*}
In addition,
\begin{align*}
\wtd \SSG_1(\eta^1,\Bu^1)-\wtd \SSG_1(\eta^2,\Bu^2)
&=(\pd_j\CE_N\eta^1-\pd_j\CE_N\eta^2)\Bu^1+(\pd_N\CE_j\eta^2)(\Bu^1-\Bu^2) \\
&=:J_1+J_2. 
\end{align*}
Lemma \ref{lem:nonl-F-1} for $\SSF_5(\eta,\Bu)$ gives us 
\begin{align*}
&\|\langle t\rangle I_3\|_{L_p(\BR_+,L_r(\BR_-^N)^N)} \\
&\leq C\|(\eta^1-\eta^2,\Bu^1-\Bu^2)\|_{K_{p,q_1,q_2}^N}
\Big(\|(\eta^1,\Bu^1)\|_{K_{p,q_1,q_2}^N}+\|(\eta^2,\Bu^2)\|_{K_{p,q_1,q_2}^N}\Big).
\end{align*}
We use $(r_1,r_2)$ defined in Lemma \ref{lem:holder-type-ineq} in what follows.

{\bf Case 1}: $N=3$. We see  by \eqref{r-holder-1} that
\begin{align*}
\|I_1(t)\|_{L_r(\BR_-^N)}
&\leq \|\pd_t\pd_j\CE_N\eta^1(t)-\pd_t\pd_j\CE_N\eta^2(t)\|_{L_{r_2}(\BR_-^N)}
\|\Bu^1(t)\|_{L_{r_1(\BR_-^N)}}, \\
\|J_1(t)\|_{L_r(\BR_-^N)}
&\leq \|\pd_j\CE_N\eta^1(t)-\pd_j\CE_N\eta^2(t)\|_{L_{r_2}(\BR_-^N)}
\|\Bu^1(t)\|_{L_{r_1}(\BR_-^N)}.
\end{align*}
It thus follows from Lemma \ref{lem:main-embed3d} that
\begin{align*}
\|\langle t\rangle I_1\|_{L_p(\BR_+,L_r(\BR_-^N)^N)}
&\leq \|\langle t\rangle^{2/3}(\pd_t\pd_j\CE_N\eta^1-\pd_t\pd_j\CE_N\eta^2)\|_{L_p(\BR_+,L_{r_2}(\BR_-^N))} \\
&\times \|\langle t \rangle^{1/3}\Bu^1\|_{L_\infty(\BR_+,L_{r_1}(\BR_-^N)^N)} \\
&\leq C\|\eta^1-\eta^2\|_{K_{p,q_1,q_2;1}^N}\|\Bu^1\|_{K_{p,q_1,q_2;2}^N}, \\
\|\langle t\rangle J_1\|_{L_p(\BR_+,L_r(\BR_-^N)^N)}
&\leq 
\|\langle t\rangle^{2/3}(\pd_j\CE_N\eta^1-\pd_j\CE_N\eta^2)\|_{L_p(\BR_+,L_{r_2}(\BR_-^N))} \\
&\times \|\langle t \rangle^{1/3}\Bu^1\|_{L_\infty(\BR_+,L_{r_1}(\BR_-^N)^N)} \\
&\leq C\|\eta^1-\eta^2\|_{K_{p,q_1,q_2;1}^N}\|\Bu^1\|_{K_{p,q_1,q_2;2}^N}. 
\end{align*}
Analogously,
\begin{align*}
\|\langle t\rangle I_2\|_{L_p(\BR_+,L_r(\BR_-^N)^N)}
&\leq C\|\eta^2\|_{K_{p,q_1,q_2;1}^N}\|\Bu^1-\Bu^2\|_{K_{p,q_1,q_2;2}^N}, \\
\|\langle t\rangle J_2\|_{L_p(\BR_+,L_r(\BR_-^N)^N)}
&\leq C\|\eta^2\|_{K_{p,q_1,q_2;1}^N}\|\Bu^1-\Bu^2\|_{K_{p,q_1,q_2;2}^N}.
\end{align*}
The above inequalities for $I_1,I_2,I_3,J_1,J_2$ yield the desired inequality for $N=3$.

{\bf Case 2}: $N\geq 4$.
We replace $\langle t \rangle^{1/3}$ with $\langle t \rangle^{1/2}$
and $\langle t\rangle^{2/3}$ with $\langle t \rangle^{1/2}$ in Case 1,
and then we obtain the desired inequality for $N\geq 4$ from Lemma \ref{lem:main-embed4d}..
This completes the proof of Lemma \ref{lem:wtd-g-1}.
\end{proof}

By Lemma \ref{lem:wtd-g-1}, we immediately obtain

\begin{prp}\label{prp:nonl-tildeG}
Suppose that Assumption $\ref{asm:p-q}$ holds. 
Then for any $(\eta^i,\Bu^i)\in K_{p,q_1,q_2}^N$, $i=1,2$,
\begin{align*}
&\sum_{r\in\{q_1/2,q_2\}}
\|\langle t \rangle( \pd_t\wtd\SSG(\eta^1,\Bu^1)-\pd_t\wtd\SSG(\eta^2,\Bu^2))\|_{L_p(\BR_+,L_r(\BR_-^N)^N)} \\
&\leq 
M_4\|(\eta^1-\eta^2,\Bu^1-\Bu^2)\|_{K_{p,q_1,q_2}^N}
\Big(\|(\eta^1,\Bu^1)\|_{K_{p,q_1,q_2}^N}+\|(\eta^2,\Bu^2)\|_{K_{p,q_1,q_2}^N}\Big), \\
&\sum_{r\in\{q_1/2,q_2\}}
\|\langle t \rangle( \wtd\SSG(\eta^1,\Bu^1)-\wtd\SSG(\eta^2,\Bu^2))\|_{L_p(\BR_+,L_r(\BR_-^N)^N)} \\
&\leq 
M_5\|(\eta^1-\eta^2,\Bu^1-\Bu^2)\|_{K_{p,q_1,q_2}^N}
\Big(\|(\eta^1,\Bu^1)\|_{K_{p,q_1,q_2}^N}+\|(\eta^2,\Bu^2)\|_{K_{p,q_1,q_2}^N}\Big),
\end{align*}
with positive constants $M_4$ and $M_5$.
\end{prp}

\subsection{Estimates of $\SSG(\eta,\Bu)$}

This subsection estimates $\SSG(\eta,\Bu)$ given by \eqref{def:G}.
Let us define
\begin{equation}\label{dfn:K-L}
\SSK_{j,k,l}(\eta,\Bu)=(\pd_j\CE_N\eta)\pd_k u_l, \quad
\SSL_{j,k,l}(\eta)=(\pd_j\CE_N\eta)\pd_k\pd_l\CE_N\eta
\end{equation}
for $j,k,l=1,\dots,N$, where $u_l$ is the $l$th component of $\Bu$.
Here $\SSK_{j,k,l}(\eta,\Bu)$ is used to estimate $\SSG(\eta,\Bu)$,
while $\SSL_{j,k,l}(\eta)$ is used to estimate $\SSH(\eta,\Bu)$ given by \eqref{dfn:H}
in the next subsection.

Let us start with the following lemma.

\begin{lem}\label{lem-K-est}
Suppose that Assumption $\ref{asm:p-q}$ holds and let $r\in\{q_1/2,q_2\}$.
Let $Z=L_p(\BR_+,H_r^1(\BR_-^N))$ or $Z=H_p^{1/2}(\BR_+,L_r(\BR_-^N))$.
Then for any $(\eta^i,\Bu^i)\in K_{p,q_1,q_2}^N$, $i=1,2$,
\begin{align*}
&\|\langle t \rangle (\SSK_{j,k,l}(\eta^1,\Bu^1)-\SSK_{j,k,l}(\eta^2,\Bu^2))\|_{Z} \\
&\leq C\|(\eta^1-\eta^2,\Bu^1-\Bu^2)\|_{K_{p,q_1,q_2}^N}
\Big(\|(\eta^1,\Bu^1)\|_{K_{p,q_1,q_2}^N}+\|(\eta^2,\Bu^2)\|_{K_{p,q_1,q_2}^N}\Big), \\
&\|\langle t \rangle (\SSL_{j,k,l}(\eta^1)-\SSL_{j,k,l}(\eta^2))\|_{Z} \\
&\leq C\|(\eta^1-\eta^2,\Bu^1-\Bu^2)\|_{K_{p,q_1,q_2}^N}
\Big(\|(\eta^1,\Bu^1)\|_{K_{p,q_1,q_2}^N}+\|(\eta^2,\Bu^2)\|_{K_{p,q_1,q_2}^N}\Big),
\end{align*}
with some positive constant $C$.
\end{lem}

\begin{proof}
Since $\SSL_{j,k,l}(\eta)$ can be treated similarly to $\SSK_{j,k,l}(\eta,\Bu)$,
we consider in this proof $\SSK_{j,k,l}(\eta,\Bu)$ only.
Let us write 
\begin{align*}
\SSK_{j,k,l}(\eta^1,\Bu^1)-\SSK_{j,k,l}(\eta^2,\Bu^2)
&=(\pd_j\CE_N\eta^1-\pd_j\CE_N\eta^2)\pd_k u_l^1
+(\pd_j\CE_N\eta^2)(\pd_k u_l^1-\pd_k u_l^2) \\
&=:I_1+I_2.
\end{align*}
We use $(r_1,r_2)$ defined in Lemma \ref{lem:holder-type-ineq} in what follows.

{\bf Case 1}: $N=3$.
 It holds by \eqref{r-holder-2} that
\begin{equation*}
\|I_1(t)\|_{H_r^1(\BR_-^N)}\leq \|\pd_j\CE_N\eta^1-\pd_j\CE_N\eta^2\|_{H_{r_2}^1(\BR_-^N)}
\|\pd_k u_l^1\|_{H_{r_2}^1(\BR_-^N)},
\end{equation*}
which, combined with Lemma \ref{lem:main-embed3d}, furnishes
\begin{align*}
\|\langle t \rangle I_1\|_{L_p(\BR_+,H_{r}^1(\BR_-^N))}
&\leq \|\langle t \rangle^{2/3}(\pd_j\CE_N\eta^1-\pd_j\CE_N\eta^2)\|_{L_\infty(\BR_+,H_{r_2}^1(\BR_-^N))} \\
&\times\|\langle t \rangle^{1/3}\pd_k u_l^1\|_{L_p(\BR_+,H_{r_2}^1(\BR_-^N))} \\
&\leq C\|\eta^1-\eta^2\|_{K_{p,q_1,q_2;1}^N}\|\Bu^1\|_{K_{p,q_1,q_2;2}^N}.
\end{align*}
On the other hand, Lemma \ref{lem:H-half} and $1/r=1/r_1+1/r_2$ show that
\begin{align*}
\|\langle t \rangle I_1\|_{H_p^{1/2}(\BR_+,L_r(\BR_-^N))}
&\leq C\|\langle t \rangle^{2/3}(\pd_j\CE_N\eta^1-\pd_j\CE_N\eta^2)\|_{H_p^1(\BR_+,L_{r_1}(\BR_-^N))} \\
&\times \|\langle t \rangle^{1/3}\pd_k u_l^1\|_{H_p^{1/2}(\BR_+,L_{r_2}(\BR_-^N))},
\end{align*}
which, combined with Lemma \ref{lem:main-embed3d}, furnishes
\begin{equation*}
\|\langle t \rangle I_1\|_{H_p^{1/2}(\BR_+,L_r(\BR_-^N))}
\leq C\|\eta^1-\eta^2\|_{K_{p,q_1,q_2;1}^N}\|\Bu^1\|_{K_{p,q_1,q_2;2}^N}.
\end{equation*}

Summing up the inequalities above for $I_1$, we have
\begin{align*}
&\|\langle t \rangle I_1\|_{L_p(\BR_+,H_{r}^1(\BR_-^N))}
+\|\langle t \rangle I_1\|_{H_p^{1/2}(\BR_+,L_r(\BR_-^N))} \\
&\leq C\|\eta^1-\eta^2\|_{K_{p,q_1,q_2;1}^N}\|\Bu^1\|_{K_{p,q_1,q_2;2}^N}.
\end{align*}
Analogously,
\begin{align*}
&\|\langle t \rangle I_2\|_{L_p(\BR_+,H_{r}^1(\BR_-^N))}
+\|\langle t \rangle I_2\|_{H_p^{1/2}(\BR_+,L_r(\BR_-^N))} \\
&\leq C\|\eta^2\|_{K_{p,q_1,q_2;1}^N}\|\Bu^1-\Bu^2\|_{K_{p,q_1,q_2;2}^N}.
\end{align*}
We therefore obtain the desired inequalities for $\SSK_{j,k,l}(\eta,\Bu)$.

{\bf Case 2}: $N\geq 4$.
We replace  $\langle t \rangle^{1/3}$ with $\langle t \rangle^{1/2}$
and $\langle t \rangle^{2/3}$ with $\langle t \rangle^{1/2}$ in Case 1,
and then obtain the desired inequalities
from Lemma \ref{lem:main-embed4d}.
This completes the proof of Lemma \ref{lem-K-est}.
\end{proof}

By Lemma \ref{lem-K-est}, we immediately obtain

\begin{prp}\label{prp:nonl-G-1}
Suppose that Assumption $\ref{asm:p-q}$ holds.
Then for any $(\eta^i,\Bu^i)\in K_{p,q_1,q_2}^N$, $i=1,2$, 
\begin{align*}
&\sum_{r\in\{q_1/2,q_2\}}\|\langle t \rangle(\SSG(\eta^1,\Bu^1)-\SSG(\eta^2,\Bu^2))\|_{L_p(\BR_+,H_r^1(\BR_-^N))} \\
&\leq M_6\|(\eta^1-\eta^2,\Bu^1-\Bu^2)\|_{K_{p,q_1,q_2}^N}
\Big(\|(\eta^1,\Bu^1)\|_{K_{p,q_1,q_2}^N}+\|(\eta^2,\Bu^2)\|_{K_{p,q_1,q_2}^N}\Big), \\
&\sum_{r\in\{q_1/2,q_2\}}
\|\langle t \rangle(\SSG(\eta^1,\Bu^1)-\SSG(\eta^2,\Bu^2))\|_{H_p^{1/2}(\BR_+,L_r(\BR_-^N))} \\
&\leq M_7\|(\eta^1-\eta^2,\Bu^1-\Bu^2)\|_{K_{p,q_1,q_2}^N}
\Big(\|(\eta^1,\Bu^1)\|_{K_{p,q_1,q_2}^N}+\|(\eta^2,\Bu^2)\|_{K_{p,q_1,q_2}^N}\Big),
\end{align*}
with positive constant $M_6$ and $M_7$.
\end{prp}

\subsection{Estimates of $\SSH(\eta,\Bu)$}

This subsection estimates $\SSH(\eta,\Bu)$ given by \eqref{dfn:H}.
Recall $\CF(\eta)$ given by \eqref{dfn:F-eta}
and $\SSK_{j,k,l}(\eta,\Bu)$, $\SSL_{j,k,l}(\eta)$ given by \eqref{dfn:K-L}.
We define for multi-index $\alpha\in\BN_0^N$
\begin{equation}\label{dfn:K-L-v2}
\SSK_{j,k,l}^\alpha(\eta,\Bu)
=(\CF(\eta))^\alpha \SSK_{j,k,l}(\eta,\Bu), \quad 
\SSL_{j,k,l}^\alpha(\eta)
=(\CF(\eta))^\alpha\SSL_{j,k,l}(\eta).
\end{equation}

Let us first prove

\begin{lem}\label{lem:K-L-est}
Suppose that Assumption $\ref{asm:p-q}$ holds and let $r\in\{q_1/2,q_2\}$.
Let $Z=L_p(\BR_+,H_r^1(\BR_-^N))$ or $Z=H_p^{1/2}(\BR_+,L_r(\BR_-^N))$.
Then for any $(\eta^i,\Bu^i)\in K_{p,q_1,q_2}^N$, $i=1,2$,
and for any multi-index $\alpha\in\BN_0^N$ 
\begin{align*}
&\|\langle t \rangle (\SSK_{j,k,l}^\alpha(\eta^1,\Bu^1)-\SSK_{j,k,l}^\alpha(\eta^2,\Bu^2))\|_Z \\
&\leq C\|(\eta^1-\eta^2,\Bu^1-\Bu^2)\|_{K_{p,q_1,q_2}^N}
\Big(\|(\eta^1,\Bu^1)\|_{K_{p,q_1,q_2}^N}+\|(\eta^2,\Bu^2)\|_{K_{p,q_1,q_2}^N}\Big)^{|\alpha|+1}, \\
&\|\langle t \rangle (\SSL_{j,k,l}^\alpha(\eta^1)-\SSL_{j,k,l}^\alpha(\eta^2))\|_{Z} \\
&\leq C\|\eta^1-\eta^2\|_{K_{p,q_1,q_2;1}^N}
\Big(\|\eta^1\|_{K_{p,q_1,q_2;1}^N}+\|\eta^2\|_{K_{p,q_1,q_2;1}^N}\Big)^{|\alpha|+1},
\end{align*}
with a positive constant $C$.
\end{lem}

\begin{proof}
Since $\SSL_{j,k,l}^\alpha(\eta,\Bu)$ can be treated similarly to $\SSK_{j,k,l}^\alpha(\eta,\Bu)$,
we consider in this proof $\SSK_{j,k,l}^\alpha(\eta,\Bu)$ only.
In the case $|\alpha|=0$, the assertion is already proved in Lemma \ref{lem-K-est},
so that  we prove the case $|\alpha|\geq 1$ in what follows.

We use $(r_1,r_2)$ defined in Lemma \ref{lem:holder-type-ineq}.
Let us write 
\begin{align*}
&\SSK_{j,k,l}^\alpha(\eta^1,\Bu^1)-\SSK_{j,k,l}^\alpha(\eta^2,\Bu^2) \\
&=((\CF(\eta^1))^\alpha-(\CF(\eta^2))^\alpha)\SSK_{j,k,l}(\eta^1,\Bu^1)
+ (\CF(\eta^2))^\alpha(\SSK_{j,k,l}(\eta^1,\Bu^1)-\SSK_{j,k,l}(\eta^2,\Bu^2)) \\
&=:I_1+I_2,
\end{align*}
and then \eqref{r-holder-2} gives us
\begin{align*}
\|I_1(t)\|_{H_r^1(\BR_-^N)}
&\leq 
C \|(\CF(\eta^1(t)))^\alpha-(\CF(\eta^2(t)))^\alpha\|_{H_{r_2}^1(\BR_-^N)} \\
&\times \|\SSK_{j,k,l}(\eta^1(t),\Bu^1(t))\|_{H_{r_2}^1(\BR_-^N)}.
\end{align*}
This implies
\begin{align}\label{est-I1-K1}
\|\langle t \rangle I_1\|_{L_p(\BR_+,H_r^1(\BR_-^N))}
&\leq C\|(\CF(\eta^1))^\alpha-(\CF(\eta^2))^\alpha\|_{L_\infty(\BR_+,H_{r_2}^1(\BR_-^N))} \notag \\
&\times \|\langle t\rangle \SSK_{j,k,l}(\eta^1,\Bu^1)\|_{L_p(\BR_+,H_{r_2}^1(\BR_-^N))}.
\end{align}
Lemma \ref{lem:nonl-F-2} shows that
\begin{align}\label{eq-F-eta-est}
&\|(\CF(\eta^1))^\alpha-(\CF(\eta^2))^\alpha\|_{L_\infty(\BR_+,H_{r_2}^1(\BR_-^N))} \notag \\
&\leq C
\|\eta^1-\eta^2\|_{K_{p,q_1,q_2;1}^N} \Big(\|\eta^1\|_{K_{p,q_1,q_2;1}^N}+\|\eta^2\|_{K_{p,q_1,q_2;2}^N}\Big)^{|\alpha|-1},
\end{align}
while Lemma \ref{lem-K-est} with $(\eta^2,\Bu^2)=(0,0)$ yields 
\begin{equation}\label{est-K-LH}
\|\langle t\rangle \SSK_{j,k,l}(\eta^1,\Bu^1)\|_{L_p(\BR_+,H_s^1(\BR_-^N))}
\leq C\|(\eta^1,\Bu^1)\|_{K_{p,q_1,q_2}^N}^2
\end{equation}
for $s\in\{q_1/2,q_2\}$.
Since $q_1/2<q_1<q_2$, Lemma \ref{lem:int-p-v2} \eqref{lem:int-p-1-2} gives us
\begin{align*}
&\|\langle t\rangle \SSK_{j,k,l}(\eta^1,\Bu^1)\|_{L_p(\BR_+,H_{q_1}^1(\BR_-^N))} \\
&\leq C\|\langle t\rangle \SSK_{j,k,l}(\eta^1,\Bu^1)\|_{L_p(\BR_+,H_{q_1/2}^1(\BR_-^N))}^{1-\theta}
\|\langle t\rangle \SSK_{j,k,l}(\eta^1,\Bu^1)\|_{L_p(\BR_+,H_{q_2}^1(\BR_-^N))}^\theta
\end{align*}
for some $\theta\in(0,1)$.
This inequality together with $a^{1-\theta}b^\theta\leq C(a+b)$ for $a,b\geq 0$ shows that
\begin{equation*}
\|\langle t\rangle \SSK_{j,k,l}(\eta^1,\Bu^1)\|_{L_p(\BR_+,H_{q_1}^1(\BR_-^N))}
\leq C\sum_{r\in\{q_1/2,q_2\}}
\|\langle t\rangle \SSK_{j,k,l}(\eta^1,\Bu^1)\|_{L_p(\BR_+,H_{r}^1(\BR_-^N))},
\end{equation*}
which, combined with \eqref{est-K-LH}, furnishes
\begin{equation*}
\|\langle t\rangle \SSK_{j,k,l}(\eta^1,\Bu^1)\|_{L_p(\BR_+,H_{q_1}^1(\BR_-^N))}
\leq C\|(\eta^1,\Bu^1)\|_{K_{p,q_1,q_2}^N}^2.
\end{equation*}
This inequality together with \eqref{est-I1-K1},
\eqref{eq-F-eta-est}, and \eqref{est-K-LH} with $s=q_2$ shows that
\begin{align*}
&\|\langle t \rangle I_1\|_{L_p(\BR_+,H_r^1(\BR_-^N))} \\
&\leq C\|(\eta^1-\eta^2,\Bu^1-\Bu^2)\|_{K_{p,q_1,q_2}^N}
\Big(\|(\eta^1,\Bu^1)\|_{K_{p,q_1,q_2}^N}+\|(\eta^2,\Bu^2)\|_{K_{p,q_1,q_2}^N}\Big)^{|\alpha|+1}.
\end{align*}
Analogously, the last inequality holds with $I_1$ replaced by $I_2$.
The desired inequality thus holds for the case $Z=L_p(\BR_+,H_r^1(\BR_-^N))$.

We next consider the case $Z=H_p^{1/2}(\BR_+,L_r(\BR_-^N))$.
Lemma \ref{lem:H-half} gives us
\begin{align*}
\|\langle t\rangle I_1\|_{H_p^{1/2}(\BR_+,L_r(\BR_-^N))}
&\leq C\|(\CF(\eta^1))^\alpha-(\CF(\eta^2))^\alpha\|_{H_p^1(\BR_+,L_\infty(\BR_-^N))} \\
&\times \|\langle t\rangle \SSK_{j,k,l}(\eta^1,\Bu^1)\|_{H_p^{1/2}(\BR_+,L_r(\BR_-^N))}.
\end{align*}
Lemma \ref{lem:nonl-F-2} yields
\begin{align*}
&\|(\CF(\eta^1))^\alpha-(\CF(\eta^2))^\alpha\|_{H_p^1(\BR_+,L_\infty(\BR_-^N))} \\
&\leq C
\|\eta^1-\eta^2\|_{K_{p,q_1,q_2;1}^N} \Big(\|\eta^1\|_{K_{p,q_1,q_2;1}^N}+\|\eta^2\|_{K_{p,q_1,q_2;2}^N}\Big)^{|\alpha|-1},
\end{align*}
while Lemma \ref{lem-K-est} with $(\eta^2,\Bu^2)=(0,0)$ yields
\begin{equation*}
\|\langle t\rangle \SSK_{j,k,l}(\eta^1,\Bu^1)\|_{H_p^{1/2}(\BR_+,L_r(\BR_-^N))}\leq C\|(\eta^1,\Bu^1)\|_{K_{p,q_1,q_2}^N}^2.
\end{equation*}
These three inequalities furnish
\begin{align*}
&\|\langle t\rangle I_1\|_{H_p^{1/2}(\BR_+,L_r(\BR_-^N))} \\
&\leq C\|(\eta^1-\eta^2,\Bu^1-\Bu^2)\|_{K_{p,q_1,q_2}^N}
\Big(\|(\eta^1,\Bu^1)\|_{K_{p,q_1,q_2}^N}+\|(\eta^2,\Bu^2)\|_{K_{p,q_1,q_2}^N}\Big)^{|\alpha|+1}.
\end{align*}
Analogously, the last inequality holds with $I_1$ replaced by $I_2$.
The desired inequality thus holds for the case $Z=H_p^{1/2}(\BR_+,L_r(\BR_-^N))$.
This completes the proof of Lemma \ref{lem:K-L-est}.
\end{proof}

Recall \eqref{matrix:JK}--\eqref{normal-vec2}.
By Lemma \ref{lem:K-L-est}, we immediately obtain

\begin{lem}\label{lem:tilde-E}
Suppose that Assumption $\ref{asm:p-q}$ holds and let $r\in\{q_1/2,q_2\}$.
Let $Z=L_p(\BR_+,H_r^1(\BR_-^N)^N)$ or $Z=H_p^{1/2}(\BR_+,L_r(\BR_-^N)^N)$.
Then for any $(\eta^i,\Bu^i)\in K_{p,q_1,q_2}^N$, $i=1,2$, 
\begin{align*}
&\|\langle t\rangle  (\BD(\Bu^1)\wht\Bn(\eta^1)
-\BD(\Bu^2)\wht\Bn(\eta^2))\|_{Z} \\
&\leq C
\|(\eta^1-\eta^2,\Bu^1-\Bu^2)\|_{K_{p,q_1,q_2}^N}
\Big(\|(\eta^1,\Bu^1)\|_{K_{p,q_1,q_2}^N}+\|(\eta^2,\Bu^2)\|_{K_{p,q_1,q_2}^N}\Big), \\
&\|\langle t\rangle (\wtd\BE(\eta^1,\Bu^1)\Bn(\eta^1)
-\wtd\BE(\eta^2,\Bu^2)\Bn(\eta^2))\|_{Z} \\
&\leq C
\|(\eta^1-\eta^2,\Bu^1-\Bu^2)\|_{K_{p,q_1,q_2}^N}
\sum_{j=1}^2\Big(\|(\eta^1,\Bu^1)\|_{K_{p,q_1,q_2}^N}+\|(\eta^2,\Bu^2)\|_{K_{p,q_1,q_2}^N}\Big)^j, \\
&\|\langle t\rangle (\BK(\eta^1)\BD(\Bu^1)\Bn(\eta^1)
-\BK(\eta^2)\BD(\Bu^2)\Bn(\eta^2))\|_{Z} \\
&\leq C
\|(\eta^1-\eta^2,\Bu^1-\Bu^2)\|_{K_{p,q_1,q_2}^N}
\sum_{j=1}^2\Big(\|(\eta^1,\Bu^1)\|_{K_{p,q_1,q_2}^N}+\|(\eta^2,\Bu^2)\|_{K_{p,q_1,q_2}^N}\Big)^j, \\
&\|\langle t\rangle (\BK(\eta^1)\wtd\BE(\eta^1,\Bu^1)\Bn(\eta^1)
-\BK(\eta^1)\wtd\BE(\eta^2,\Bu^2)\Bn(\eta^2))\|_{Z} \\
&\leq C
\|(\eta^1-\eta^2,\Bu^1-\Bu^2)\|_{K_{p,q_1,q_2}^N}
\sum_{j=1}^3\Big(\|(\eta^1,\Bu^1)\|_{K_{p,q_1,q_2}^N}+\|(\eta^2,\Bu^2)\|_{K_{p,q_1,q_2}^N}\Big)^j,
\end{align*}
with a positive constant $C$.
\end{lem}

\begin{proof}
We can prove the desired inequalities from Lemma \ref{lem:K-L-est} with some calculations,
so that we may omit the detailed proof. 
\end{proof}

We next prove

\begin{lem}\label{lem:F-est}
Suppose that Assumption $\ref{asm:p-q}$ holds and let $q\in\{q_1,q_2\}$.
Let $Y=L_\infty(\BR_+,H_q^1(\BR_-^N))$ or $Y=H_p^{1}(\BR_+,L_\infty(\BR_-^N))$.
Then the following assertions hold.
\begin{enumerate}[$(1)$]
\item
Let $F$ be a smooth function on $[0,\infty)$ with 
\begin{equation*}
\sup_{s\geq 0}\Big(|F(s)|+|F'(s)|+|F''(s)|\Big)<\infty.
\end{equation*}
Then for any $\eta^i \in K_{p,q_1,q_2;1}^N$, $i=1,2$, 
\begin{align}\label{F-bibun-est-1}
&\|F(|\nabla'\CE_N\eta^1|^2)-F(|\nabla'\CE_N\eta^2|^2)\|_Y \notag \\
&\leq C\|\eta^1-\eta^2\|_{K_{p,q_1,q_2;1}^N}
\sum_{j=1}^3\Big(\|\eta^1\|_{K_{p,q_1,q_2;1}^N}+\|\eta^2\|_{K_{p,q_1,q_2;1}^N}\Big)^j,
\end{align}
with a positive constant $C$.
\item
Let $G$ be a smooth function on $[-3/4,\infty)$ with 
\begin{equation*}
\sup_{s\geq -3/4}\Big(|G(s)|+|G'(s)|+|G''(s)|\Big)<\infty.
\end{equation*}
Then for any $\eta^i \in K_{p,q_1,q_2;1}^N$, $i=1,2$, satisfying \eqref{est-1/2}
\begin{align*}
&\|G(\pd_N\CE_N\eta^1)-G(\pd_N\CE_N\eta^2)\|_Y \\
&\leq C\|\eta^1-\eta^2\|_{K_{p,q_1,q_2;1}^N}
\sum_{j=0}^1\Big(\|\eta^1\|_{K_{p,q_1,q_2;1}^N}+\|\eta^2\|_{K_{p,q_1,q_2;1}^N}\Big)^j,
\end{align*}
with a positive constant $C$.
\end{enumerate}
\end{lem}

\begin{proof}
(1) We first consider the case $Y=L_\infty(\BR_+,H_q^1(\BR_-^N))$.
Recall $\CF(\eta)$ given by \eqref{dfn:F-eta}
and set $\alpha_j=(0,\dots,\underset{\text{$j$ th}}{2},\dots0)$. 
Since 
\begin{equation*}
|\nabla'\CE_N\eta^i|^2=\sum_{j=1}^{N-1}(\pd_j\CE_N\eta^i)^2=\sum_{j=1}^{N-1}(\CF(\eta^i))^{\alpha_j}
\quad \text{for $i=1,2$,}
\end{equation*}
it holds that
\begin{align*}
&F(|\nabla'\CE_N\eta^1|^2)-F(|\nabla'\CE_N\eta^2|^2)
=\int_0^1 \frac{d}{d\theta} F(\theta|\nabla'\CE_N\eta^1|^2+(1-\theta)|\nabla'\CE_N\eta^2|^2)\intd \theta \\
&=\int_0^1  F'(\theta|\nabla'\CE_N\eta^1|^2+(1-\theta)|\nabla'\CE_N\eta^2|^2)\intd \theta 
\sum_{j=1}^{N-1}\Big((\CF(\eta^1))^{\alpha_j}-(\CF(\eta^2))^{\alpha_j}\Big).
\end{align*}
In addition, for $l=1,\dots,N$,
\begin{align*}
&\pd_l(F(|\nabla'\CE_N\eta^1|^2)-F(|\nabla'\CE_N\eta^2|^2)) \\
&=F'(|\nabla'\CE_N\eta^1|^2)\pd_l|\nabla'\CE_N\eta^1|^2-F'(|\nabla'\CE_N\eta^2|^2)\pd_l|\nabla'\CE_N\eta^2|^2 \\
&=(F'(|\nabla'\CE_N\eta^1|^2)-F'(|\nabla'\CE_N\eta^2|^2))\pd_l|\nabla'\CE_N\eta^1|^2 \\
&+F'(|\nabla'\CE_N\eta^2|^2)\pd_l(|\nabla'\CE_N\eta^1|^2-|\nabla'\CE_N\eta^2|^2) \\
&=\int_0^1 F''(\theta|\nabla'\CE_N\eta^1|^2+(1-\theta)|\nabla'\CE_N\eta^2|^2)\intd\theta \\
&\times 
\sum_{j=1}^{N-1}\Big((\CF(\eta^1))^{\alpha_j}-(\CF(\eta^2))^{\alpha_j}\Big)
\sum_{k=1}^{N-1}\pd_l(\CF(\eta^1))^{\alpha_k} \\
&+F'(|\nabla'\CE_N\eta^2|^2)\sum_{j=1}^{N-1}\pd_l\Big((\CF(\eta^1))^{\alpha_j}-(\CF(\eta^2))^{\alpha_j}\Big).
\end{align*}

Let $s=p$ or $s=\infty$. From the formulas above, we observe that
\begin{align*}
&\|F(|\nabla'\CE_N\eta^1|^2)-F(|\nabla'\CE_N\eta^2|^2)\|_{L_s(\BR_+,L_q(\BR_-^N))} \notag \\
&\leq \int_0^1\|F'(\theta|\nabla'\CE_N\eta^1|^2+(1-\theta)|\nabla'\CE_N\eta^2|^2)\|_{L_\infty(\BR_+,L_\infty(\BR_-^N))}\intd \theta \notag \\
&\times \sum_{j=1}^{N-1}\|(\CF(\eta^1))^{\alpha_j}-(\CF(\eta^2))^{\alpha_j}\|_{L_s(\BR_+,L_q(\BR_-^N))} \notag \\
&\leq \Big(\sup_{s\geq 0}|F'(s)|\Big)
\sum_{j=1}^{N-1}\|(\CF(\eta^1))^{\alpha_j}-(\CF(\eta^2))^{\alpha_j}\|_{L_s(\BR_+,L_q(\BR_-^N))}
\end{align*}
and
\begin{align*}
&\|\pd_l(F(|\nabla'\CE_N\eta^1|^2)-F(|\nabla'\CE_N\eta^2|^2)) \|_{L_s(\BR_+,L_q(\BR_-^N))} \notag \\
&\leq 
\int_0^1 \|F''(\theta|\nabla'\CE_N\eta^1|^2+(1-\theta)|\nabla'\CE_N\eta^2|^2)\|_{L_\infty(\BR_+,L_\infty(\BR_-^N))}\intd\theta \notag \\
&\times \sum_{j,k=1}^{N-1}
\|(\CF(\eta^1))^{\alpha_j}-(\CF(\eta^2))^{\alpha_j}\|_{L_\infty(\BR_+,L_\infty(\BR_-^N))}
\|(\CF(\eta^1))^{\alpha_k}\|_{L_s(\BR_+,H_q^1(\BR_-^N))} \notag \\
&+\|F'(|\nabla'\CE_N\eta^2|^2)\|_{L_\infty(\BR_+,L_\infty(\BR_-^N))} 
\sum_{j=1}^{N-1}\|(\CF(\eta^1))^{\alpha_j}-(\CF(\eta^2))^{\alpha_j}\|_{L_s(\BR_+,H_q^1(\BR_-^N))} \notag \\
&\leq \Big(\sup_{s\geq 0}|F''(s)|\Big)
\sum_{j,k=1}^{N-1}
\|(\CF(\eta^1))^{\alpha_j}-(\CF(\eta^2))^{\alpha_j}\|_{L_\infty(\BR_+,L_\infty(\BR_-^N))} \notag \\
&\times \|(\CF(\eta^1))^{\alpha_k}\|_{L_s(\BR_+,H_q^1(\BR_-^N))} \notag \\
&+\Big(\sup_{s\geq 0}|F'(s)|\Big)
\sum_{j=1}^{N-1}\|(\CF(\eta^1))^{\alpha_j}-(\CF(\eta^2))^{\alpha_j}\|_{L_s(\BR_+,H_q^1(\BR_-^N))}.
\end{align*}
Notice that by Lemma \ref{lem:nonl-F-2} with $\eta^2=0$
\begin{equation*}
\|(\CF(\eta^1))^{\alpha_k}\|_{L_s(\BR_+,H_q^1(\BR_-^N))}
\leq C\|\eta^1\|_{K_{p,q_1,q_2;1}^N}^2.
\end{equation*}
Lemma \ref{lem:nonl-F-2} thus shows that
\begin{align*}
&\|F(|\nabla'\CE_N\eta^1|^2)-F(|\nabla'\CE_N\eta^2|^2)\|_{L_s(\BR_+,L_q(\BR_-^N))} \\
&\leq C\|\eta^1-\eta^2\|_{K_{p,q_1,q_2;1}^N}\Big(\|\eta^1\|_{K_{p,q_1,q_2;1}^N}+\|\eta^2\|_{K_{p,q_1,q_2;1}^N}\Big)
\end{align*}
and 
\begin{align*}
&\|\nabla(F(|\nabla'\CE_N\eta^1|^2)-F(|\nabla'\CE_N\eta^2|^2))\|_{L_s(\BR_+,L_q(\BR_-^N)^N)} \\
&\leq C\|\eta^1-\eta^2\|_{K_{p,q_1,q_2;1}^N}\Big(\|\eta^1\|_{K_{p,q_1,q_2;1}^N}+\|\eta^2\|_{K_{p,q_1,q_2;1}^N}\Big)
\Big(\|\eta^1\|_{K_{p,q_1,q_2;1}^N}^2+1\Big)
\end{align*}
Hence
\begin{align}\label{ineq-F-eta}
&\|F(|\nabla'\CE_N\eta^1|^2)-F(|\nabla'\CE_N\eta^2|^2)\|_{L_s(\BR_+,H_q^1(\BR_-^N))} \notag \\
&\leq C\|\eta^1-\eta^2\|_{K_{p,q_1,q_2;1}^N}
\sum_{j=1}^3\Big(\|\eta^1\|_{K_{p,q_1,q_2;1}^N}+\|\eta^2\|_{K_{p,q_1,q_2;1}^N}\Big)^j
\end{align}
for $s=p$ or $s=\infty$.
From this inequality with $s=\infty$, we obtain \eqref{F-bibun-est-1} for $Y=L_\infty(\BR_+,H_q^1(\BR_-^N))$.

We next consider the case $Y=H_p^1(\BR_+,L_\infty(\BR_-^N))$.
Since $L_p(\BR_+,H_{q_2}^1(\BR_-^N))$ is continuously embedded into $L_p(\BR_+,L_\infty(\BR_-^N))$ by 
Lemma \ref{lem:fund-embed} \eqref{lem:fund-embed-3},
we obtain from \eqref{ineq-F-eta} with $s=p$ and $q=q_2$
\begin{align}\label{est-F-inf-1}
&\|F(|\nabla'\CE_N\eta^1|^2)-F(|\nabla'\CE_N\eta^2|^2)\|_{L_p(\BR_+,L_\infty(\BR_-^N))} \notag \\
&\leq C\|\eta^1-\eta^2\|_{K_{p,q_1,q_2;1}^N}
\sum_{j=1}^3\Big(\|\eta^1\|_{K_{p,q_1,q_2;1}^N}+\|\eta^2\|_{K_{p,q_1,q_2;1}^N}\Big)^j.
\end{align}

Let us compute the time derivative. We see that
\begin{align*}
&\pd_t(F(|\nabla'\CE_N\eta^1|^2)-F(|\nabla'\CE_N\eta^2|^2)) \\
&=F'(|\nabla'\CE_N\eta^1|^2)\pd_t|\nabla'\CE_N\eta^1|^2-F'(|\nabla'\CE_N\eta^2|^2)\pd_t |\nabla'\CE_N\eta^2|^2 \\
&= (F'(|\nabla'\CE_N\eta^1|^2)-F'(|\nabla'\CE_N\eta^2|^2))\pd_t|\nabla'\CE_N\eta^1|^2 \\
&+F'(|\nabla'\CE_N\eta^2|^2) \pd_t(|\nabla'\CE_N\eta^1|^2-|\nabla'\CE_N\eta^2|^2) \\
&=\int_0^1 F''(\theta|\nabla'\CE_N\eta^1|^2+(1-\theta)|\nabla'\CE_N\eta^2|^2)\intd\theta \\
&\times 
\sum_{j=1}^{N-1}\Big((\CF(\eta^1))^{\alpha_j}-(\CF(\eta^2))^{\alpha_j}\Big)
\sum_{k=1}^{N-1}\pd_t(\CF(\eta^1))^{\alpha_k} \\
&+F'(|\nabla'\CE_N\eta^2|^2)\sum_{j=1}^{N-1}\pd_t\Big((\CF(\eta^1))^{\alpha_j}-(\CF(\eta^2))^{\alpha_j}\Big).
\end{align*}
This formula gives us 
\begin{align*}
&\|\pd_t(F(|\nabla'\CE_N\eta^1|^2)-F(|\nabla'\CE_N\eta^2|^2))\|_{L_p(\BR_+,L_\infty(\BR_-^N))} \\
&\leq \Big(\sup_{s\geq 0}|F''(s)|\Big)
\sum_{j,k=1}^{N-1}\|(\CF(\eta^1))^{\alpha_j}-(\CF(\eta^2))^{\alpha_j}\|_{L_\infty(\BR_+,L_\infty(\BR_-^N))} \\
&\times \|(\CF(\eta^1))^{\alpha_k}\|_{H_p^1(\BR_+,L_\infty(\BR_-^N))} \\
&+\Big(\sup_{s\geq 0}|F'(s)|\Big)\sum_{j=1}^{N-1}\|(\CF(\eta^1))^{\alpha_j}-(\CF(\eta^2))^{\alpha_j}\|_{H_p^1(\BR_+,L_\infty(\BR_-^N))}.
\end{align*}
Notice that by Lemma \ref{lem:nonl-F-2} with $\eta^2=0$
\begin{equation*}
\|(\CF(\eta^1))^{\alpha_k}\|_{H_p^1(\BR_+,L_\infty(\BR_-^N))}
\leq C\|\eta^1\|_{K_{p,q_1,q_2;1}^N}^2.
\end{equation*}
Lemma \ref{lem:nonl-F-2} thus shows that
\begin{align*}
&\|\pd_t(F(|\nabla'\CE_N\eta^1|^2)-F(|\nabla'\CE_N\eta^2|^2))\|_{L_p(\BR_+,L_\infty(\BR_-^N))} \\
&\leq C
\|\eta^1-\eta^2\|_{K_{p,q_1,q_2;1}^N}\Big(\|\eta^1\|_{K_{p,q_1,q_2;1}^N}+\|\eta^2\|_{K_{p,q_1,q_2;1}^N}\Big)
\Big(\|\eta^1\|_{K_{p,q_1,q_2;1}^N}^2+1\Big).
\end{align*}
This inequality together with \eqref{est-F-inf-1} yields \eqref{F-bibun-est-1} for $Y=H_p^1(\BR_+,L_\infty(\BR_-^N))$.
This completes the proof of (1).

(2) The proof is similar to (1), so that the detailed proof may be omitted.
This completes the proof of Lemma \ref{lem:F-est}.
\end{proof}

Define
\begin{equation*}
F_1(s)=\frac{1}{(1+\sqrt{1+s})\sqrt{1+s}}, \quad 
F_2(s)=\frac{1}{(1+s)^{3/2}}, \quad G(s)=\frac{1}{1+s},
\end{equation*}
and $\beta_j=(0,\dots,\underset{\text{$j$th}}{1},\dots,0)$.
Recall \eqref{matrix:JK}--\eqref{normal-vec2}, \eqref{dfn:H-2}, and \eqref{dfn:K-L-v2}.
Then $\SSH_\kappa(\eta)$ can be written as
\begin{equation}\label{form-H-kappa}
\SSH_\kappa(\eta)=F_1(|\nabla'\CE_N\eta|^2)\sum_{j,k=1}^{N-1}\SSL_{j,k.k}^{\beta_j}(\eta)
+
F_2(|\nabla'\CE_N\eta|^2)\sum_{j,k=1}^{N-1}\SSL_{j,j,k}^{\beta_k}(\eta),
\end{equation}
and also
\begin{align}\label{form-E-normal}
\BE(\eta,\Bu)\Bn(\eta)
&=G(\pd_N\CE_N\eta)\wtd \BE(\eta,\Bu)\Bn(\eta), \notag \\
\BK(\eta ) \BE(\eta,\Bu)\Bn(\eta)
&=G(\pd_N\CE_N\eta)\BK(\eta)\wtd \BE(\eta,\Bu)\Bn(\eta).
\end{align}

Let us prove the following lemma.

\begin{lem}\label{lem:H-kapppa}
Suppose that Assumption $\ref{asm:p-q}$ holds and let $r\in\{q_1/2,q_2\}$.
Let $Z=L_p(\BR_+,H_r^1(\BR_-^N)^N)$ or $Z=H_p^{1/2}(\BR_+,L_r(\BR_-^N)^N)$.
Then the following assertions hold.
\begin{enumerate}[$(1)$]
\item\label{lem:H-kapppa-1}
For any $(\eta^i,\Bu^i)\in K_{p,q_1,q_2}^N$, $i=1,2$,
\begin{align*}
&\|\langle t \rangle(\SSH_\kappa(\eta^1)\Be_N-\SSH_\kappa(\eta^2)\Be_N)\|_{Z} \\
&\leq C\|\eta^1-\eta^2\|_{K_{p,q_1,q_2;1}^N}
\sum_{j=1}^6\Big(\|\eta^1\|_{K_{p,q_1,q_2;1}^N}+\|\eta^2\|_{K_{p,q_1,q_2;1}^N}\Big)^j,
\end{align*}
with a positive constant $C$.
\item\label{lem:H-kapppa-2}
For any $(\eta^i,\Bu^i)\in K_{p,q_1,q_2}^N$, $i=1,2$, satisfying \eqref{est-1/2}
\begin{align*}
&\|\langle t \rangle(\BE(\eta^1,\Bu^1)\Bn(\eta^1)-\BE(\eta^2,\Bu^2)\Bn(\eta^2))\|_{Z} \\
&\leq C\|(\eta^1-\eta^2,\Bu^1-\Bu^2)\|_{K_{p,q_1,q_2}^N}
\sum_{j=1}^4\Big(\|(\eta^1,\Bu^1)\|_{K_{p,q_1,q_2}^N}+\|(\eta^2,\Bu^2)\|_{K_{p,q_1,q_2}^N}\Big)^j,
\end{align*}
with a positive constant $C$.
\item\label{lem:H-kapppa-3}
For any $(\eta^i,\Bu^i)\in K_{p,q_1,q_2}^N$, $i=1,2$, satisfying \eqref{est-1/2}
\begin{align*}
&\|\langle t \rangle(\BK(\eta^1)\BE(\eta^1,\Bu^1)\Bn(\eta^1)-\BK(\eta^2)\BE(\eta^2,\Bu^2)\Bn(\eta^2))\|_{Z} \\
&\leq C\|(\eta^1-\eta^2,\Bu^1-\Bu^2)\|_{K_{p,q_1,q_2}^N}
\sum_{j=1}^5\Big(\|(\eta^1,\Bu^1)\|_{K_{p,q_1,q_2}^N}+\|(\eta^2,\Bu^2)\|_{K_{p,q_1,q_2}^N}\Big)^j,
\end{align*}
with a positive constant $C$.
\item\label{lem:H-kapppa-4}
For any $(\eta^i,\Bu^i)\in K_{p,q_1,q_2}^N$, $i=1,2$, satisfying \eqref{est-1/2}
\begin{align*}
&\|\langle t \rangle(\wtd\SSH(\eta^1,\Bu^1)-\wtd\SSH(\eta^2,\Bu^2))\|_{Z} \\
&\leq C\|(\eta^1-\eta^2,\Bu^1-\Bu^2)\|_{K_{p,q_1,q_2}^N}
\sum_{j=1}^5\Big(\|(\eta^1,\Bu^1)\|_{K_{p,q_1,q_2}^N}+\|(\eta^2,\Bu^2)\|_{K_{p,q_1,q_2}^N}\Big)^j,
\end{align*}
with a positive constant $C$,
where $\wtd\SSH(\eta,\Bu)$ is given by \eqref{dfn:H-2}.
\end{enumerate}
\end{lem}

\begin{proof}
We use $(r_1,r_2)$ defined in Lemma \ref{lem:holder-type-ineq} in this proof.

(1) By \eqref{form-H-kappa}, we see that
\begin{align*}
\SSH_\kappa(\eta^1)-\SSH_\kappa(\eta^2)
&=(F_1(|\nabla'\CE_N\eta^1|^2)-F_1(|\nabla'\CE_N\eta^2|^2))\sum_{j,k=1}^{N-1}\SSL_{j,k,k}^{\beta_j}(\eta^1) \\
&+(F_1(|\nabla'\CE_N\eta^2|^2)-F_1(0))\sum_{j,k=1}^{N-1}(\SSL_{j,k.k}^{\beta_j}(\eta^1)-\SSL_{j,k.k}^{\beta_j}(\eta^2)) \\
&+F_1(0)\sum_{j,k=1}^{N-1}(\SSL_{j,k.k}^{\beta_j}(\eta^1)-\SSL_{j,k.k}^{\beta_j}(\eta^2)) \\
&+(F_2(|\nabla'\CE_N\eta^1|^2)-F_2(|\nabla'\CE_N\eta^2|^2))\sum_{j,k=1}^{N-1}\SSL_{j,j.k}^{\beta_k}(\eta^1) \\
&+(F_2(|\nabla'\CE_N\eta^2|^2)-F_2(0))\sum_{j,k=1}^{N-1}(\SSL_{j,j.k}^{\beta_k}(\eta^1)-\SSL_{j,j.k}^{\beta_k}(\eta^2))  \\
&+F_2(0)\sum_{j,k=1}^{N-1}(\SSL_{j,j.k}^{\beta_k}(\eta^1)-\SSL_{j,j.k}^{\beta_k}(\eta^2))  \\
&=:I_1+I_2+I_3+I_4+I_5+I_6.
\end{align*}

Let us first consider $I_1$.
It holds by \eqref{r-holder-2} that 
\begin{align*}
&\|I_1(t)\|_{H_r^1(\BR\-^N)} \\
&\leq C\|F_1(|\nabla'\CE_N\eta^1(t)|^2)-F_1(|\nabla'\CE_N\eta^2(t)|^2)\|_{H_{r_2}^1(\BR\-^N)} \\
&\times \sum_{j,k=1}^{N-1}\|\SSL_{j,k,k}^{\beta_j}(\eta^1(t))\|_{H_{r_2}^1(\BR_-^N)},
\end{align*}
which gives us
\begin{align}\label{est-I1-s-tension}
&\|\langle t\rangle I_1\|_{L_p(\BR_+,H_r^1(\BR_-^N))} \notag \\
&\leq \|F_1(|\nabla'\CE_N\eta^1|^2)-F_1(|\nabla'\CE_N\eta^2|^2)\|_{L_\infty(\BR_+,H_{r_2}^1(\BR_-^N))} \notag \\
&\times \sum_{j,k=1}^{N-1}\|\langle t\rangle\SSL_{j,k,k}^{\beta_j}(\eta^1)\|_{L_p(\BR_+,H_{r_2}^1(\BR_-^N))}.
\end{align}
Lemma \ref{lem:K-L-est} with $\eta_2=0$ shows that
\begin{equation}\label{L-est-s-eta1}
\|\langle t\rangle\SSL_{j,k,k}^{\beta_j}(\eta^1)\|_{L_p(\BR_+,H_s^1(\BR_-^N))}
\leq C\|\eta^1\|_{K_{p,q_1,q_2;1}^N}^3
\end{equation}
for $s\in\{q_1/2,q_2\}$.
Since $q_1/2<q_1<q_2$, Lemma \ref{lem:int-p-v2} \eqref{lem:int-p-1-2} gives us
\begin{align*}
&\|\langle t\rangle\SSL_{j,k,k}^{\beta_j}(\eta^1)\|_{L_p(\BR_+,H_{q_1}^1(\BR_-^N))} \\
&\leq C\|\langle t\rangle\SSL_{j,k,k}^{\beta_j}(\eta^1)\|_{L_p(\BR_+,H_{q_1/2}^1(\BR_-^N))}^{1-\theta}
\|\langle t\rangle\SSL_{j,k,k}^{\beta_j}(\eta^1)\|_{L_p(\BR_+,H_{q_2}^1(\BR_-^N))}^\theta
\end{align*}
for some $\theta\in(0,1)$.
This inequality together with $a^{1-\theta}b^\theta\leq C(a+b)$ for $a,b\geq 0$ shows that
\begin{equation*}
\|\langle t\rangle \SSL_{j,k,l}^{\beta_j}(\eta^1)\|_{L_p(\BR_+,H_{q_1}^1(\BR_-^N))}
\leq C\sum_{r\in\{q_1/2,q_2\}}
\|\langle t\rangle \SSL_{j,k,l}^{\beta_j}(\eta^1)\|_{L_p(\BR_+,H_{r}^1(\BR_-^N))},
\end{equation*}
which, combined with \eqref{L-est-s-eta1}, furnishes
\begin{equation*}
\|\langle t\rangle \SSL_{j,k,l}^{\beta_j}(\eta^1)\|_{L_p(\BR_+,H_{q_1}^1(\BR_-^N))}
\leq C\|\eta^1\|_{K_{p,q_1,q_2;1}^N}^3.
\end{equation*}
This inequality together with \eqref{est-I1-s-tension}, Lemma \ref{lem:F-est}, and \eqref{L-est-s-eta1} with $s=q_2$ yields
\begin{align*}
&\|\langle t\rangle I_1\|_{L_p(\BR_+,H_r^1(\BR_-^N))} \notag \\
&\leq  C\|\eta^1-\eta^2\|_{K_{p,q_1,q_2;1}^N}\sum_{j=1}^3\Big(\|\eta^1\|_{K_{p,q_1,q_2;1}^N}+\|\eta^2\|_{K_{p,q_1,q_2;1}^N}\Big)^j
\|\eta^1\|_{K_{p,q_1,q_2;1}^N}^3.
\end{align*}
On the other hand, Lemma \ref{lem:H-half} yields
\begin{align*}
&\|\langle t\rangle I_1\|_{H_p^{1/2}(\BR_+,L_r(\BR_-^N))} \\
&\leq C\|F_1(|\nabla'\CE_N\eta^1|)-F_1(|\nabla'\CE_N\eta^2|)\|_{H_p^1(\BR_+,L_\infty(\BR_-^N))} \\
&\times \sum_{j,k=1}^{N-1}\|\langle t\rangle\SSL_{j,k,k}^{\beta_j}(\eta^1)\|_{H_p^{1/2}(\BR_+,L_r(\BR_-^N))},
\end{align*}
which, combined with Lemmas \ref{lem:K-L-est} and \ref{lem:F-est}, furnishes
\begin{align*}
&\|\langle t\rangle I_1\|_{H_p^{1/2}(\BR_+,L_r(\BR_-^N))} \\
&\leq  C\|\eta^1-\eta^2\|_{K_{p,q_1,q_2;1}^N}\sum_{j=1}^3\Big(\|\eta^1\|_{K_{p,q_1,q_2;1}^N}+\|\eta^2\|_{K_{p,q_1,q_2;1}^N}\Big)^j
\|\eta^1\|_{K_{p,q_1,q_2;1}^N}^3.
\end{align*}
Analogously, for $Z=L_p(\BR_+,H_r^1(\BR_-^N))$ or $Z=H_p^{1/2}(\BR_+,L_r(\BR_-^N))$,
\begin{alignat*}{2}
\|\langle t\rangle I_l\|_{Z} 
&\leq C
\|\eta^2\|_{K_{p,q_1,q_2;1}^N}\sum_{j=1}^3\|\eta^2\|_{K_{p,q_1,q_2;1}^N}^j \\
&\times \|\eta^1-\eta^2\|_{K_{p,q_1,q_2;1}^N}
\Big(\|\eta^1\|_{K_{p,q_1,q_2;1}^N}+\|\eta^2\|_{K_{p,q_1,q_2;1}^N}\Big)^2 && \quad \text{for $l=2,5$,} \\
\|\langle t\rangle I_l\|_{Z} 
&\leq C
\|\eta^1-\eta^2\|_{K_{p,q_1,q_2;1}^N}\Big(\|\eta^1\|_{K_{p,q_1,q_2;1}^N}+\|\eta^2\|_{K_{p,q_1,q_2;1}^N}\Big)^2 
&& \quad \text{for $l=3,6$,} 
\end{alignat*}
and
\begin{equation*}
\|\langle t\rangle I_4\|_{Z}
\leq C 
\|\eta^1-\eta^2\|_{K_{p,q_1,q_2;1}^N}\sum_{j=1}^3\Big(\|\eta^1\|_{K_{p,q_1,q_2;1}^N}+\|\eta^2\|_{K_{p,q_1,q_2;1}^N}\Big)^j
\|\eta^1\|_{K_{p,q_1,q_2;1}^N}^3.
\end{equation*}
Hence, the desired inequality holds.

(2), (3)
From \eqref{form-E-normal} and Lemmas \ref{lem:tilde-E} and \ref{lem:F-est},
we can prove the desired inequalities in the same manner as in (1).

(4) The desired inequality follows from (2), (3), and Lemma \ref{lem:tilde-E} immediately.
This completes the proof of Lemma \ref{lem:H-kapppa}.
\end{proof}

By Lemma \ref{lem:H-kapppa} \eqref{lem:H-kapppa-1}, \eqref{lem:H-kapppa-4},
we immediately obtain

\begin{prp}\label{prp:nonl-H}
Suppose that Assumption $\ref{asm:p-q}$ holds. 
Then for any $(\eta^i,\Bu^i)\in K_{p,q_1,q_2}^N$, $i=1,2$, satisfying \eqref{est-1/2}
\begin{align*}
&\sum_{r\in\{q_1/2,q_2\}}
\|\langle t\rangle (\SSH(\eta^1,\Bu^1)-\SSH(\eta^2,\Bu^2))\|_{L_p(\BR_+,H_r^1(\BR_-^N))^N)} \\
&\leq M_8\|(\eta^1-\eta^2,\Bu^1-\Bu^2)\|_{K_{p,q_1,q_2}^N}
\sum_{j=1}^6\Big(\|(\eta^1,\Bu^1)\|_{K_{p,q_1,q_2}^N}+\|(\eta^2,\Bu^2)\|_{K_{p,q_1,q_2}^N}\Big)^j, \\
&\sum_{r\in\{q_1/2,q_2\}}
\|\langle t\rangle (\SSH(\eta^1,\Bu^1)-\SSH(\eta^2,\Bu^2))\|_{H_p^{1/2}(\BR_+,L_r(\BR_-^N)^N)} \\
&\leq M_9\|(\eta^1-\eta^2,\Bu^1-\Bu^2)\|_{K_{p,q_1,q_2}^N}
\sum_{j=1}^6\Big(\|(\eta^1,\Bu^1)\|_{K_{p,q_1,q_2}^N}+\|(\eta^2,\Bu^2)\|_{K_{p,q_1,q_2}^N}\Big)^j, 
\end{align*}
with positive constants $M_8$ and $M_9$.
\end{prp}

\section{Proof of Theorem \ref{thm:transformed}}\label{sec:nonlinear1}
This section proves Theorem \ref{thm:transformed}.
Let $p$, $q_1$, and $q_2$ satisfy Assumption \ref{asm:p-q} in what follows.
Recall $I_{q,p}$ and $K_{p,q_1,q_2}^N$ given by \eqref{dfn:ini-sp} and Subsection \ref{subsec:fspaces}, respectively.
Let $r_0>0$ and define for $N\geq 3$
\begin{align}\label{dfn:ul-space}
K_{p,q_1,q_2}^N(r_0;\eta_0,\Bu_0)
=\{&(\eta,\Bu)\in K_{p,q_1,q_2}^N : \|(\eta,\Bu)\|_{K_{p,q_1,q_2}^N}\leq r_0, \notag \\
&\eta|_{t=0} =\eta_0 \text{ on $\BR^{N-1}$, } \Bu|_{t=0}=\Bu_0 \text{ in $\BR_-^N$,}  \notag \\
&|(\pd_N\CE_N\eta)(x,t)|\leq \frac{1}{2}
\text{ for any $(x,t)\in \BR_-^N \times (0,\infty)$}\} ,
\end{align}
where $(\eta_0,\Bu_0)\in \bigcap_{r\in\{q_1/2,q_2\}} I_{r,p}$ with \eqref{comp-1} and \eqref{comp-2}.

Let us treat $N=3$ and $N\geq 4$ at the same time.
Let $L_j$ $(j=1,2,3,4)$ and $M_k$ $(k=1,\dots,9)$ be the positive constants given by
Subsection \ref{subsec:6-2} and Section \ref{sec:nonl-terms}, respectively.
In addition, $M_0$ and $C_0$ are positive constants given by Lemma \ref{lem:height-func}
and Remark \ref{rmk:equi-norm-4d}, respectively.
Define $L=L_1+L_2+L_3+L_4$.

Choose $r_0\in(0,1)$ so small that
\begin{align}\label{small-cond-pr-1}
&2M_0r_0\max(1,C_0) \leq \frac{1}{2}, \quad M_k r_0^2\leq\frac{r_0}{10L} \quad (k=1,2,4,5,6,7), \notag \\
&M_3r_0 \sum_{j=1}^8r_0^j \leq \frac{r_0}{10L}, \quad 
M_k r_0\sum_{j=1}^6 r_0^j \leq \frac{r_0}{10L} \quad (k=8,9)
\end{align}
and
\begin{align}\label{small-cond-pr-2}
&2M_kr_0\leq \frac{1}{18L} \quad (k=1,2,4,5,6,7), \notag \\
&M_3\sum_{j=1}^8(2r_0)^j\leq \frac{1}{18L}, \quad
M_k \sum_{j=1}^6(2r_0)^j \leq \frac{1}{18L} \quad (k=8,9).
\end{align}
For this fixed $r_0$, we choose $\varepsilon_0>0$ so small that $\varepsilon_0\leq r_0/(10L)$.

Let $\sum_{r\in\{q_1/2,q_2\}}\|(\eta_0,\Bu_0)\|_{I_{r,p}}\leq \varepsilon_0$
and  $(\theta,\Bv)\in K_{p,q_1,q_2}^N(r_0;\eta_0,\Bu_0)$. 
We consider the following linear system:
\begin{equation}\label{fix-eq-1}
\left\{\begin{aligned}
\pd_t\eta-u_N &=\SSD(\theta,\Bv) && \text{on $\BQ_0$, }\\
\pd_t\Bu-\mu\Delta \Bu +\nabla \Fp&=\SSF(\theta,\Bv) && \text{in $\BQ_-$,} \\
\dv\Bv=\SSG(\theta,\Bv)&=\dv\wtd\SSG(\theta,\Bv) && \text{in $\BQ_-$,} \\
(\mu\BD(\Bu)-\Fp\BI)\Be_N+(c_g-c_\sigma\Delta')\eta\Be_N
&=\SSH(\theta,\Bv) && \text{on $\BQ_0$,} \\
\eta|_{t=0}=\eta_0 \quad \text{on $\BR^{N-1}$}, 
\quad \Bu|_{t=0}&=\Bu_0 && \text{in $\BR_-^N$.}
\end{aligned}\right.
\end{equation}
By Propositions \ref{prp:nonl-D}, \ref{prp:nonl-F},
\ref{prp:nonl-tildeG}, \ref{prp:nonl-G-1}, and \ref{prp:nonl-H} with $(\eta^1,\Bu^1)=(\theta,\Bv)$ and $(\eta^2,\Bu^2)=(0,0)$,
we observe from \eqref{small-cond-pr-1} that
\begin{align*}
\sum_{r\in\{q_1/2,q_2\}}
\|\langle t \rangle\SSD(\theta,\Bv)\|_{L_p(\BR_+,W_r^{2-1/r}(\BR_0^N))} 
&\leq M_1 r_0^2 \leq \frac{r_0}{10L}, \\
\sum_{r\in\{q_1/2,q_2\}}
\|\langle t \rangle \SSD(\theta,\Bv)\|_{L_\infty(\BR_+,W_r^{1-1/r}(\BR_0^N))}
&\leq  M_2 r_0^2 \leq \frac{r_0}{10L}, \\
\sum_{r\in\{q_1/2,q_2\}}
\|\langle t\rangle\SSF(\theta,\Bv)\|_{L_p(\BR_+,L_r(\BR_-^N)^N)}
&\leq M_3r_0\sum_{j=1}^8 r_0^j \leq \frac{r_0}{10L}, \\
\sum_{r\in\{q_1/2,q_2\}}
\|\langle t\rangle\pd_t\wtd\SSG(\theta,\Bv)\|_{L_p(\BR_+,L_r(\BR_-^N)^N)}
&\leq M_4r_0^2 \leq  \frac{r_0}{10L}, \\
\sum_{r\in\{q_1/2,q_2\}}\|\langle t\rangle\wtd\SSG(\theta,\Bv)\|_{L_p(\BR_+,L_r(\BR_-^N)^N)}
&\leq M_5r_0^2 \leq   \frac{r_0}{10L}, \\
\sum_{r\in\{q_1/2,q_2\}}\|\langle t\rangle \SSG(\theta,\Bv)\|_{L_p(\BR_+,H_r^1(\BR_-^N))}
&\leq M_6 r_0^2 \leq \frac{r_0}{10L}, \\
\sum_{r\in\{q_1/2,q_2\}}\|\langle t\rangle \SSG(\theta,\Bv)\|_{H_p^{1/2}(\BR_+,H_r^1(\BR_-^N))}
&\leq M_7 r_0^2 \leq \frac{r_0}{10L}, \\
\sum_{r\in\{q_1/2,q_2\}}
\|\langle t\rangle \SSH(\theta,\Bv)\|_{L_p(\BR_+,H_r^1(\BR_-^N))^N)}
&\leq M_8 r_0\sum_{j=1}^6 r_0^j \leq \frac{r_0}{10L}, \\
\sum_{r\in\{q_1/2,q_2\}}
\|\langle t\rangle \SSH(\theta,\Bv)\|_{H_p^{1/2}(\BR_+,L_r(\BR_-^N))^N)}
&\leq M_9 r_0\sum_{j=1}^6 r_0^j \leq  \frac{r_0}{10L}.
\end{align*}
Thus
\begin{align*}
&\sum_{r\in\{q_1/2,q_2\}}\Big(
\|(\SSD(\theta,\Bv),\SSF(\theta,\Bv),\wtd\SSG(\theta,\Bv),\SSG(\theta,\Bv),\SSH(\theta,\Bv))\|_{F_{p,r}(\langle t \rangle)} \\
&+\|\langle t \rangle \SSD(\theta,\Bv)\|_{L_\infty(\BR_+,W_r^{1-1/r}(\BR_0^{N}))}\Big) 
\leq \frac{9r_0}{10L},
\end{align*}
where $F_{p,r}(\langle t \rangle)$ is defined in Subsection \ref{subsec:MR}.
In addition, the compatibility conditions \eqref{comp:t-wegiht1} and \eqref{comp:t-wegiht2} are satisfied
by \eqref{comp-1} and \eqref{comp-2}.

Let $\alpha(N)=1$ when $N=3$ and $\alpha(N)=0$ when $N\geq 4$.
Propositions \ref{prp:linear4d}, \ref{prp:linear3d} and \ref{prp:linear3d-v2} can now be applied to \eqref{fix-eq-1}, i.e.,
there exists a unique solution $(\eta,\Bu,\Fp)$ of \eqref{fix-eq-1}, which satisfies
\begin{align*}
&\|(\eta,\Bu)\|_{K_{p,q_1,q_2}^N} \\
&\leq L\sum_{r\in\{q_1/2,q_2\}}\Big(
\|(\eta_0,\Bu_0)\|_{I_{r,p}}+\alpha(N)\|\langle t \rangle^{1/3}\SSD(\theta,\Bv)\|_{L_\infty(\BR_+,W_r^{1-1/r}(\BR_0^{N}))} \\
&+\|(\SSD(\theta,\Bv),\SSF(\theta,\Bv),\wtd\SSG(\theta,\Bv),\SSG(\theta,\Bv),\SSH(\theta,\Bv))\|_{F_{p,r}(\langle t \rangle)}\Big)  \\
&\leq L\left(\frac{r_0}{10L}+\frac{9r_0}{10L}\right)=r_0.
\end{align*}
In addition, Lemma \ref{lem:height-func} gives us
\begin{align*}
&\sup_{(x,t)\in \BR_-^N \times (0,\infty)}|(\pd_N\CE_N\eta)(x,t)| \\
&\leq M_0
\Big(\|\pd_t\pd_N\CE_N\eta\|_{L_p(\BR_+,H_{q_2}^1(\BR_-^N))}
+\|\pd_N\CE_N\eta\|_{L_p(\BR_+,H_{q_2}^2(\BR_-^N))} \Big) \\
&\leq 
\left\{\begin{aligned}
&2M_0 r_0 \leq \frac{1}{2} && \text{when $N=3$,} \\
&2M_0C_0r_0 \leq \frac{1}{2} && \text{when $N\geq 4$.} 
\end{aligned}\right. 
\end{align*}
We can therefore define the mapping 
\begin{equation*}
\Phi: K_{p,q_1,q_2}^N(r;\eta_0,\Bu_0)\ni(\theta,\Bv)\mapsto (\eta,\Bu)\in K_{p,q_1,q_2}^N(r;\eta_0,\Bu_0).
\end{equation*}

Let us prove that $\Phi$ is a contraction mapping on $K_{p,q_1,q_2}^N(r_0;\eta_0,\Bu_0)$.
Let $(\theta^i,\Bv^i)\in K_{p,q_1,q_2}^N(r_0;\eta_0,\Bu_0)$ for $i=1,2$
and $(\eta^i,\Bu^i)=\Phi(\theta^i,\Bv^i)$.
Setting $\kappa=\eta^1-\eta^2$ and $\Bw =\Bu^1-\Bu^2$,
we see for some pressure $\Fq$ that
\begin{equation}\label{fix-eq-2}
\left\{\begin{aligned}
\pd_t\kappa-w_N &=\SSD(\theta^1,\Bv^1)-\SSD(\theta^2,\Bv^2) && \text{on $\BQ_0$, }\\
\pd_t\Bw-\mu\Delta \Bw +\nabla \Fq&=\SSF(\theta_1,\Bv_1)-\SSF(\theta_2,\Bv_2) && \text{in $\BQ_-$,} \\
\dv\Bw=\SSG(\theta^1,\Bv^1)-\SSG(\theta^2,\Bv^2) 
&=\dv(\wtd\SSG(\theta^1,\Bv^1)-\wtd\SSG(\theta^2,\Bv^2)) && \text{in $\BQ_-$,} \\
(\mu\BD(\Bw)-\Fq\BI)\Be_N+(c_g-c_\sigma\Delta')\kappa\Be_N
&=\SSH(\theta^1,\Bv^1)-\SSH(\theta^2,\Bv^2) && \text{on $\BQ_0$,} \\
(\kappa,\Bw)|_{t=0}&=(0,0). 
\end{aligned}\right.
\end{equation}
By Propositions \ref{prp:nonl-D}, \ref{prp:nonl-F},
\ref{prp:nonl-tildeG}, \ref{prp:nonl-G-1}, and \ref{prp:nonl-H}, we see from \eqref{small-cond-pr-2} that
\begin{align*}
&\sum_{r\in\{q_1/2,q_2\}}
\Big(\|\langle t\rangle
(\SSD(\theta^1,\Bv^1)-\SSD(\theta^2,\Bv^2))\|_{L_p(\BR_+,W_r^{2-1/r}(\BR_0^N))} \\
&+\|\langle t \rangle(\SSD(\theta^1,\Bv^1)-\SSD(\theta^2,\Bv^2))\|_{L_\infty(\BR_+,W_r^{1-1/r}(\BR_0^{N}))} \\ 
&+\|\langle t \rangle(\SSF(\theta^1,\Bv^1)-\SSF(\theta^2,\Bv^2))\|_{L_p(\BR_+,L_r(\BR_-^N)^N)} \\ 
&+\|\langle t \rangle( \pd_t\wtd\SSG(\theta^1,\Bv^1)-\pd_t\wtd\SSG(\theta^2,\Bv^2))\|_{L_p(\BR_+,L_r(\BR_-^N)^N)} \\
&+\|\langle t \rangle( \wtd\SSG(\theta^1,\Bv^1)-\wtd\SSG(\theta^2,\Bv^2))\|_{L_p(\BR_+,L_r(\BR_-^N)^N)} \\
&+\|\langle t \rangle(\SSG(\theta^1,\Bv^1)-\SSG(\theta^2,\Bv^2))\|_{L_p(\BR_+,H_r^1(\BR_-^N))} \\ 
&+\|\langle t \rangle(\SSG(\theta^1,\Bv^1)-\SSG(\theta^2,\Bv^2))\|_{H_p^{1/2}(\BR_+,L_r(\BR_-^N))} \\
&+\|\langle t \rangle(\SSH(\theta^1,\Bv^1)-\SSH(\theta^2,\Bv^2))\|_{L_p(\BR_+,H_r^1(\BR_-^N)^N)} \\ 
&+\|\langle t \rangle(\SSH(\theta^1,\Bv^1)-\SSH(\theta^2,\Bv^2))\|_{H_p^{1/2}(\BR_+,L_r(\BR_-^N)^N)} 
\Big) \\
&\leq \|(\theta^1-\theta^2,\Bv^1-\Bv^2)\|_{K_{p,q_1,q_2}^N} \bigg[2M_1r_0+2M_2r_0+M_3\sum_{j=1}^8(2r_0)^j \\
&+2M_4r_0 +2M_5r_0+2M_6r_0+2M_7 r_0+M_8\sum_{j=1}^6(2r_0)^j+M_9\sum_{j=1}^6(2r_0)^j\bigg] \\
&\leq \frac{1}{2L}\|(\theta^1-\theta^2,\Bv^1-\Bv^2)\|_{K_{p,q_1,q_2}^N} .
\end{align*}
Then we can apply Propositions \ref{prp:linear4d}, \ref{prp:linear3d}, and \ref{prp:linear3d-v2} to \eqref{fix-eq-2}
in order to obtain
\begin{align*}
\|(\kappa,\Bw)\|_{K_{p,q_1,q_2}^N}
\leq \frac{1}{2}\|(\theta^1-\theta^2,\Bv^1-\Bv^2)\|_{K_{p,q_1,q_2}^N}.
\end{align*}
This implies that $\Phi$ is a contraction mapping on $K_{p,q_1,q_2}^N(r_0;\eta_0,\Bu_0)$.
We therefore obtain a unique fixed point $(\rho_*,\Bu_*)$ of $\Phi$ in $K_{p,q_1,q_2}^N(r_0;\eta_0,\Bu_0)$,
i.e., $\Phi(\rho_*,\Bu_*)=(\rho_*,\Bu_*)\in K_{p,q_1,q_2}^N(r_0;\eta_0,\Bu_0)$.
This $(\rho_*,\Bu_*)$ is a unique solution of \eqref{nonl:fix1}--\eqref{initial-cond-flat} 
in $K_{p,q_1,q_2}^N(r;\eta_0,\Bu_0)$,
which completes the proof of Theorem \ref{thm:transformed}.

\section{Proof of Theorems \ref{thm:original} and \ref{thm:decay}}\label{sec:nonlinear2}

\subsection{Proof of Theorem \ref{thm:original}}

We prove the case $N=3$ only.
The case $N\geq 4$ can be proved  in the same manner as in the case $N=3$.
Notice that \eqref{add-pq} is assumed in addition to Assumption \ref{asm:p-q}.

Let us first prove (1). 
Recall $\Theta_0(x)=(x',x_N+(\CE_N\eta_0)(x))$
and $C_1$ is a positive constant given by Lemma \ref{lem:classic-embed-eta}.
We choose $\varepsilon_1>0$ so small that $C_1\varepsilon_1\leq 1/2$.
Lemma \ref{lem:classic-embed-eta} then yields
\begin{equation}\label{eta-smooth-prp1}
\CE_N\eta_0\in C^2(\overline{\BR_-^N})
\end{equation}
and
\begin{equation}\label{eta-smooth-prp2}
\|\nabla \CE_N\eta_0\|_{H_\infty^1(\BR_-^N)}\leq 
C_1\|\eta_0\|_{B_{q_2,p}^{3-1/p-1/q_2}(\BR^{N-1})}
\leq C_1\varepsilon_1\leq \frac{1}{2}.
\end{equation}

Let $(y',y_N)=(x',x_N+(\CE_N\eta_0)(x))$.
Since
\begin{equation*}
\frac{\pd y_N}{\pd x_N}=1+\frac{\pd}{\pd x_N}\CE_N\eta_0
\geq 1-\|\nabla\CE_N\eta_0\|_{L_\infty(\BR_-^N)}\geq \frac{1}{2},
\end{equation*}
$\Theta_0$ is a bijective mapping from $\BR_-^N$ onto $\Omega_0$.
Thus $\Theta_0$ is a $C^2$ diffeomorphism from $\BR_-^N$ onto $\Omega_0$ from \eqref{eta-smooth-prp1}.

We next prove (2).   
Define $\Bu_0=\Bv_0\circ\Theta_0$.
From \eqref{eta-smooth-prp1} and \eqref{eta-smooth-prp2},
we observe by the change of variables that for $r\in\{q_1/2,q_2\}$
\begin{equation*}
\|\Bu_0\|_{L_r(\BR_-^N)} \leq C\|\Bv_0\|_{L_r(\Omega_0)}, \quad
\|\Bu_0\|_{H_r^2(\BR_-^N)} \leq C\|\Bv_0\|_{H_r^2(\Omega_0)},
\end{equation*}
where $C$ is a positive constant independent of $\eta_0$.
This together with the real interpolation method gives us
\begin{equation}\label{int-initial-data}
\|\Bu_0\|_{B_{r,p}^{2-2/p}(\BR_-^N)} \leq C\|\Bv_0\|_{B_{r,p}^{2-2/p}(\Omega_0)}.
\end{equation}
Thus
\begin{align*}
&\sum_{r\in\{q_1,q_2\}}\|(\eta_0,\Bu_0)\|_{I_{r,p}} \\
&\leq \wtd C_1\sum_{r\in\{q_1/2,q_2\}}\Big(\|\eta_0\|_{B_{r,p}^{3-1/p-1/r}(\BR^{N-1})}+\|\Bv_0\|_{B_{r,p}^{2-2/p}(\Omega_0)}\Big),
\end{align*}
with some positive constant $\wtd C_1$ independent of $\eta_0$, $\Bu_0$, and $\Bv_0$.
Choose $\varepsilon_1$ smaller so that $\wtd C_1\varepsilon_1\leq \varepsilon_0$ if necessary,
where $\varepsilon_0$ is the positive constant given by Theorem \ref{thm:transformed}.
Then $(\eta_0,\Bu_0)$ satisfies the smallness condition \eqref{small-1},
and also the compatibility conditions \eqref{comp-1} and \eqref{comp-2} 
follow from \eqref{comp-ori-1} and \eqref{comp-ori-2}, respectively.
Theorem \ref{thm:transformed} thus gives us a global-in-time solution $(\eta,\Bu,\Fp)$
of \eqref{nonl:fix1} in $K_{p,q_1,q_2}^N(r_0;\eta_0,\Bu_0)$.

We finally prove (3). Let $G_L=\BR^{N-1}\times(-L,0)$ for $L>0$.
Since
\begin{equation*}
\eta\in H_p^1(\BR_+,W_{q_2}^{2-1/q_2}(\BR^{N-1}))\cap L_p(\BR_+,W_{q_2}^{3-1/q_2}(\BR^{N-1})),
\end{equation*}
we see from Lemma \ref{lem:classic-embed-eta} that
\begin{equation}\label{eta-smooth-add1}
\CE_N\eta\in C([0,\infty),C_B^2(\overline{G_L}))
\quad \text{for any $L>0$}.
\end{equation}
In addition, \eqref{f-trace-est} and Lemma \ref{lem:extra-embed} \eqref{lem:extra-embed-1} yield
\begin{equation*}
u_N , \SSD(\eta,\Bu)\in C([0,\infty),B_{q_2,p}^{2-2/p-1/q_2}(\BR_0^N)),
\end{equation*}
while the first equation of \eqref{nonl:fix1} gives us
\begin{equation*}
\pd_t\CE_N\eta =\CE_N u_N+\CE_N\SSD(\eta,\Bu) \quad \text{in $\BQ_-$.}
\end{equation*}
Lemma \ref{lem:extra-embed} \eqref{lem:extra-embed-2} thus furnishes
\begin{equation*}
\pd_t\CE_N\eta \in C([0,\infty),C_B^1(\overline{G_L}))
\quad \text{for any $L>0$}. 
\end{equation*}
Combining this with \eqref{eta-smooth-add1} yields
\begin{equation}\label{diffeo-final}
\CE_N\eta
\in C^1([0,\infty), C_B^1(\overline{G_L})) 
\quad \text{for any $L>0$.}
\end{equation}
Since \eqref{cond-diffeo} holds by $(\eta,\Bu)\in K_{p,q_1,q_2}^N(r_0;\eta_0,\Bu_0)$, 
\eqref{diffeo-final} shows that 
$\Theta$ is a $C^1$ diffeomorphism from $\BQ_-$ onto $\Omega_\infty$.
Furthermore, \eqref{eta-smooth-add1} guarantees that 
$\Theta(\cdot,t)$ is a $C^2$ diffeomorphism from $\BR_-^N$ onto $\Omega(t)$ for each $t>0$.
This completes the proof of Theorem \ref{thm:original}.

\subsection{Proof of Theorem \ref{thm:decay}}
Let $(\eta,\Bu,\Fp)$ be the solution of \eqref{nonl:fix1} given by the previous subsection
and let $q\in\{q_1,q_2\}$.
Let $\beta(N)=1/3$ for $N=3$ and $\beta(N)=1/2$ for $N\geq 4$.
By Lemma \ref{lem:time-sp} \eqref{lem:time-sp-1}
\begin{align*}
&\sup_{\tau\in[0,\infty)}\|\langle \tau\rangle^{\beta(N)} \eta(\tau)\|_{B_{q,p}^{3-1/p-1/q}(\BR^{N-1})} \\
&\leq C\Big(\|\langle \tau\rangle^{\beta(N)} \eta\|_{H_p^1(\BR_+,W_q^{2-1/q}(\BR^{N-1}))}
+\|\langle \tau\rangle^{\beta(N)} \eta\|_{L_p(\BR_+,W_q^{3-1/q}(\BR^{N-1}))}\Big) \\
&\leq C r_0,
\end{align*}
and thus
\begin{equation*}
\|\eta(\tau)\|_{B_{q,p}^{3-1/p-1/q}(\BR^{N-1})}=O(\tau^{-\beta(N)}) \quad (\tau\to\infty).
\end{equation*}

We next consider $\Bv=\Bu(\Theta^{-1}(y,\tau))$ for $(y,\tau)\in\Omega_\infty$. 
Similarly to the proof of \eqref{int-initial-data}, we obtain
\begin{equation*}
\|\Bv(\tau)\|_{B_{q,p}^{2-2/p}(\Omega(\tau))}\leq C\|\Bu(\tau)\|_{B_{q,p}^{2-2/p}(\BR_-^N)}.
\end{equation*}
Thus
\begin{equation*}
\sup_{\tau\in [0,\infty)}\Big(\langle \tau\rangle^{\beta(N)}\|\Bv(\tau)\|_{B_{q,p}^{2-2/p}(\Omega(\tau))}\Big) 
\leq C\sup_{\tau\in [0,\infty)}\Big(\langle \tau\rangle^{\beta(N)}\|\Bu(\tau)\|_{B_{q,p}^{2-2/p}(\BR_-^N)}\Big),
\end{equation*}
which, combined with Lemma \ref{lem:time-sp} \eqref{lem:time-sp-2}, furnishes
\begin{align*}
&\sup_{\tau\in [0,\infty)}\Big(\langle \tau\rangle^{\beta(N)}\|\Bv(\tau)\|_{B_{q,p}^{2-2/p}(\Omega(\tau))}\Big)  \\
&\leq C\Big(\|\langle \tau\rangle^{\beta(N)}\Bu\|_{H_p^1(\BR_+,L_q(\BR_-^{N})^N)}
+\|\langle \tau\rangle^{\beta(N)}\Bu\|_{L_p(\BR_+,H_q^2(\BR_-^{N})^N)}\Big) \\
&\leq C r_0.
\end{align*}
Therefore
\begin{equation*}
\|\Bv(\tau)\|_{B_{q,p}^{2-2/p}(\Omega(\tau))}=O(\tau^{-\beta(N)}) \quad (\tau\to \infty).
\end{equation*}
This completes the proof of Theorem \ref{thm:decay}.

\def\thesection{A}
\renewcommand{\theequation}{A.\arabic{equation}}
\section{}

This appendix derives \eqref{nonl:fix1} from \eqref{nonl:eq1}.

Recall \eqref{om-gam-2}, \eqref{om-gam-1}, \eqref{dfn:Q-Q0}, and \eqref{dfn:theta}.
Let $y=(y',y_N)=(y_1,\dots,y_{N-1},y_N)\in\Omega(\tau)$, $\tau>0$,
and let $x=(x',x_N)=(x_1,\dots,x_{N-1},x_N)\in\BR_-^N$ and $t>0$.
In what follows, $D_j=\pd/\pd y_j$ and $\pd_j=\pd/\pd x_j$ for $j=1,\dots,N$,
while $\pd_\tau=\pd/\pd\tau$ and $\pd_t=\pd/\pd t$.

The Jacobian matrix $\pd (y,\tau)/\pd (x,t)$ of $\Theta$ is given by
\begin{equation*}
\frac{\pd (y,\tau)}{\pd(x,t)} 
=
\begin{pmatrix}
1 & \dots & 0 & 0 & 0 \\
\vdots & \ddots & \vdots & \vdots & \vdots \\
0 & \dots & 1 & 0 & 0 \\
\pd_1\CE_N\eta & \dots & \pd_{N-1}\CE_N\eta & 1+\pd_N\CE_N\eta & \pd_t\CE_N\eta \\
0 & \dots & 0 & 0 & 1
\end{pmatrix}, \quad
\end{equation*}
and its determinant $J=\det(\pd(y,\tau)/\pd (x,t))$ is given by $J = 1+\pd_N\CE_N\eta$.
Since $J\geq 1/2$ by \eqref{cond-diffeo},
there exists the inverse matrix $(\pd (y,\tau)/ \pd (x,t))^{-1}$ of $\pd (y,\tau)/\pd (x,t)$ such that
\begin{equation*}
\bigg(\frac{\pd (y,\tau)}{\pd(x,t)}\bigg)^{-1}
=
\begin{pmatrix}
1 & \dots & 0 & 0 & 0\\
\vdots & \ddots & \vdots & \vdots & \vdots \\
0 & \dots & 1 & 0  &0 \\
-\frac{\pd_1\CE_N\eta}{1+\pd_N\CE_N\eta}
& \dots & -\frac{\pd_{N-1}\CE_N\eta}{1+\pd_N\CE_N\eta} & \frac{1}{1+\pd_N\CE_N\eta} & -\frac{\pd_t\CE_N\eta}{1+\pd_N\CE_N} \\
0 & \dots & 0 & 0 & 1
\end{pmatrix}.
\end{equation*}
Thus 
\begin{align}
D_j &=\pd_j-\bigg(\frac{\pd_j\CE_N\eta}{1+\pd_N\CE_N\eta}\bigg)\pd_N \quad (j=1,\dots,N), \label{deriv-cv} \\
\pd_\tau &=\pd_t -\bigg(\frac{\pd_t\CE_N\eta}{1+\pd_N\CE_N\eta}\bigg)\pd_N. \label{time-deriv}
\end{align}
The relation \eqref{deriv-cv} yields
\begin{equation}\label{s-deriv-change}
D_jD_k=\pd_j \pd_k-\CD_{jk}(\eta) \quad (j,k=1,\dots,N)
\end{equation}
with the second order differential operator $\CD_{jk}(\eta)$ given by \eqref{deriv-second},
and also
\begin{equation}\label{nabla-cv}
\nabla_y = (\BI-J^{-1}\BJ(\eta))\nabla_x
\end{equation}
for $\nabla_y=(D_1,\dots,D_N)^\SST$ and $\nabla_x=(\pd_1,\dots,\pd_N)^\SST$
with $\BJ(\eta)$ given in \eqref{matrix:JK}.

Recall $(\Bu,\Fp)$ in \eqref{dfn:u-p}. Let
\begin{align*}
\BD_y(\Bv) 
&= \nabla_y\Bv+(\nabla_y\Bv)^\SST =(D_i v_j+D_j v_i)_{1\leq i,j\leq N}, \\
\BD_x(\Bu)
&= \nabla_x\Bu+(\nabla_x\Bu)^\SST=(\pd_i u_j+\pd_j u_i)_{1\leq i,j\leq N}.
\end{align*}
It follows from \eqref{deriv-cv} and \eqref{dfn:E} that
\begin{equation}\label{defom-cv}
\BD_x(\Bv)=\BD_\xi(\Bu)-\BE(\eta,\Bu). 
\end{equation}
In addition,
\begin{align}\label{dv-cv}
J\dv_y\Bv
&=(1+\pd_N\CE_N\eta)\dv_x\Bu- \nabla_x\CE_N\eta \cdot \pd_N\Bu \notag \\
&=\dv_x\{J(\BI-J^{-1}\BJ(\eta))^\SST\Bu\},
\end{align}
where $\dv_y\Bv=\sum_{j=1}^ND_j v_j$ and $\dv_x\Bu=\sum_{j=1}^N\pd_j u_j$.
In fact,
the first equality of \eqref{dv-cv} follows from \eqref{deriv-cv} immediately
and the second equality of \eqref{dv-cv} is obtained as follows:
let $\varphi\in C_0^\infty(\Omega_\infty)$
and $\overline\varphi=\varphi(\Theta(x,t))$, and then \eqref{nabla-cv} yields
\begin{align*}
&(\dv_y\Bv,\varphi)_{\Omega_\infty}
=-(\Bv,\nabla_y\varphi)_{\Omega_\infty} 
=-\int_{\BQ_-}\{\Bu\cdot (\BI-J^{-1}\BJ(\eta))\nabla_x\overline{\varphi}\}J \intd xdt \\
&=\int_{\BQ_-}[\dv_x\{J(\BI-J^{-1}\BJ(\eta))^\SST\Bu\}] \overline{\varphi}\intd xdt \\
&=\int_{\Omega_\infty}
[J^{-1}\dv_x\{J(\BI-J^{-1}\BJ(\eta))^\SST\Bu\}](\Theta^{-1}(y,\tau))\varphi(y,\tau)\intd yd\tau .
\end{align*}
This gives us
\begin{equation*}
\dv_y\Bv=J^{-1}\dv_x\{J(\BI-J^{-1}\BJ(\eta))^\SST\Bu\},
\end{equation*}
and thus multiplying the both side of this equation by $J$ furnishes the second equality of \eqref{dv-cv}.

{\bf Step 1}: Derive the first equation of \eqref{nonl:fix1}.
Let $\alpha'=(\alpha_1,\dots,\alpha_{N-1})\in\BN_0^{N-1}$. Then \eqref{deriv-cv} yields
\begin{align}\label{tangent-deriv}
\frac{\pd^{|\alpha'|}}{\pd y_1^{\alpha_1}\dots\pd y_{N-1}^{\alpha_{N-1}}} \eta(y',t)
&=\frac{\pd^{|\alpha'|}}{\pd x_1^{\alpha_1}\dots\pd x_{N-1}^{\alpha_{N-1}}} \eta(x',t) \notag \\
&=\frac{\pd^{|\alpha'|}}{\pd x_1^{\alpha_1}\dots\pd x_{N-1}^{\alpha_{N-1}}}\CE_N\eta (x',0,t).
\end{align}
Notice that $\pd_\tau\eta(y',\tau)=\pd_t\eta(x',t)$ by \eqref{time-deriv}.
By these relations,
we obtain the first equation of \eqref{nonl:fix1} from the first equation of \eqref{nonl:eq1} immediately.

{\bf Step 2}: Derive the second equation of \eqref{nonl:fix1}.
One sees by \eqref{deriv-cv} that
\begin{equation*}
(\Bv\cdot\nabla_y)\Bv 
=(\Bu\cdot\nabla_x)\Bu-\frac{1}{1+\pd_N\CE_N\eta}(\Bu\cdot\nabla_x\CE_N\eta)\pd_N\Bu,
\end{equation*}
while \eqref{time-deriv} yields
\begin{equation*}
\pd_t\Bv=\pd_t\Bu -\bigg(\frac{\pd_t\CE_N\eta}{1+\pd_N\CE_N\eta}\bigg)\pd_N\Bu.
\end{equation*}
Combining these relations with \eqref{s-deriv-change} and \eqref{nabla-cv},
we obtain the following equation from the second equation of \eqref{nonl:eq1}:
\begin{equation}\label{eq:ap1}
\pd_t\Bu-\mu\Delta\Bu+(\BI-J^{-1}\BJ(\eta))\nabla\Fp=\frac{1}{(1+\pd_N\CE_N\eta)^3}\wtd \SSF(\eta,\Bu),
\end{equation}
where $\wtd \SSF(\eta,\Bu)$ is given by \eqref{dfn:F-tilde}.
By an elementary calculation
\begin{equation*}
\BJ(\eta)^2 = (\pd_N\CE_N\eta) \, \BJ(\eta),
\end{equation*}
which furnishes
\begin{equation*}
(\BI+\BJ(\eta))(\BI-J^{-1}\BJ(\eta))=\BI.
\end{equation*}
Thus $(\BI-J^{-1}\BJ(\eta))^{-1}=\BI+\BJ(\eta)$.
One multiplies the both side of \eqref{eq:ap1} by $\BI+\BJ(\eta)$
in order to obtain
\begin{equation*}
(\BI+\BJ(\eta))(\pd_t\Bu-\mu\Delta\Bu)+\nabla\Fp=
\frac{1}{(1+\pd_N\CE_N\eta)^3}(\BI+\BJ(\eta))\wtd\SSF(\eta,\Bu),
\end{equation*}
which implies
\begin{equation*}
\pd_t\Bu-\mu\Delta\Bu+\nabla\Fp
=\frac{1}{(1+\pd_N\CE_N\eta)^3}(\BI+\BJ(\eta))\wtd\SSF(\eta,\Bu)
-\BJ(\eta)(\pd_t\Bu-\mu\Delta\Bu).
\end{equation*}
The second equation of \eqref{nonl:fix1} is therefore obtained.

{\bf Step 3}: Derive the third equation of \eqref{nonl:fix1}. 
The third equation of \eqref{nonl:fix1} follows from 
\eqref{dv-cv} and the third equation of \eqref{nonl:eq1} immediately.

{\bf Step 4}: Derive the fourth equation of \eqref{nonl:fix1}.
Recall that $\Bn_{\Gamma(t)}$, $\kappa_{\Gamma(t)}$, and $\BK(\eta)$ are given by
\eqref{normal-vec}, \eqref{mean-cuv}, and \eqref{matrix:JK}, respectively.
By \eqref{tangent-deriv}, $\Bn_{\Gamma(t)}$ and $\kappa_{\Gamma(t)}$ can be written as
\begin{equation*}
\Bn_{\Gamma(t)}=\frac{1}{\sqrt{1+|\nabla'\CE_N\eta|^2}}
\begin{pmatrix}
-\nabla'\CE_N\eta \\ 
1
\end{pmatrix}, \quad
\kappa_{\Gamma(t)}
=\Delta'\eta-\SSH_\kappa(\eta) 
\quad \text{on $\BQ_0$,}
\end{equation*}
where $\nabla'=(\pd_1,\dots,\pd_{N-1})^\SST$, $\Delta'=\sum_{j=1}^{N-1}\pd_j^2$,
and $\SSH_\kappa(\eta)$ is given by \eqref{dfn:H-2}.
Thus
\begin{equation}\label{eq:normal-relat}
\sqrt{1+|\nabla'\CE_N\eta|^2}(\BI+\BK(\eta))\Bn_{\Gamma(t)}=\Be_N \quad \text{on $\BQ_0$.}
\end{equation}
One multiplies the fourth equation of \eqref{nonl:eq1} by $\sqrt{1+|\nabla'\CE_N\eta|^2}(\BI+\BK(\eta))$,
and then one has by \eqref{defom-cv} and \eqref{eq:normal-relat}
\begin{multline}\label{eq:ap2}
\sqrt{1+|\nabla'\CE_N\eta|^2}(\BI+\BK(\eta))\cdot \mu(\BD(\Bu)-\BE(\eta,\Bu))\Bn_{\Gamma(t)}
\\ -\Fp\Be_N +c_g\eta\Be_N =c_\sigma(\Delta'\eta-\SSH_\kappa(\eta))\Be_N \quad \text{on $\BQ_0$.}
\end{multline}
Notice that
\begin{equation*}
\sqrt{1+|\nabla'\CE_N\eta|^2}\Bn_{\Gamma(t)}=\Bn(\eta)=\Be_N+\wht\Bn(\eta) \quad \text{on $\BQ_0$,}
\end{equation*}
where $\Bn(\eta)$ and $\wht\Bn(\eta)$ are given by \eqref{normal-vec2}.
We therefore see by \eqref{dfn:H-1} that
\begin{align*}
&\sqrt{1+|\nabla'\CE_N\eta|^2}(\BI+\BK(\eta))\cdot \mu(\BD(\Bu)-\BE(\eta,\Bu))\Bn_{\Gamma(t)} \\
&=\mu\BD(\Bu)\Be_N-\wtd\SSH(\eta,\Bu) \quad \text{on $\BQ_0$.}
\end{align*}
Combining this relation with \eqref{eq:ap2} gives us the fourth equation of \eqref{nonl:fix1}.
This completes the derivation of \eqref{nonl:fix1}.


\end{document}